\documentclass[11pt]{article}

\usepackage{fullpage}
\usepackage[dvipsnames]{xcolor}  
\usepackage{tipa}
\usepackage{dsfont}
\usepackage{tikz-cd}
\usepackage{tikz, tikz-3dplot, pgfplots}
\usetikzlibrary{decorations.markings, patterns, decorations.pathmorphing}

\tikzset{->-/.style={decoration={
markings,
mark=at position #1 with {\arrow{>}}},postaction={decorate}}}

\usepackage{graphicx}

\usepackage{amsxtra}
\usepackage{amsmath}
\usepackage{amssymb}
\usepackage{amsfonts}
\usepackage{mathrsfs}
\usepackage{amsthm}
\usepackage[all]{xy}
\usepackage{enumitem}
\usepackage{hyperref}
\usepackage{subfigure}
\usepackage{mathabx}
\usepackage{euscript}
\usepackage{xargs}                      

\usepackage[colorinlistoftodos,prependcaption,textsize=tiny]{todonotes}
\newcommandx{\note}[2][1=]{\todo[linecolor=Plum,backgroundcolor=Plum!25,bordercolor=Plum,#1]{#2}}

\newtheoremstyle{example}{\topsep}{\topsep}%
     {}
     {}
     {\bfseries}
     {.}
     {2pt}
     {\thmname{#1}\thmnumber{ #2}\thmnote{ #3}}

\theoremstyle{example}
\newtheorem{exa}[equation]{Example}
\newtheorem{exas}[equation]{Examples}
\newtheorem{ex}[equation]{Example}
\newtheorem{rem}[equation]{Remark}
\newtheorem{rems}[equation]{Remarks}
\newtheorem{defi}[equation]{Definition}

\newtheorem{thm}[equation]{Theorem}
 \newtheorem{cor}[equation]{Corollary}
\newtheorem{lem}[equation]{Lemma}
\newtheorem{prop}[equation]{Proposition}

\hypersetup{%
    colorlinks,%
    linkcolor={red!50!black},%
    citecolor={blue!50!black},%
    urlcolor={blue!80!black}%
}

\graphicspath{{svg/}}

\setlist[enumerate,1]{label=(\arabic{*})}
\setlist[enumerate,2]{label=(\roman{*})}
\setlist[enumerate,3]{label=(\alph{*})}

\def\A{ {\EuScript A}}

\def\Ab{{\EuScript{A}b}}
\def\Ac{\mathcal {A}}

\def\Aut{\on{Aut}}

\def\Bd{\overrightarrow{B}}
\def\be{\begin{equation}}
\def\bef{\begin{figure}}

\def\btp{\begin{tikzpicture}}

\def\C{{\EuScript C}}

\def\CC{\mathbb{C}}

\def\CDK{C_{\operatorname{DK}}}

 \def\Co{{\on{C}}}

\def\colim{\operatorname*{colim}}
\def\varinjlim{\operatorname*{colim}}

\def\cm{\langle m \rangle}
\def\cn{\langle n \rangle}

\def\D{{\EuScript D}}
\def\Dd{{\overrightarrow{D}}}
\def\Dp{{\EuScript D}^+\!}
\def\Di{{\EuScript D}}

\def\Dc{{\mathcal{D}}}
\def\DC{{\on{DC}}}
\def\DD{{\mathbb{D}}}
\def\del{{\partial}}

\def\dg{{\on{dg}}}

 \def\Disko{{\on{Disk}^o}}
 \def\Diskc{{\on{Disk}^c}}
 \def\Diskoc{{\on{Disk}^{oc}}}
 \def\drarr{\draw  [decoration={markings,mark=at position 0.9 with
{\arrow[scale=1.5,>=stealth]{>}}},postaction={decorate},
line width=.2mm]}

\def\E{{\EuScript E}}
\def\Ec{\mathcal{E}}
\def\ee{\end{equation}}

\def\Emb{\on{Emb}}
\def\Entr{\on{Entr}}

\def\enf{\end{figure}}

\def\eps{{\varepsilon}}
\def\Eq{{\on{Eq}}}
\def\etp{\end{tikzpicture}}

\def\Exit{\on{Exit}}

\def\F{{ \EuScript F}}

\def\Fun{\operatorname{Fun}}

\def\G{{ \EuScript G}}

\def\Ho{\on{Ho}}

\def\Hom{\operatorname{Hom}}

\def\hra{\hookrightarrow}

\def\I{{\mathcal I}}
\def\id{\on{id}}
\def\Id{\on{Id}}

 \def\iso{{ \operatorname{iso}}}

\def\J{{\mathcal J}}
\def\J{ {\EuScript J}}

\def\Ker{{\on{Ker}}}

\def\Lambdad{{\overrightarrow{\Lambda}\!}}
\def\Lambdap{{\Lambda\!^{+}\!}}

\def\L{{\EuScript L}}

\def\lla{\longleftarrow}

\def\lra{\longrightarrow}
\def\lim{\operatorname*{lim}}
\def\varprojlim{\operatorname*{lim}}

\def\Md{\overrightarrow{M}}
\def\Mdp{\overrightarrow{M}^+\!}
\def\Map{\operatorname{Map}}

\def\Mor{\on{Mor}} 
 \def\mcS{{\mathcal S}}

\def\N{\operatorname{N}}
\def\NN{\mathbb{N}}

\def\NDK{\operatorname{N_{DK}}}

\def\Ob{{\on{Ob}}}

\def\on{\operatorname}
\def\one{ {\bf 1}}
\def\ol{\overline}
\def\oo{{\infty}}
\def\op{{\operatorname{op}}}
\def\Op{{\mathfrak{O}}}

\def\ord{{\on{ord}}}

 \def\P{{\EuScript P}}

\def\phi{{\varphi}}

\def\PS{{\on{PS}}}

\def\Qd{\overrightarrow{Q}}

\def\RG{{\on{R}\!\Gamma}}

\def\Ran{\on{Ran}}

\def\RR{\mathbb{R}}

\def\Sd{\overrightarrow{S}}

 \def\Set{ {\operatorname{Set}}}

\def\Sh{\on{Sh}}

 \def\Sing{\on{Sing}}
\def\Sp{{\on{Sp}}}

 \def\sSet{{\Set}_{\Delta}}

\def\surj{{\on{surj}}}
 
 \def\Sym{{\on{Sym}}}

\def\U{ {\EuScript U}}

\def\ul{\underline}

\def\Vd{\overrightarrow{V}}

\def\wt{\widetilde}

\def\ZZ{\mathbb{Z}}

 \def\={{\,\, \simeq\,\, }}
 \def\-{{\setminus}}
\def\<<{\langle {}\hskip -.1cm {}\langle}
\def\>>{\rangle \hskip -.1cm \rangle}

\def\dot{\includegraphics[scale=.18]{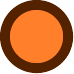}}

\newcommand{\disk}[2]{\raisebox{-.4\height}{\includegraphics[scale=.18]{svg/disk_#1_#2.pdf}}}

\newcommand{\bdisk}[2]{\raisebox{-.4\height}{\includegraphics[scale=.3]{svg/disk_#1_#2.pdf}}}
\newcommand{\bdiskdual}[2]{\raisebox{-.4\height}{\includegraphics[scale=.3]{svg/disk_#1_#2_dual.pdf}}}

\newcommand{\rotdisk}[3]{\rotatebox[origin=c]{#3}{\raisebox{-.4\height}{\includegraphics[scale=.18]{svg/disk_#1_#2.pdf}}}}

\title{Perverse sheaves on Riemann surfaces as Milnor sheaves}
\author{Tobias Dyckerhoff, Mikhail Kapranov,  Yan Soibelman}

\begin{document}

\maketitle

\begin{abstract}
	Constructible sheaves of abelian groups on a stratified space can be equivalently
	described in terms of representations of the exit-path category. In this work, we provide a
	similar presentation of the abelian category of perverse sheaves on a stratified surface in
	terms of representations of the so-called paracyclic category of the surface. The category
	models a hybrid exit-entrance behaviour with respect to chosen sectors of direction, placing
	it ``in between'' exit and entrance path categories. In particular, this perspective yields an
	intrinsic definition of perverse sheaves as an abelian category without reference to derived
	categories and t-structures. 
\end{abstract}

\tableofcontents

\addcontentsline{toc}{section}{Introduction}

\numberwithin{equation}{section}

\section*{Introduction}

\paragraph{Contents and future applications.} 
This paper is the first step in a larger project
devoted to a systematic development of the theory of {\em perverse schobers}. The latter
are categorical analogs of perverse sheaves, in which vector spaces are replaced by (enhanced)
triangulated categories. The idea of perverse schobers was proposed in
\cite{kapranov-schechtman:schobers} based on the features of various ``elementary'' descriptions of
perverse sheaves in terms of quivers. Namely, these descriptions are often of such form that a
natural categorical analog (quiver representations formed by categories instead of vector spaces)
suggests itself readily.  For example, for the classical description \cite{beil-gluing, GGM} of
perverse sheaves on the disk in terms of diagrams
\begin{equation}\label{eq:ab}
		\begin{tikzcd}
			\Phi \ar[bend left=20]{r}{a} & \Psi \ar[bend left=20]{l}{b},
		\end{tikzcd}
	\end{equation}
with $\id - ab$ and $\id - ba$ invertible, such a categorical analog is found in the concept of a
{\em spherical adjunction}, see 
\cite{kapranov-schechtman:schobers}. 
  
However, the quiver descriptions do not give satisfying {\em definitions} of the category of
perverse sheaves, since they depend on auxiliary choices. For example, in the above case,  a
choice of a direction at the origin is needed to define vanishing and nearby cycles.  On the other
hand, from the customary point of view, a perverse sheaf is an object of an abelian category
that arises as the heart of a certain t-structure on the derived category of constructible sheaves
on a stratified topological space. It is not clear whether such an approach can be categorified
directly. 

In this paper, we identify perverse sheaves (not yet schobers) on a stratified surface $X$ with
so-called {\em Milnor sheaves} (Theorem \ref{thm:milnorsheaves}). Similarly to the description of
constructible sheaves as representations of the exit path category (see \cite{treumann}), our result
follows from an alternative parametrization in terms of a hybrid of the exit and entrance path
categories, called the {\em Milnor category of the surface}. Its objects, {\em Milnor disks}, are
given by disks in $X$ together with a choice of a finite number of {\em boundary intervals}. 
These intervals determine the interaction with the stratification: a disk may move on the surface via 
isotopy such that the points in the $0$-dimensional stratum exit the disk through the chosen boundary intervals 
and enter the disk through their complement. In addition, the boundary intervals themselves can
interact in a way familiar from Connes' cyclic category (see below for more details). A Milnor sheaf
is then defined as a representation  of the Milnor category subject to certain natural gluing
conditions that arise from cutting Milnor disks into pieces. 

As a result, we obtain an intrinsic definition of perverse sheaves on Riemann surfaces that is
internal to the framework of abelian categories, without reference to derived categories, and which
can therefore serve as an alternative to the definition given in \cite{BBD}. Our main incentive is
that the definition has a comparatively straightforward categorification offering a good framework
for perverse schobers. This approach will be elaborated in sequels to this paper.

Even in the uncategorified context of perverse sheaves, Milnor sheaves provide a novel perspective on
classical aspects of the theory. For example, one motivation for the introduction of perverse
sheaves is the fact that, in contrast to constructible sheaves, they are preserved under Verdier
duality. This phenomenon becomes almost self-evident in the Milnor sheaf model. Namely, it is a
direct consequence of a canonical self-duality of the Milnor category obtained by swapping the
boundary intervals with their complements (generalizing the well-known self-dualities of the cyclic
and paracyclic categories).

In higher complex dimensions, a possible generalization could involve mimicking more closely the
topology related to forming perverse sheaves of vanishing cycles associated to holomorphic
functions.  When such a perverse sheaf  is supported at a single point (the ``isolated microlocal
singularity" case), it reduces to a single vector space so we have purity just like for Riemann
surfaces.  We hope to explore this approach in future work.

\paragraph{Details of the main result.}
Fundamental for us is the concept of a {\em Milnor disk}, a pair $(A,A')$ where $A \subset X$ is a
closed disk, containing at most one point from the $0$-dimensional stratum $N$, and $A'\subset
\partial A$ is a finite nonempty disjoint union of closed intervals. These Milnor disks will be
depicted by the symbols
\[
	\disk{0}{1}, \quad \disk{0}{2},\quad \disk{0}{3},\quad \dots 
\]
We call the points in the $0$-dimensional stratum $N$ {\em special} and signify them via the symbol
$\raisebox{.3\height}{\dot}$. For example, a Milnor disk $(A,A')$ with one boundary interval
containing a special point will be referred to as
\[
 (A,A') = \disk{1}{1}
\]
leaving the embedding of $A$ into the surface $X$ implicit. Milnor disks form the objects of the
{\em Milnor category} $M(X,N)$ where a morphism from $(A,A')$ to $(B,B')$ is given by an equivalence 
class of isotopies $H: I \times \DD \to X$ with $H_0: \DD \cong A$ and $H_1: \DD \cong B$, together 
with a choice of bordism $P \subset I \times S^1$ from $H_0^{-1}(A')$ to $H_1^{-1}(B')$ such that
the inclusion $H_1^{-1}(B') \subset P$ is a homotopy equivalence (see Figure \ref{eq:bordsurface}).
\begin{figure}[ht] \label{eq:bordsurface}
	\centering
	\def\svgwidth{13cm}
	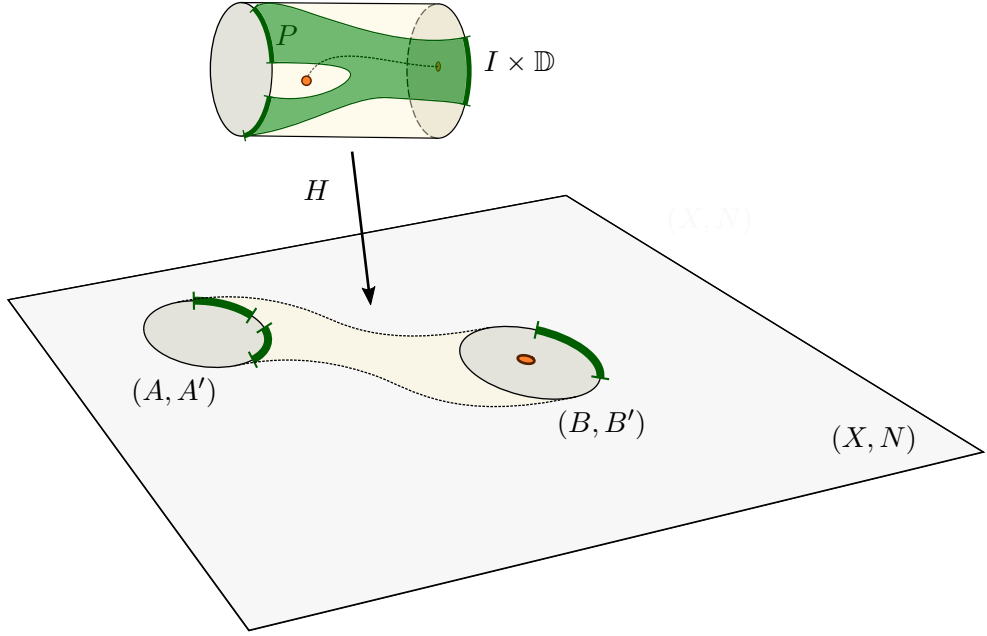
\caption{A morphism in $M(X,N)$ from $(A,A')$ to $(B,B')$ represented given by the isotopy $H$}
\end{figure}
Here, roughly speaking, the trajectories $H^{-1}(N)$ of the special points are required to enter the
cylinder through $(I \times S^1) \setminus P$ and exit through $P$. This hybrid exit-entry behaviour 
puts the Milnor category ``in between'' the exit and entrance path categories of $(X,N)$. As will be
explained in the main body of this work, this phenomenon can be regarded as a geometric
manifestation of the fact that the perverse $t$-structure lies ``in between'' the standard
$t$-structure and its Verdier dual. 

In particular, while the exit and entrance path categories are dual to one another, the Milnor
category is self-dual: On objects, the duality is given by 
\[
	(A,A') \mapsto (A, \overline{\partial(A) \setminus A'})
\]
on morphisms, it is obtained by replacing the bordism $P$ by the closure of $(I \times S^1)
\setminus P$ and reversing the direction of the isotopy $H$. For example, the action of the
self-duality associates to the morphism
\begin{equation}\label{eq:a}
		\begin{tikzcd}[row sep={1em,between origins}]
			& \includegraphics[width=3cm]{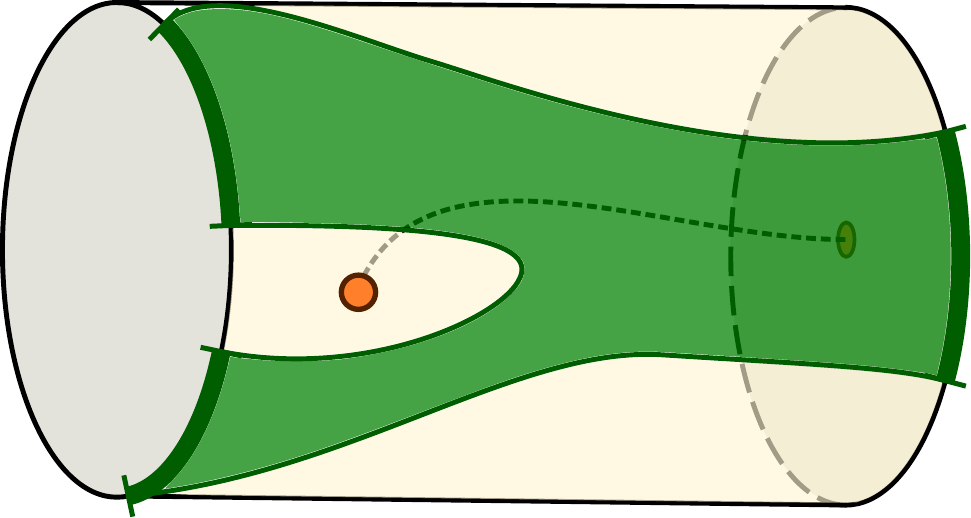} & \\
			\bdisk{0}{2} \ar{rr} & & \bdisk{1}{1}
		\end{tikzcd}
	\end{equation}
depicted in Figure \ref{eq:bordsurface}, the morphism
\begin{equation}\label{eq:b}
		\begin{tikzcd}[row sep={1em,between origins}]
			& \includegraphics[width=3cm]{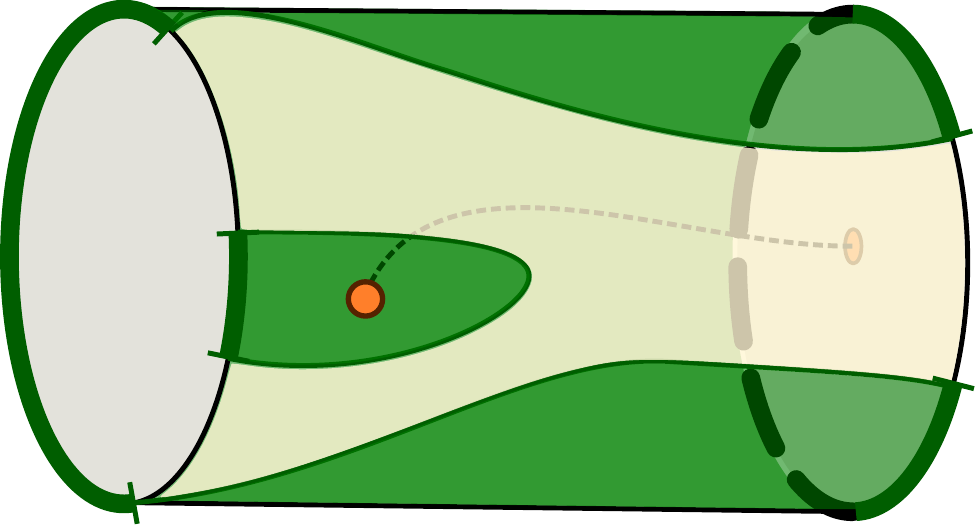} & \\
			\bdiskdual{0}{2} & &\ar{ll}  \bdiskdual{1}{1}
		\end{tikzcd}
	\end{equation}

Given an object $\F$ of the derived constructible category $D(X,N;\A)$ and a morphism $(H,P): (A,A')
\to (B,B')$ of Milnor disks, we obtain a correspondence on
relative (hyper) cohomology
\begin{equation}\label{eq:correspondence}
		\begin{tikzcd}
			\RG(A,A';\F) & \ar{l} \RG(I \times \DD,P;H^*\F) \ar{r}{\simeq} & \RG(B,B';\F)
		\end{tikzcd}
\end{equation}
and hence a functor
\begin{equation}\label{eq:milnorfunctor}
	\RG(-;\F): M(X,N)^{\op} \lra D(\A).
\end{equation}
We note that $\RG(A,A';\F)$ can be identified with $\Phi_f(\F)$, the sheaf of vanishing cycles for
$\F$ with respect to an appropriate holomorphic function $f$ (possibly with a zero of arbitrary
order). Hence the name ``Milnor disk'',  modelled after ``Milnor fibers" in singularity theory. In
particular, we may now express the local classification data \eqref{eq:ab} at a special point
$\dot \in N$ in terms of our terminology:
\begin{enumerate}[label=\arabic*.]
	\item The space of {\em vanishing cycles}: 
		\[
			\Phi \cong  \F(\disk{1}{1})
		\]
	\item The space of {\em nearby cycles}:
		\[
			\Psi \cong  \F(\disk{0}{2})
		\]
	\item The {\em variation} map 
		\[
			a = \on{var}: \Phi  \to \Psi
		\]
		is the value of $\RG(-;\F)$ on the morphism \eqref{eq:a}.
	\item The {\em canonical} map 
		\[
			b = \on{can}: \Psi  \to \Phi
		\]
		is the value of $\RG(-;\F)$ on the morphism \eqref{eq:b}.
\end{enumerate}
See \S \ref{subsec:milntoclass} for a discussion of how to recover the relations $T_{\Psi} = \id -
ab$ and $T_{\Phi} = \id - ba$, expressing the monodromy in terms of this data. 

Our main result is based on the observations that 
\begin{enumerate}
	\item\label{obs:1} Perverse sheaves can be characterized by the fact that their relative (hyper) cohomology on
		  Milnor disks is concentrated in degree $0$,  
	\item\label{obs:2} A perverse sheaf $F$ is completely described by its values $\F(A,A')=H^0(A,A';F)$
	on Milnor disks.
\end{enumerate}
Observation \ref{obs:1} immediately implies that, for a perverse sheaf $\F$, the functor $\RG(-,\F)$
from \eqref{eq:milnorfunctor} takes values in the {\em abelian} category $\A \subset D(\A)$ given by
the heart of the standard $t$-structure.
Observation \ref{obs:2} then leads to the main result of this work: Theorem \ref{thm:milnorsheaves}
establishes that the association $\F \mapsto \RG(-;\F)$ provides an equivalence between the abelian
category of perverse sheaves on the stratified Riemann surface $(X,N)$ and the category of Milnor
sheaves: $\A$-valued presheaves on the Milnor category $M(X,N)$ that satisfy descent conditions with
respect to the cutting and pasting Milnor disks.
	
\paragraph{Method of proof: $\oo$-categorical Kan extension.}  Although the statement of Theorem
\ref{thm:milnorsheaves} is ``purely abelian'', the proof utilizes the ambient derived category and
relies on $\infty$-categorical techniques. That is, we establish a result (Corollary
\ref{cor:milnorparam}) identifiying constructible sheaves with values in a stable $\oo$-category
$\D$, and appropriately defined Milnor sheaves valued in $\D$. When $\D = \D(\A)$ is the
$\oo$-categorical enhancement of the derived category of a Grothendieck abelian category $\A$, then
perverse sheaves are recovered among all constructible complexes via the observation \ref{obs:1}
above. 
	
The method of proof of Corollary \ref{cor:milnorparam} is as follows. In general, identifying two
given $\oo$-categories is hard to achieve by hand, due to the infinite amount of coherence data
involved. The technique of Kan extensions allows for an efficient means of handling such data and
``mediating'' it across parametrizing diagram categories (see Proposition \ref{prop:kanres}).  
Using this technique, we produce equivalences between representations of various subcategories of
the larger paracyclic category $\Lambda(X,N)$ to mediate the subcategories of {\em standard disks},
{\em Milnor disks}, and {\em bounded disks}. In this framework, we provide an alternative
construction of the presheaf $\RG(-,\F)$ on the Milnor category $M(X,N)$ as a Kan extension from the
category of standard disks (cf. \S \ref{sec:milnor}).  
	
Corollary \ref{cor:milnorparam} and various technical tools developed for its proof, provide not
only a stepping stone for the more classical-looking Theorem \ref{thm:milnorsheaves} but present a
possible framework for the generalization to perverse schobers. In that generalization, a stable
$\oo$-categorical enhancement of triangulated categories is important from the very beginning.

\paragraph{The role of paracyclic Segal objects.} 

Our approach to perverse sheaves via Milnor sheaves naturally involves structures familiar in the
theory of cyclic homology \cite{connes, elmendorf, loday}. One of them is the {\em paracyclic
category} $\Lambda_\oo$ which can be regarded as the universal central extension (by $\ZZ$) of the
cyclic category $\Lambda$ of Connes \cite{connes}.   

Namely, in the most classical case, when $(X,N)$ is the disk $(\DD,\{0\})$ with the origin as special point,
a Milnor sheaf can be uniquely recovered from its values on Milnor disks containing $0$. These disks
form a subcategory of $M(\DD,\{0\})$ equivalent to the paracyclic category $\Lambda_{\oo}$, and 
our approach identifies $\A$-valued perverse sheaves, with the following structures: 
{\em paracyclic objects $Y: \Lambda_\oo^\op\to\A$ whose restriction to
$\Delta^\op \subset \Lambda_\oo^\op$ is a Segal \cite{bergner, HSS} simplicial objects} (see
Corollary \ref{cor:local}). Further, the equivalence of such structures with the more customary
classification data \eqref{eq:ab} can be understood as a special instance the duplicial 
Dold-Kan correspondence (see \S \ref{subsection:para-DK}).

This point of view turns out to be important for the generalization to perverse schobers. The
corresponding analog of a perverse sheaf on the disk is, as mentioned above, a spherical adjunction.
It turns out that any such adjunction gives, via a variant of the relative Waldhausen
$S_{\bullet}$-construction \cite{waldhausen}, rise to a {\em paracyclic object whose restriction to
$\Delta^\op$ is 2-Segal}, i.e., satisfies a 2-dimensional generalization of the Segal condition
introduced in \cite{HSS}. Such data then forms the local data comprising the structure of a perverse
schober, as will be explained in subsequent work. 
	 
\paragraph{Relation to previous work.} The dream of defining perverse sheaves in a way that would
be at the same time topological (avoiding analysis and $D$-modules) and abelian-categorical
(avoiding derived categories) is of course as old as the theory of perverse sheaves itself. We
should particularly mention the 1990 preprint of MacPherson \cite{macpherson:intersection} that
introduced (in arbitrary dimension) the concept of {\em Fary sheaves} which are certain ``cohomology
theory'' data on an appropriate class of pairs $(U_+, U_-)$ of opens in a stratified manifold.  Our
concept of a Milnor sheaf can be seen as an adaptation and a simplification of that of a Fary sheaf
to the case of two real dimensions, when instead of a functor associating a  {\em graded} vector
space  (i.e., {\em several} cohomology groups) to a  pair of opens, we have a functor associating a
single vector space, more in line with the idea of a ``sheaf''.

\paragraph{Acknowledgements.} We would like to thank J. Francis, D. Gaitsgory and R. D. MacPherson
for useful discussions that influenced the direction of this work. We are further grateful to V.
Schechtman who has contributed to this work but wishes to not be listed as a coauthor. 
T.D. acknowledges the support of the VolkswagenStiftung through the Lichtenberg
Professorship Programme. The research of T.D. is further supported by the Deutsche Forschungsgemeinschaft 
under Germany‘s Excellence Strategy -- EXC 2121 ``Quantum Universe'' -- 390833306. The research of M.K. was supported by 
World Premier International Research Center Initiative (WPI Initiative),
 MEXT, Japan. The work of Y.S. was supported by Munson-Simu Star award. 

\numberwithin{equation}{subsection}
 
\section{Perverse sheaves on stratified surfaces}\label{subsec:per-sh}

\subsection{Perverse sheaves with values in abelian categories} \label{par:sheaf-abel}

\paragraph{Sheaves with values in abelian categories.} 

Let $\A$ be an Grothendieck abelian category. In particular, $\A$ has arbitrary products and projective limits. 
 
For any topological space $X$ we denote by $\Sh(X,\A)$ the
category of $\Ac$-valued sheaves over $X$. By definition, such a sheaf $\F$ is a contravariant functor
from the poset of opens in $X$ into $\Ac$, satisfying descent.  That is, for any open covering $\{U_i\}$ of an open set $U$ the map
\[
	\F(U) \lra \Ker \biggl\{ \prod_i \F(U_i) \lra \prod_{i,j} \F(U_i\cap U_j)\biggr\}
\]
is an isomorphism.
 
By $D(X, \Ac)$ we denote the
(unbounded) derived category of $\Sh(X,\A)$.  We consider it as a triangulated category. 

For any continuous map $f: X\to Y$ of topological spaces we have
 the standard adjoint functors
   \[
 f^*: D(Y,\A)\to D(X, \A), \quad Rf_*: D(X,\A)\to D(Y, \A). 
 \]
If $X,Y$ are locally compact, we also have the functors
  \[
 Rf_!: D(X,\A)\to D(Y, \A), \quad  f^!: D(Y,\A)\to D(X, \A),
 \]
with their standard adjunctions, cf. \cite{kashiwara-schapira}.


\paragraph{Decompositions, stratifications and exit paths.} 
Concerning stratified spaces, we follow the terminology of
 \cite{GM-morse} pt.II \S 1.1-2.
 
 Thus, a {\em decomposition} of a topological space $X$ is a collection $\mcS$
 of locally closed subsets $S\in\mcS$ called {\em strata} such that 
 $X=\bigsqcup_{S\in\mcS} S$ is a disjoint decomposition and the closure
 of a stratum is a union of strata. The set $\mcS$ acquired then a partial
 order $\preceq$ by inclusion of the closures, i.e.,   $S\preceq S'$ if $S\subset\ol{S'}$.   For each $x\in X$ we denote by  $S_x\in \mcS$ the stratum
 containing $x$. A {\em decomposed space} $(X,\mcS)$ is a space equipped
with a decomposition.
 
   The concept of decomposition is identical to that
of an $(\mcS, \preceq)$-stratification in the sense of \cite{HA} Def. A.5.1.  
Recall that the latter defined as a continuous map $f: X\to\mcS$,
where the poset $\mcS$ is given the topology consisting of 
{\em upwardly closed sets}, i.e., of  $\I\subset\mcS$  
  such that  $S\in \I$ implies $S'\in \I$ whenever $S\preceq S'$. 
 Explicitly,  the map $f$ is given by $f(x)=S_x$.

 Let $(X,\mcS)$ be a decomposed space.  We denote the inclusions of the
strata by $i_S: S \to X$. 
By $\Sh(X,\mcS, \A)\subset \Sh(X,\A)$ we denote the category of sheaves 
$\F$ which are constructible with respect to $\mcS$,
i.e.,  such that each $i_S^*\F$ is locally constant on $S$. 
  By $D(X, \mcS; \A)\subset D(X,\A)$
we denote the subcategory of complexes  of sheaves $\F$  whose cohomology sheaves 
$\ul H^i(\F)$ are constructible
with respect to $\mcS$.

  Let us recall the concept of exit paths for $(X,\mcS)$, originally introduced by MacPherson,
 see \cite{treumann} for a more detailed treatment. 
For $x\in X$ we denote by $S_x\in\mcS$ the stratum containing $x$. This gives a partial
order $\preceq$ on $X$ (as a set) given by $x\preceq y$, if
$S_x\preceq S_y$, i.e.,  $S_x\subset\ol{S_y}$. 
An {\em exit path} for $(X,\mcS)$ is a continuous parametrized path $\gamma: [0,1]\to X$
which is monotone with respect to $\prec$, i.e., such that for $t_1\leq t_2$ we have
$\gamma(t_1)\preceq\gamma(t_2)$. The {\em category of exit paths} $\Exit(X,\mcS)$
has, as objects, all points $x\in X$, with $\Hom_{\Exit(X,\mcS)}(x,y)$ being the
set of isotopy classes of exit paths $\gamma$ with $\gamma(0)=x$ and $\gamma(1)=y$. 
Thus $\Exit(X,\mcS)$ can be considered as a stratified version of the fundamental
groupoid of $X$ (to which it reduces in the particular case when $\mcS$ consists of just one
stratum $X$). By reversing the direction of the paths (or passing to
the opposite category) we get the {\em category of entrance paths}
$\Entr(X,\mcS) = \Exit(X,\mcS)^\op$. 

We will use some particular types of decompositions in which one imposes
various ``conicity'' conditions describing the neighborhood of a stratum 
in the closure of a larger stratum:

\begin{itemize}
\item[(1)]  Whitney stratifications, see 
 \cite{GM-morse} pt.II \S 1.2. In this case the strata are $C^\oo$-manifolds.
 
 \item[(2)]  Topological stratifications, see \cite{GM-IC-II} and \cite{treumann} \S 3.1. 
 In this case the strata are topological manifolds. 
 
 \item[(3)]  Conical stratifications, see \cite{HA} Def. A.5.5. In this case strata
 are not required to be manifolds, but   near a stratum $S$, the space $X$
 is  locally  identified with the product of   $S$ and the cone over another decomposed
 space with strata labelled by $S'\in\mcS$ with $S\prec S'$. 
 
 \end{itemize}

It is known that these three conditions  are of increasing generality, i.e.,
(1)$\Rightarrow$(2)$\Rightarrow$(3).

\begin {prop}\label{prop:exit}
Let $(X,\mcS)$ be a space with a conical stratification. The 
  category $\Sh(X,\mcS, \A)$ is equivalent to $\Fun(\Exit(X,\mcS), \A)$
(the category of covariant functors). 
\end{prop} 

 \noindent{\sl Proof:}  For topological stratifications, this 
 is the original result of MacPherson, see 
  \cite{treumann} Th. 1.2. For conical stratifications this follows from 
  \cite{HA} Th. A.9.3 which gives an 
   $\oo$-categorical upgrade of $\Exit(X,\mcS)$.  
   
\qed


Suppose now that $X$ is a complex manifold and $\mcS$ is a complex analytic Whitney
stratification of $X$. 
By $\PS(X,\mcS, \A)\subset D(X, \mcS, \A)$ we denote the subcategory of {\em perverse sheaves}
(with respect to the middle perversity). Recall \cite{BBD}\cite{kashiwara-schapira} that $\F\in \P(X,\mcS, \A)$
iff two conditions are satisfied:
\begin{enumerate}
	\item[($P^+$)] For every $S \in \mcS$, we have $\ul H^n(i_S^*\F) = 0$ for $n > -\dim_\CC(S)$, 
	\item[($P^-$)] For every $S \in \mcS$, we have $\ul H^n(i_S^!\F) = 0$ for $n < -\dim_\CC (S))$. 
\end{enumerate}
It is well known \cite{BBD} that the category $\PS(X,\mcS;\A)$  is 
the heart of a $t$-structure, and so is an
abelian category.


\paragraph {The case of stratified surfaces.}
 We specialize to the case of $\dim_\CC(X)=1$,  so  $X$ is a Riemann surface, possibly noncompact and
with nonempty boundary. We fix a finite subset $N \subset X$ of interior points which we refer
to as {\em special points} and denote the corresponding stratification $X = N \cup (X \setminus N)$
by $\mcS = \mcS_N$.   This gives a topological stratification  and we adopt the following definition. 

\begin{defi} 
 By a {\em stratified surface}  we mean  a pair $(X,N)$ consisting of :
 \begin{itemize}
	 \item[(1)] A topological manifold $X$ of real dimension
$2$,  possibly noncompact and  with boundary. 

\item[(2)]  A finite subset $N \subset X$  of interior points which we refer
to as {\em special points}.
 \end{itemize}
 
   We denote by $j: X\setminus N\to X$ and $i: N\to X$ the embeddings of the strata. 
   \end{defi}

 
Let us fix a Grothendieck abelian category $\A$. We denote by $D(X,N;\A)\subset D(X,\A)$ the full
subcategory of complexes whose cohomology sheaves are constructible with respect to the
stratification $\mcS_N$, i.e., in our case, locally constant on $X\setminus N$. 

Further, the concept of a perverse sheaf  makes sense in this context, and is given explicitly as
follows. 
 
\begin{defi}\label{defi:perverse} Let $(X,N)$ be a stratified surface and $\A$ a Grothendieck
	abelian category. An object $\F$ of $D(X, N; \A)$ is called {\em perverse} if 
	\begin{enumerate}
		\item $j^*\F$ is isomorphic to $L[1]$ where $L$ is a local system on $X\setminus N$ with values in $\A$, 
		\item $H^n(i^*\F) = 0$ for $n > 0$,
		\item $H^n(i^!\F) = 0$ for $n < 0$.
	\end{enumerate}

\end{defi}

The category of perverse sheaves with respect to $N$ will be denoted $\PS(X,N;\A)$. As explained
above, it is an abelian category.

\subsection{Milnor disks, Milnor pairs and the purity property}

We
denote by $\DD \subset \CC$ the closed unit disk. 
Let $(X,N)$ be a surface $X$ with a set of special points $N \subset X$ as in \S \ref{par:sheaf-abel}. 
By a {\em closed disk} we mean a subspace  $A\subset X$ homeomorphic to $\DD$. 

\begin{defi}\label{def:miln-disk} 
A {\em Milnor disk} in $(X,N)$ is a pair $(A,A')$, where:
\begin{itemize}
\item[(1)] $A\subset X$ is a closed disk containing at most one special point.

\item[(2)] $A'\subset\del A\simeq S^1$ is a disjoint union of finitely many closed arcs, different from
$\emptyset$ and the whole $\del A$. 
\end{itemize}
\end{defi} 

See the left of Fig. \ref{fig:basic}. The concept of a Milnor disk can be compared with the following possibly more intuitive concept.

\begin{defi}\label{def:miln-pair}
A  {\em Milnor pair} for $(X,N)$ is a pair
 $(U,U')$,  $U' \subset U$, of  closed  subsets of $X$, such that 
	\begin{enumerate}
		\item $U$ is a closed disk containing at most one special point.
		
		\item $U'$ is a finite, nonempty, disjoint union of closed disks $\{U_i\}_{i \in I}$ such that $K=U \setminus U'$
			is contractible.
	\end{enumerate}
\end{defi}

 \begin{figure}[ht]
   \centering
    
	\def\svgwidth{5cm}
	\raisebox{-.5\height}{
\begingroup%
  \makeatletter%
  \providecommand\color[2][]{%
    \errmessage{(Inkscape) Color is used for the text in Inkscape, but the package 'color.sty' is not loaded}%
    \renewcommand\color[2][]{}%
  }%
  \providecommand\transparent[1]{%
    \errmessage{(Inkscape) Transparency is used (non-zero) for the text in Inkscape, but the package 'transparent.sty' is not loaded}%
    \renewcommand\transparent[1]{}%
  }%
  \providecommand\rotatebox[2]{#2}%
  \newcommand*\fsize{\dimexpr\f@size pt\relax}%
  \newcommand*\lineheight[1]{\fontsize{\fsize}{#1\fsize}\selectfont}%
  \ifx\svgwidth\undefined%
    \setlength{\unitlength}{153.60156937bp}%
    \ifx\svgscale\undefined%
      \relax%
    \else%
      \setlength{\unitlength}{\unitlength * \real{\svgscale}}%
    \fi%
  \else%
    \setlength{\unitlength}{\svgwidth}%
  \fi%
  \global\let\svgwidth\undefined%
  \global\let\svgscale\undefined%
  \makeatother%
  \begin{picture}(1,1.07368949)%
    \lineheight{1}%
    \setlength\tabcolsep{0pt}%
    \put(0,0){\includegraphics[width=\unitlength,page=1]{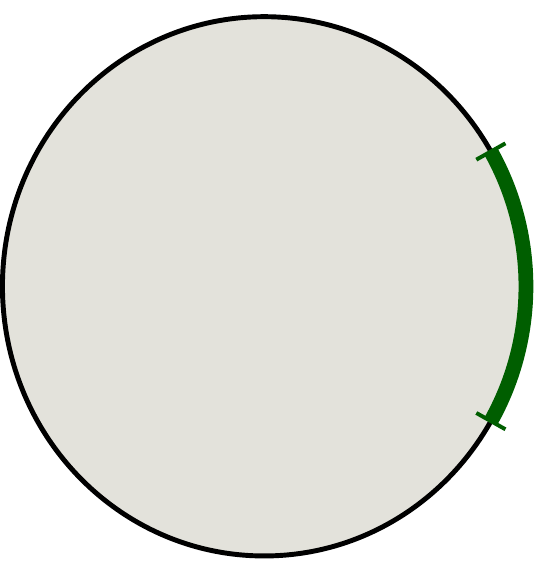}}%
    \put(0.28599036,0.88074992){\color[rgb]{0,0,0}\makebox(0,0)[lt]{\lineheight{1.25}\smash{\begin{tabular}[t]{l}$A'$\end{tabular}}}}%
    \put(0.8614589,0.52500576){\color[rgb]{0,0,0}\makebox(0,0)[lt]{\lineheight{1.25}\smash{\begin{tabular}[t]{l}$A'$\end{tabular}}}}%
    \put(0.26855196,0.16054241){\color[rgb]{0,0,0}\makebox(0,0)[lt]{\lineheight{1.25}\smash{\begin{tabular}[t]{l}$A'$\end{tabular}}}}%
    \put(0.83530108,0.98538048){\color[rgb]{0,0,0}\makebox(0,0)[lt]{\lineheight{1.25}\smash{\begin{tabular}[t]{l}$A$\end{tabular}}}}%
    \put(0.43596092,0.59650343){\color[rgb]{0,0,0}\makebox(0,0)[lt]{\lineheight{1.25}\smash{\begin{tabular}[t]{l}$x \in N$\end{tabular}}}}%
    \put(0,0){\includegraphics[width=\unitlength,page=2]{disk_1_3_labelled.pdf}}%
  \end{picture}%
\endgroup%
} \quad\quad\quad\quad
	\def\svgwidth{5cm}
	\raisebox{-.5\height}{
\begingroup%
  \makeatletter%
  \providecommand\color[2][]{%
    \errmessage{(Inkscape) Color is used for the text in Inkscape, but the package 'color.sty' is not loaded}%
    \renewcommand\color[2][]{}%
  }%
  \providecommand\transparent[1]{%
    \errmessage{(Inkscape) Transparency is used (non-zero) for the text in Inkscape, but the package 'transparent.sty' is not loaded}%
    \renewcommand\transparent[1]{}%
  }%
  \providecommand\rotatebox[2]{#2}%
  \newcommand*\fsize{\dimexpr\f@size pt\relax}%
  \newcommand*\lineheight[1]{\fontsize{\fsize}{#1\fsize}\selectfont}%
  \ifx\svgwidth\undefined%
    \setlength{\unitlength}{151.28287bp}%
    \ifx\svgscale\undefined%
      \relax%
    \else%
      \setlength{\unitlength}{\unitlength * \real{\svgscale}}%
    \fi%
  \else%
    \setlength{\unitlength}{\svgwidth}%
  \fi%
  \global\let\svgwidth\undefined%
  \global\let\svgscale\undefined%
  \makeatother%
  \begin{picture}(1,1.0315915)%
    \lineheight{1}%
    \setlength\tabcolsep{0pt}%
    \put(0,0){\includegraphics[width=\unitlength,page=1]{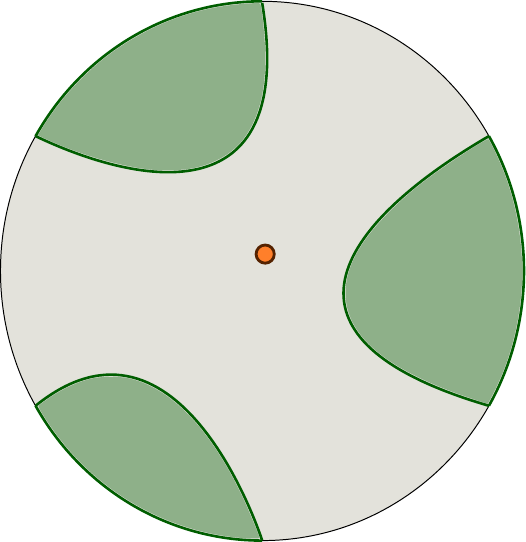}}%
    \put(0.29193808,0.83664277){\color[rgb]{0,0,0}\makebox(0,0)[lt]{\lineheight{1.25}\smash{\begin{tabular}[t]{l}$U'$\end{tabular}}}}%
    \put(0.8620623,0.49669304){\color[rgb]{0,0,0}\makebox(0,0)[lt]{\lineheight{1.25}\smash{\begin{tabular}[t]{l}$U'$\end{tabular}}}}%
    \put(0.24767409,0.19215468){\color[rgb]{0,0,0}\makebox(0,0)[lt]{\lineheight{1.25}\smash{\begin{tabular}[t]{l}$U'$\end{tabular}}}}%
    \put(0.86029175,0.95704161){\color[rgb]{0,0,0}\makebox(0,0)[lt]{\lineheight{1.25}\smash{\begin{tabular}[t]{l}$U$\end{tabular}}}}%
    \put(0.43004291,0.58522164){\color[rgb]{0,0,0}\makebox(0,0)[lt]{\lineheight{1.25}\smash{\begin{tabular}[t]{l}$x \in N$\end{tabular}}}}%
  \end{picture}%
\endgroup%
}
  
\caption{A Milnor disk $(A,A')$ and a Milnor pair $(U,U')$.}\label{fig:basic}
\end{figure}

\noindent Thus a Milnor disk can be seen as a Milnor pair $(U,U')$ with $U'$ being very thin, reducing to
a union of boundary arcs, see Fig. \ref {fig:basic}. Up to  homotopy equivalence, there is no
difference between the two concepts. 

\begin{ex}\label{ex:milnor}
Let $X$ be a Riemann surface (1-dimensional complex manifold), $z$ be a holomorphic coordinate near
an interior point $x\in X$ and $f$ be a holomorphic function defined near $x$ such that $f(x)=0$.
Then for sufficiently small $\eps > \delta>0$ the pair formed by 
\[
	 U=\{|z|\leq \eps\}, \quad U' = \{|z| \leq  \eps, \,\, \Re(f(z)) \leq \delta\}
\]
is a Milnor pair. This explains our terminology, motivated by the
concept of  Milnor fibers in singularity theory. Note that the cardinality $|\pi_0(U')|$ is equal
to $\ord_x(f)$, the order of vanishing of $f$ at $x$. 
\end{ex}

The role of Milnor disks for our purposes stems from the following:

\begin{prop}[(Purity property)]\label{prop:purity} Let $(X,N)$ be a stratified surface, let $\A$ be a Grothendieck
	abelian category, and let $\F$ be an object of the derived constructible category
	$D(X,N;\A)$. Then the following are equivalent:
	\begin{enumerate}
		\item[(i)] $\F$ is a perverse sheaf.
		\item[(ii)] For every Milnor disk $(A,A')$, 
		the relative hypercohomology $H^i(A,A';\F)$ vanishes for $i\neq 0$.
	\end{enumerate}
\end{prop}
 
We will refer to the condition (ii) as {\em purity}.

\begin{proof}[Proof of Proposition \ref{prop:purity}]
 (i) $\Rightarrow$ (ii): Assume that $\F$ is perverse. 
 
 Assume first that $A$ either contains no special point or contains exactly one
 special point $x$ in its interior. Note that the first possibility is really a particular
 case of the second, as we can always introduce a ``dummy'' special point,
 where a singularity is allowed but not present. So we assume that the second
 possibility holds. Denote by by $i_x: \{x\} \to X$ the inclusion of the
	point. Note that $R\Gamma_{\{x\}}(A,\F)\simeq i_x^!\F$ and so its
	cohomology, by Definition \ref{defi:perverse}(3), is concentrated in
	degrees $\geq 0$. Further, $R\Gamma(A,F)\simeq i_x^*\F$ and so its
	cohomology, by  Definition \ref{defi:perverse}(2), is concentrated in
	degrees $\leq 0$. Consider now the following diagram with rows and columns
	being exact triangles:

\be\label{eq:big-dia}
\xymatrix{
R\Gamma_{\{x\}}(A,A'; \F)\ar[d] \ar[r]& R\Gamma(A,A';\F)\ar[d] \ar[r]& R\Gamma(A\- \{x\}, A'\-\{x\}; \F)\ar[d]
\\
R\Gamma_{\{x\}}(A;\F)\ar[d] \ar[r]& R\Gamma(A,\F)\ar[d]^c \ar[r]^a &
R\Gamma(A\-\{x\};\F)\ar[d]^b
\\
R\Gamma_{\{x\}}(A',\F) \ar[r] & R\Gamma(A';\F) \ar[r]^d& R\Gamma(A'\-\{x\};\F). 
}
\ee	
Note that $\F|_{A\-\{x\}}\= L[1]$, a local system in degree $(-1)$ and $A\-\{x\}$
is homotopy equivalent to $S^1$. So $R\Gamma(A\-\{x\}; \F)$ has
cohomology only in degrees $\{-1,0\}$. The LES of cohomology of the middle
row of the diagram gives, using the information above, the following:
 \begin{equation}\label{eq:bound1}
 \begin{gathered}
		H^n(A, \F) = 0 \quad  \text{for $n \notin \{-1,0\}$},
		\\
		H^n_{\{x\}} (A, \F) = 0 \quad  \text{for $n \notin \{0,1\}$}. 
		\end{gathered}
	\end{equation}	
 Look now at the middle column of the diagram. Since $R\Gamma(A';\F)$ is
 concentrated in degree $(-1)$, in order to show that $R\Gamma(A,A';\F)$
 has cohomology only in degree $0$, it suffices to show that
 $c: H^{-1}(A;F)\to H^{-1}(A';\F)$ is injective. For this, it suffices to prove
 that the maps induced by $a$ and $b$ on $H^{-1}$ are injective.
 For $a$ it follows from the fact \eqref{eq:bound1} that $R\Gamma_{\{x\}}(A;\F)$
 has no cohomology in degree $(-1)$. For $b$, we use the identification
 $\F|_{A\-\{x\}}\= L[1]$ as above. Then the statement becomes that
 $H^0(A\-\{x\}; L)\to H^0(A';L)$ is injective which is clear. 
 
 Suppose now that the special point $x$ lies in $\del A$. If $x\in A'$, then
 by excisison we reduce to the case when $A\cap N=\emptyset$ treated above.
 So let $x\in\del A\- A'$. In this case the argument is similar to the above, as
 $A\-\{x\}$ is contractible, and so
 $\F|_{A\-\{x\}}=L[1]$ has cohomology only in degree $(-1)$. 
 
 \vskip .2cm

 	(ii) $\Rightarrow$ (i): Vice versa, suppose that $\F$ is an object of $D(X,N;\A)$  
	satisfying the purity condition.  Let $A \subset X$ be a closed disk not containing
	any special points. Let $A' \subset \del A$  be a  disjoint union of two closed arcs, 
	 so that
	$(A,A')$ is a Milnor disk. Since by our assumptions, $\F|_A$ has locally
	constant, hence constant cohomology, 
	it is straightforward to conclude that 
	\[
		\RG(A,\F) \simeq \RG(A,A';\F)[1].
	\]
	By purity,  this implies that $j^* \F[-1]\= L$ is quasi-isomorphic to a  single local system with values in 
	$\A$. This shows Condition (1) of Definition \ref{defi:perverse}. 
	
	Now let $A$ be an closed disk that contains exactly one special point
	$x$ in its interior. Let 
	$A'\subset \del A$ be the disjoint union of two arcs. 
	We consider again the diagram \eqref{eq:big-dia},
	arguing now  ``in the other direction''. 
	
	That is, look at the middle column.
	By purity, $R\Gamma(A,A';\F)$ has cohomology only in degree $0$.
	But since $j^* \F[-1]=L$ is a single local system in degree $0$, the complex
	$R\Gamma(A';F)$ has cohomology only in degree $(-1)$.
	Therefore $R\Gamma(A,F)\= i_x^*\F$ has cohomology only in degrees
	$\{-1,0\}$, thus establishing Condition (2) of Definition \ref{defi:perverse}.

	Next, look at the left column. Clearly, $R\Gamma_{\{x\}}(A'; \F)=0$, as
	 $x\notin A'$, and 
	so $i_x^! F\= R\Gamma_{\{x\}}(A, \F)$ is identified with
	$R\Gamma_{\{x\}}(A,A'; \F)$. Now, the latter can be analyzed via
	the top row of the diagram, which contains $R\Gamma(A,A';\F)$,
	with cohomology in degree $0$ and $R\Gamma(A\-\{x\}, A'\-\{x\};\F)$
	which, we claim,  has cohomology only in degree $0$.
	This follows from looking at the right column, where the statement
	reduces to the claim that $H^0(A\-\{x\};L) \to H^0(A'; L)$ is 
	injective. Therefore $i_x^! \F$ has cohomology only in degrees $\{0,1\}$, 
  thus establishing Condition (3) of Definition \ref{defi:perverse}. 
 \end{proof}

\begin{rem}\label{rem:Milnor-van}
	Assume that we are in the situation of Example \ref{ex:milnor}. Then 
	$R\Gamma(U,U';  \F)$ is identified with  $\Phi_f(\F)_x$, the stalk at $x$ of the complex of
	vanishing cycles for $\F$ with respect to $f$, see \cite{kashiwara-schapira}. It is well
	known ({\em loc. cit.}) that $\Phi_f(\F)$ is itself a perverse sheaf which, in our case,
	amounts to saying that $\Phi_f(\F)_x$ is quasi-isomorphic to a single vector space in degree
	$0$. This provides an alternative proof of purity for such Milnor pairs, at least in the
	classical case when $\A$ is the category of vector spaces over a field. 
\end{rem}
 
\section{The paracyclic category and constructible sheaves}
\label{sec:cyclicsurface}

In this section, we will introduce the paracyclic category $\Lambda(X,N)$ of a stratified surface and explain how
the formalism of Kan extensions, applied to a directed version of $\Lambda(X,N)$, can be used to
describe the Verdier duality of the derived constructible category. The ideas and constructions
introduced in this section serve as a preparation for the main part of this work, \S
\ref{sec:milnor}, where we will apply similar techniques to parametrize perverse sheaves in terms of
the subcategory $M(X,N) \subset \Lambda(X,N)$ of Milnor disks. 

\subsection{The standard paracyclic category and the Ran space of the circle}\label{par:cyc-paracyc} 

Recall that the standard {\em simplex category} $\Delta$ has, as objects,
the standard finite nonempty ordinals $[n] = \{0,1,\cdots, n\}$, $n \geq 0$,
with morphisms being monotone maps. The morphisms of $\Delta$ are
 generated by the coface and codegeneracy maps
\[
\begin{gathered}
\delta_i: [n-1] \lra [n], \,\,\, i=0, \cdots, n \quad \text{(omitting $i$)};
\\
 \sigma_j:  [n+1] \lra [n] \,\,\, j=0,\cdots, n  \quad \text{(repeating $j$)},
\end{gathered}
 \]
 subject to well known relations, see, e.g., \cite{connes}, Ch. III, App.A, Prop.2.
 We denote by $\Delta^\surj\subset\Delta$ the subcategory with the same objects
 and only surjective maps as morphisms. In other words, morphisms of $\Delta^\surj$
 are generated by the $\sigma_j$ only. As usual, we call a {\em simplicial object}  in a category
 $\A$ a contravariant functor $Z: \Delta\to\A$. Thus $Z$ consists of objects
 $Z_n = Z([n]) \in\A, n\geq 0$ and morphisms (face and degenaracy maps)
 \[
 \del_i: Z_n\lra Z_{n-1}, \,\, i=0,\cdots, n; \quad  s_j: Z_n\lra Z_{n+1},\,\, j=0,\cdots, n+1,
 \]
 satisfying the relations dual to those among the $\delta_i$ and $\sigma_j$. 
  We will also use the term 
 {\em half-simplicial object} for a contravariant functor $\Delta^\surj\to \A$. 
 Thus a half-simplicial object has only degeneracy maps but no face maps.

  \begin{defi}[(\cite{connes} Ch. III App. A, 
\cite{loday} Def. 6.1.1)]\label{def:paracyc}
 (a) The standard {\em paracyclic category} $\Lambda_\infty$ has the  objects $\cn$, $n\geq 0$
 which are in bijection with those of  $\Delta$. Its
 morphisms  are generated by those of $\Delta$ (i.e., the $\delta_i: \langle n-1\rangle  \to \cn$
 and $ \sigma_j:  \langle n+1 \rangle  \to \cn$ as above satisfying the same relations) 
 together with additional {\em automorphisms}
 $\tau_n:\cn\to\cn$ which are subject to the following  relations:
  \[
\begin{gathered}
\tau_n\delta_i = \delta_{i-1}\tau_{n-1} \text{ for } 1\leq i\leq n, \quad \tau_n\delta_0=\delta_n;
\\
\tau_n\sigma_i = \sigma_{i+1}\tau_{n+1} \text{ for } 
1\leq i\leq n, \quad \tau_n\sigma_0 = \sigma_n\tau_{n+1}^2;
\end{gathered}
\]

(b) The {\em cyclic category} $\Lambda$ is  obtained from $\Lambda_\infty$ by imposing the additional relations
$\tau_n^{n+1}=\Id$. 
 \end{defi}

\noindent The following proposition  is well known, see \cite{drinfeld}.  
   It  can be expressed by saying that $\Lambda_\infty$ is a {\em central extension of $\Lambda$
    by $\ZZ$}. 

  \begin{prop}\label{prop:lambda-infty-cex}
  (a) The automorphisms $\tau_n^{n+1}\in\Hom_{\Lambda_\infty}(\cn, \cn)$ form  a central system 
  (i.e., define a natural transformation from the identity functor to itself). 
  
  (b) Let 
  $p: \Lambda_\infty\to\Lambda$  be  natural functor (identical on objects, surjective on morphisms).  
   The fibers of each induced map
 \[
  \Hom_{\Lambda^\infty} (\cm, \cn) \lra \Hom_{\Lambda}(\cm,\cn)
  \]
  are principal homogeneous spaces with respect to the action of $\ZZ$ given by composition with 
  powers of $\tau_m^{m+1}$ or, what  by (a)  is the same, by composition with powers $\tau_n^{n+1}$. \qed
    \end{prop}

    We also denote $\Lambda_\oo^\surj\subset\Lambda_\oo$ the subcategory on the same objects
    with the morphisms generated by the $\sigma_j$ and $\tau_n$ only. 
    By a {\em paracyclic} {\em object} in a category $\Ac$
    we will mean a contravariant functor $Z: \Lambda_\oo\to \A$ As for simplicial objects, we write $Z_n$
    for the value of $Z$ on $\cn$ and $\del_i, s_j, t_n$ for the values
    on $\delta_i, \sigma_j, \tau_n$. 
    By a {\em half-paracyclic object} we will mean  a contravariant functor 
     $\Lambda_\oo^\surj\to\A$. 
    
        \begin{rem}\label{rem:lambda-self-dual}
    The categories $\Lambda$ and $\Lambda_\infty$ are self-dual, i.e., isomorphic to their opposite
    categories \cite {connes} \cite{elmendorf}. In fact, by introducing the additional co-degeneracies
    $\sigma_{n+1}= \tau_n \sigma_n\tau_{n+1}^{-1}: \langle n+1\rangle \to \cn$, one can
    write their presentations in a manifestly self-dual way, so that cofaces and co-degeneracies will be
    dual to each other. 
    \end{rem}
    
    \paragraph{A partial interpretation via the Ran space.} 
    We recall the topological version of the Ran space construction \cite{BD}. As
    pointed out in \cite{BD}, this version  goes back to
    Borsuk and Ulam \cite{borsuk-ulam}. 
    
    Let $M$ be a $C^\oo$-manifold. The {\em Ran space} of $M$ is the set $\Ran(M)$
    of all finite nonempty subsets $I\subset M$ equipped with a natural (Vietoris) topology.
    If  we choose  a metric on $M$ inducing the topology, then $\Ran(M)$ can be metrized using the corresponding Hausdorff distance.
The space $\Ran(M)$ has a filtration by closed subspaces $\Ran^{\leq d}(M) = \{I\subset M: \, |I|\leq d\}$, and the complement
\[
\Ran^{\leq d}(M) \-\Ran^{\leq d-1}(M) \= \Sym^d_\neq(M)
\]
is the configuration space of unordered $d$-tuples of distinct points in $M$. 
In this way each $\Ran^{\leq d}(M)$ becomes a Whitney stratified space, and
$\Ran(M)$ can be considered as a (infinite-dimensional)  space
with a conical stratification, see \S \ref{par:sheaf-abel}. 
 In particular, we can speak about the category of exit paths $\Exit(\Ran(M))$
and, for a Grothendieck abelian category $\A$, about  $\A$-valued constructible sheaves on $\Ran(M)$ (with respect to
 the stratification by the $\Sym^d_\neq(M)$). 
 
 \begin{rems}\label{rem:ran-bacteria}
 (a) An exit path in $\Ran(M)$ can be seen as a history of a colony of bacteria living in $M$
 which can move and multiply (by splitting) but not merge together, and cannot die, see
 Fig. \ref{fig:exit-ran}.

 (b) A constructible sheaf $\F$ on $\Ran(M)$ assigns to any finite nonempty $I\subset M$
 an object $\F_I\in\A$ (the stalk). When $I$ ``evolves'' into $J$ by moving and splitting,
 we have a morphism $\F_I\to\F_J$ (the generalization map). 
 \end{rems}
 
 \begin{figure}[h]
 \centering
 \begin{tikzpicture}[scale=0.4]
 \def\drarr{\draw  [decoration={markings,mark=at position 0.9 with
{\arrow[scale=1.5,>=stealth]{>}}},postaction={decorate},
line width=.2mm]}
 
 \node at (0,5){\small$\bullet$}; 
\node at (0,8){\small$\bullet$}; 
\node at (0,11){\small$\bullet$}; 

 \node at (10,3){\small$\bullet$}; 
 \node at (10,4){\small$\bullet$}; 
 \node at (10,6){\small$\bullet$}; 
 \node at (10,7){\small$\bullet$}; 
 \node at (10,8){\small$\bullet$};  
 \node at (10,10){\small$\bullet$}; 
 \node at (10,12){\small$\bullet$};    
 
 \drarr (0,8) .. controls (3,7.5) and (6,8.5) ..  (10,8);  
 
 \drarr (0,5) -- (2,5);  
 \drarr (2,5) .. controls (6,7) .. (10,7); 
 \drarr (2,5) .. controls (6,3) .. (10,3); 
 \drarr (2,5) -- (6,5) ; 
 \drarr (6,5) .. controls (8,6) .. (10,6); 
  \drarr (6,5) .. controls (8,4) .. (10,4); 
  \drarr (0,11) .. controls (1.5, 11.5) and (3, 10.5) .. (5,11);  
\drarr (5,11) .. controls (8,12) .. (10,12); 
\drarr (5,11) .. controls (8,10) .. (10,10);

 \end{tikzpicture}
 \caption{An exit path in $\Ran(M)$.}\label{fig:exit-ran}
 \end{figure}
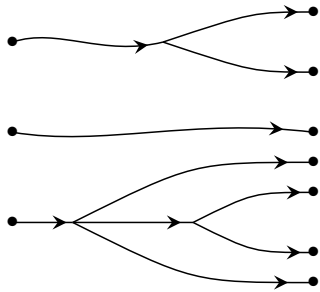

 Let us focus, in particular, on  the Ran spaces of the real line $\RR$ and the circle $S^1$.

\begin{ex}
 It goes back to Bott \cite{bott} that $\Ran^{\leq 3}(S^1)$ is homeomorphic to the $3$-sphere $S^3$. 
 Further, inside this sphere $\Ran^{\leq 1}(S^1)=S^1$ is embedded as a trefoil knot, and 
 $\Ran^{\leq 2}(S^1)$ is a Moebius band  bounding this knot. See 
 \cite{mostovoy} for a beautiful treatment using elliptic functions. 
The topology and homotopy type of
  $\Ran^{\leq d}(S^1)$ for higher $d$ was studied in \cite{tuffley-s1}. 
 \end{ex}

 The following result was proven in \cite{cepek:thesis}:
 
 \begin{prop}\label{prop:lambda-exit}
 (a) The category $\Exit(\Ran(\RR))$ is equivalent to $(\Delta^\surj)^\op$. In particular,  $\A$-valued
 constructible sheaves on $\Ran(\RR)$  can be identified with half-simplicial objects in $\A$.
 
 (b) The category  $\Exit(\Ran(S^1))$ is equivalent to $(\Lambda_\oo^\surj)^\op$. In particular,  $\A$-valued
 constructible sheaves on $\Ran(S^1)$  can be identified with half-paracyclic objects in $\A$.
 \end{prop}
 
 \noindent{\sl Proof:} (a) An exit path $\gamma$ in any $\Ran(M)$ going from
 $I$ to $J$ gives, for any $x\in I$, a tree of descendents of $x$ which terminates in
 a subset of $J$. This gives a surjection $a_\gamma: J\to I$ (the ``ancestry map''). 
  Isotopic exit paths lead to the same surjection. 
 If $M=\RR$, 
 then the order of $\RR$ makes both 
 $I$ and $J$ into nonempty finite ordinals and the surjection $a_\gamma$ is monotone.

 (b) Recall  from \cite{connes} Ch. III, Appendix A  the geometric definition
 of the cyclic  
 category $\Lambda$. For this we identify $\cn$ with the set of $(n+1)$st roots of $1$
 in the standard circle $S^1$. Then $\Hom_\Lambda(\cm, \cn)$
 is the set of connected components of the space  of degree $1$ monotone maps
 $f: (S^1,\cm)\to (S^1, \cn)$. Each such connected component has the homotopy type of
 $S^1$, and $\Hom_{\Lambda_\oo}(\cm, \cn)$ is obtained by passing to the universal
 coverings of these components. That is, $\Hom_{\Lambda_\oo}(\cm, \cn)$ is the
 set of isotopy classes of data $(f,s)$ consisting of  $f$ 
 as above together with a homotopy $s$ between  $f$ and  the identity 
 (as maps $S^1\to S^1$). Note now that 
 for $M=S^1$, an exit path $\gamma$ as in (a) gives not only a surjection $a_\gamma$
 but a well defined isotopy class of pairs $(f,s)$, where  $f:(S^1,J)\to (S^1, I)$ 
 is a monotone degree $1$ map and $s$ is
  homotopy of $f$ to the identity.\qed
  
 \begin{rem}\label{rem:full-lambda}
One would like  to extend the approach with the Ran spaces so as to realize the
  full categories $\Delta, \Lambda_\oo$ or functors out of them
   in terms of some categories of exit paths or constructible sheaves. 
   For this, in the language of Remark \ref{rem:ran-bacteria}(a), we would need
    to modify the concept
  of an exit path as a history of a colony of bacteria so as to allow the bacteria to die,
  see Fig. \ref{fig:exit-ran-death}. Then for such a ``history with deaths'' evolving from
  $I$ to $J$ we will
  still have the ancestry map $J\to I$ but it need not be surjective, as some lines may
  die out.

 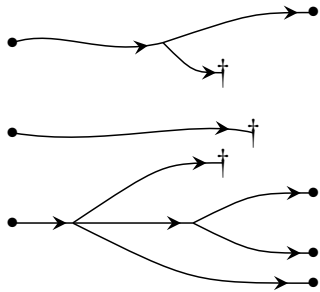
\begin{figure}[h]
 \centering
 \begin{tikzpicture}[scale=0.4]
 \def\drarr{\draw  [decoration={markings,mark=at position 0.9 with
{\arrow[scale=1.5,>=stealth]{>}}},postaction={decorate},
line width=.2mm]}
 
 \node at (0,5){\small$\bullet$}; 
\node at (0,8){\small$\bullet$}; 
\node at (0,11){\small$\bullet$}; 

 \node at (10,3){\small$\bullet$}; 
 \node at (10,4){\small$\bullet$}; 
 \node at (10,6){\small$\bullet$}; 
 
 \node at (10,12){\small$\bullet$};    
 
 \drarr (0,8) .. controls (3,7.5) and (6,8.5) ..  (8,8);  
 \node at (8,8){\large$\dagger$}; 
 \node at (7,10) {\large$\dagger$}; 
  \node at (7,7) {\large$\dagger$}; 
 
 \drarr (0,5) -- (2,5);  
 \drarr (2,5) .. controls (5,7) .. (7,7); 
 \drarr (2,5) .. controls (6,3) .. (10,3); 
 \drarr (2,5) -- (6,5) ; 
 \drarr (6,5) .. controls (8,6) .. (10,6); 
  \drarr (6,5) .. controls (8,4) .. (10,4); 
  \drarr (0,11) .. controls (1.5, 11.5) and (3, 10.5) .. (5,11);  
\drarr (5,11) .. controls (8,12) .. (10,12); 
\drarr (5,11)  .. controls (6, 10) ..  (7,10);

 \end{tikzpicture}
 \caption{An exit path in $\Ran(M)$ with deaths.}\label{fig:exit-ran-death}
 \end{figure}
 
 To account for 
 such ``exit paths with deaths'', one needs  to consider   constructible sheaves $\F$ on
  $\Ran(M)$ equipped with
 an additional {\em monotone structure} which is a system of 
   maps $\F_J\to \F_I$ given for any
nested pair $I\subset J\subset S^1$ of nonempty finite sets and transitive in nested
triples.  

We do not pursue this approach further but note that our
 point of view based on Milnor disks $(A,A')$ has $A'$, a finite union of intervals
in the circle $\del A\= S^1$, playing the role of a finite subset $I\in\Ran(\del A)$. 

\end{rem}
  
A systematic approach to the matter discussed in Remark \ref{rem:full-lambda} via ``unital'' Ran
spaces was developed in \cite{cepek:thesis, cepek:ran}. The author recovers the paracyclic category
and Joyal's categories $\Theta_n$ as unital exit path categories associated to the Ran spaces of
$S^1$ and $\RR^n$, respectively.

\subsection{The paracyclic category of a stratified surface}\label{subsec:para-strat}

Let $(X,N)$ be a stratified surface as defined above. Throughout this text, we will assume that, if
$X \cong S^2$, then $|N| \ge 2$. In this section, we introduce the paracyclic category
$\Lambda(X,N)$ of $(X,N)$ which can be seen as a certain amalgamation of the copies of $\Lambda^\oo$
associated with the circles of directions at all the points $x\in X$

\paragraph{Pant cobordisms and the paracyclic category.} 

We  will  use the notation $I=[0,1]$ for the closed  unit  interval 
and, as before, $\DD$ for the closed unit disk.

\begin{defi}\label{def:paradisk} 
By a {\em para-disk} in $(X,N)$ we mean a pair 
$(A,A')$ where $ A \subset X$ is a closed disk such that $|A \cap N| \le 1$ 
 and $A' \subset
\partial A \cong S^1$ is a compact $1$-dimensional submanifold, i.e. one of the following:
\begin{enumerate}[label=(\roman*)]
	\item[(i)] the empty set,
	\item[(ii)] a finite nonempty union of closed intervals,
	\item[(iii)]  the full boundary circle, 
\end{enumerate}
\end{defi}

Thus a Milnor disk is a particular case of a para-disk corresponding to the
possibility (ii) of Definition \ref{def:paradisk}. 
In the other two cases,   a para-disk $(A,A')$ will be called:

\begin{enumerate}[label=(\alph*)]
	\item a  {\em standard disk}, if $A' = \emptyset$,
	 \item a  {\em bounded disk}, if $A' = \partial A$.
\end{enumerate}

We now define morphisms between para-disks. Intuitively, such a morphism should be
a certain isotopy class of paths
  $(A_t, A'_t)_{t\in I}$ in the space of para-disks. We want such paths to satisfy the following 
dynamical requirements as $t$ increases from $0$ to $1$: 

\begin{itemize}
\item[(PD1)]  The
 components $A'_t$ can merge together and can appear {\em ex nihilo} (growing out of single 
 points)
but cannot split. 

\item[(PD2)]  A special point $x\in N$ can enter the interior of
 $A_t$ (i.e., $A_t$ can ``run it over'') only
through the complement $A_t\- A'_t$ and exit $A_t$ only through $A'_t$. 
 \end{itemize}
 
 \noindent To implement this formally, we represent  paths in the space of para-disks via 
 maps $I\times \DD\to X$. We start with formalizing the merging behavior of the components
 $A_t$ as in (PD1). 
 
 \begin{defi}\label{def:pant-cobordism}
 \begin{itemize}
 \item[(1)] Let $P\subset I\times S^1$ be a subset. For any $t\in I$ we denote
 by  $P_t = P \cap (\{t\} \times S^1)$ the slice of $P$ over $t$. We can view $P_t$
 as a subset in $S^1$. 
 
 \item[(2)] By a {\em pant cobordism} we will mean a  closed $2$-dimensional (topological) submanifold $P \subset I \times S^1$ with
		boundary such that:
		
		\begin{itemize}
		\item[(2a)] The slices $P_0, P_1\subset S^1$
		 are compact $1$-dimensional submanifolds
		with boundary, as in Definition \ref{def:paradisk}. 
		
		\item[(2b)] The inclusion $P_1 \subset P$ is a homotopy equivalence.
		
		\end{itemize} 
 \end{itemize}
  \end{defi}

  \begin{figure}[h]
  \centering

	\includegraphics[scale=.7]{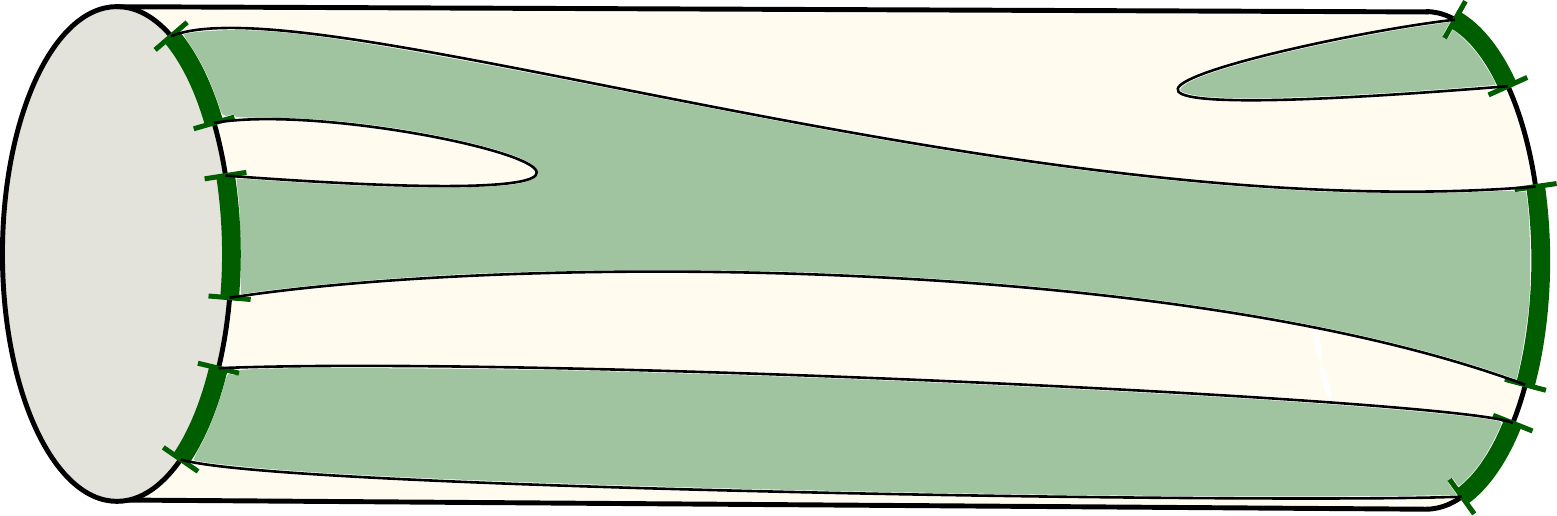}
  \caption{A pant cobordism.}\label{fig:pant-cob}
  \end{figure}
  
  An example of a pant cobordism is depicted in Fig. \ref{fig:pant-cob}. 
  
  \begin{rems}
 (a)  Strictly speaking, a pant cobordism $P$ is a manifold with corners, not just boundary, the corners being the boundary points of $P_0$ and $P_1$, as one can see 
  in Fig. \ref{fig:pant-cob}. Since we consider $P$ as a
 topological manifold, we ignore this subtlety. 
 
 (b) Intuitively, the slices $P_t\subset S^1$ correspond to the $1$-dimensional submanifolds
 $A'_t\subset A_t$ in the picture with  paths in the space of para-disks. Of course, 
 for some values of $t$ such slices may not be of the form allowed in
 Definition  \ref{def:paradisk}, in particular, they may have, as components,
  single points (which can then  disappear or grow to become intervals) Nevertheless,  the condition (2b) of Definition \ref{def:pant-cobordism} corresponds to the
 requirement (PD1) on  the paths.  In this way a pant cobordism can 
 (after time reversal $t\mapsto 1-t$) be seen as a thickened version
 of an ``exit path with deaths" from Remark \ref{rem:full-lambda}.

  \end{rems}

\begin{defi}\label{defi:undirected}
The {\em  paracyclic category $\Lambda(X,N)$ of $(X,N)$} is the category with 
objects being  para-disks $(A,A')$ for $(X,N)$. A morphism 
		\[
			f: (A_0,A'_0) \lra (A_1, A'_1)
		\]
in $\Lambda(X,N)$ consists of 
\begin{itemize}
	\item a  pant cobordism $P\subset I\times S^1$.
	
	\item a continuous map $H: I \times \DD \to X$, which we also  consider as a family of maps 
	 $H_t: \DD \to X$, $t\in I$,  such that
		\begin{enumerate}[label=(\arabic*)]
		
		\item\label{condition:isotopy} $H$ is an isotopy, i.e., each $H_t$ is an embedding, 
		
			\item\label{condition:para} for $i \in \{0,1\}$, the embedding $H_i$ induces homeomorphisms $\DD
				\cong A_i$ and $P_i \cong A'_i$,
			\item for every $t \in I$, we have $|H_t(\DD) \cap N| \le 1$,
			
				\item \label{condition:ent-ex} for every $t_0 \in I$ and $x \in H_{t_0}(P_{t_0}) \cap N$, there
		exists $\varepsilon >0$ such, for every $t_0 \le t \le t + \varepsilon$, $x \notin
		H_{t}(\DD \setminus P_{t_0})$,

		\end{enumerate}
	\item two such data $(H,P)$, $(H',P')$ define the same morphism if there exists a homeomorphism $\varphi: I
		\times \DD \to I \times \DD$ such that $\varphi|P$ induces a homeomorphism with
		$P'$, together with a homotopy $\alpha: I^2 \times \DD \to X$ with $\alpha_0 = H$
		and $\alpha_1 = H'$ such that, for every $s \in I$,
		$\alpha_s$ satisfies the above conditions.
\end{itemize}
\end{defi}

We denote by $S(X,N) \subset \Lambda(X,N)$ the full subcategory of standard disks, by
$B(X,N) \subset \Lambda(X,N)$ the full subcategory of bounded disks, and by $M(X,N)
\subset \Lambda(X,N)$ the full subcategory of Milnor disks. 
We refer to $M(X,N)$ as the {\em Milnor category} of $(X,N)$. 

\begin{rems}

 (a)  Given a  morphism $f$  with a representative $(P,H)$, we have, for any $t\in I$,  a closed
 disk $A_t=H_t(\DD)\subset X$ and a closed subset $A'_t = H_t(P_t) \subset\del A_t$.  The pair
 $(A_t, A'_t)$ depends only on $f$. For generic values of $t$, the slice $P_t$ belongs to one of the
 three types described in Definition \ref{def:paradisk} and so  $(A_t, A'_t)$ is a para-disk by the
 condition  \ref{condition:para} The condition \ref{condition:ent-ex}  corresponds to the intuitive
 requirement (PD2) on paths in the space of para-disks while (PD1) corresponds, as mentioned above,
 to the condition (2b) of Definition \ref{def:pant-cobordism} of a pant cobordism. 

 (b)  Our assumption that if $X \cong S^2$, then $|N| \ge 2$ implies that the mapping spaces which
 appear implicitly in our definition of $\Lambda(X,N)$ have contractible components, so that it is
 justified to consider it as an ordinary category (rather than an $\oo$-category). 
\end{rems}

\begin{exa}
	The category $M(\CC,\emptyset)$ of Milnor disks in $(\CC,0)$  is equivalent to the
	paracyclic category $\Lambda_{\infty}$.  This is shown similarly to the proof of Proposition
	\ref{prop:lambda-exit}. Further, the category $\Lambda(\CC,\emptyset)$ is equivalent to the
	category obtained from $\Lambda_{\infty}$ by adjoining an initial and a final objects which
	correspond to the objects
			\[
				(A, \emptyset) = \disk{0}{0} \quad \text{and} \quad 
			 (A, \partial A) = \disk{0}{s},
			\]
	respectively. 
\end{exa}

\paragraph{The Milnor category and perverse sheaves.} 

The role of the category $M(X,N)$ for our purposes is explained by the following.

\begin{prop}\label{prop:M-perv}
Let $\F\in\PS(X,N;\A)$ be a perverse sheaf on $(X,N)$ with values in a Grothendieck
abelian category $\A$. Then the correspondence $(A,A') \mapsto H^0(A,A'; \F)$
extends to a functor $h_\F: M(X,N)^\op\to \A$. 
\end{prop}

\noindent{\sl Proof:} Let $f: (A_0, A'_0)\to (A_1, A'_1)$ be a morphism between two
Milnor disks represented by a pair $(P,H)$ as in Definition \ref{defi:undirected}. 
Let $\wt N= H^{-1}(N)\subset I\times \DD$. Because of condition \ref{condition:isotopy}
of that definition, $\wt N$
is a $1$-dimensional  topological submanifold with
boundary, i.e., a disjoint union of closed curvilinear intervals
  in the cylinder $I\times D$, each of them projecting to $I$ in an injective way.
  We orient these curves following the increase of $t\in I$.

   Let $\wt N^+\subset \wt N$
  be the union of components that terminate (in the sense of the above orientation) on $P$.
  Let 
  $\wt N^-\subset \wt N$ be the union of components that terminate on $\{1\}\times \DD$.
  Thus $\wt N^+ \cup \wt N^- =\wt N$ and $\wt N^+\cap\wt N^-$ is the union of
  components that terminate on  the slice $P_1$.
  
    Further, let  $\wt \F= H^{*}(\F)$. It is a complex of sheaves on $I\times D$
  constructible with respect to the stratification given by $\wt N$. By Proposition 
\ref{prop:purity},
\[
H^0(A_i, A'_i;\F) \=\, R\Gamma(\{i\}\times \DD, P_i; \wt \F), \quad i\in \{0,1\}\subset I.
\]
Consider the diagram of restrictions
\[
R\Gamma(\{1\}\times \DD, P_1; \wt \F)\buildrel \rho_1\over\lla R\Gamma(I\times \DD, P; \wt \F) 
\buildrel \rho_0\over\lra R\Gamma(\{0\}\times \DD, P_0; \wt \F). 
\]
We claim that $\rho_1$ is a quasi-isomorphism (and therefore, by purity, 
 it reduces to an isomorphism of objects of $\A$). 
Indeed, denote
 \[
 P^+ = P\cup \wt N^+\subset I\times \DD, \,\, \DD^- = \{1\}\times \DD \cup \wt N^- \quad \subset
 \quad  I\times\DD.
 \]
  Because of the condition (2b)
  of Definition \ref{def:pant-cobordism} and the entry-exit condition
 \ref{condition:ent-ex}  of Definition \ref{defi:undirected}, 
  the inclusion of the slice $P_1\subset P^+$
  is a homotopy equivalence, and the inclusion $\{1\}\times\DD\hra \DD^-$ is a homotopy equivalence
  as well. This means that each of the two restriction morphisms
  \[
  R\Gamma(I\times\DD,P; \F) \lra R\Gamma(I\times\DD,P'; \F) \lra R\Gamma(\{1\}\times\DD,P_1; \F) 
  \]
  whose composition is $\rho_1$, is a quasi-isomorphism. 
  
  We now define  the value of the functor $h_\F$ on $f$, i.e., the
   morphism $h_\F(f): H^0(A_1, A'_1;\F)\to H^0(A_0,A'_0; \F)$ to be given by 
   $\rho_2\rho_1^{-1}$.
  The necessary verifications are left to the reader. \qed
  
  \begin{rem}
  In a similar way, utilizing the $\infty$-category of spans, one can show that the association 
  $(A,A')\mapsto R\Gamma(A,A'; \F)$ extends to an $\oo$-functor from $\Lambda(X,N)$ to $\D_\oo(\A)$, the
  $\oo$-categorical enhancement of the derived category of $\A$, see \S \ref{par:dercat}. 
  
  \end{rem}

\begin{exa}\label{exa:paraexit} 
The categories $S(X,N)$ of standard disks and $B(X,N)$ of bounded
disks are equivalent to $\Entr(X,N)$ and $\Exit(X,N)$, the
categories of entrance and exit paths of the stratified space $(X,N)$
respectively. The first equivalence  has the form
\[
		\Entr(X,N) \to S(X,N),\; x \mapsto (A_x,\emptyset)
	\]
	where $A_x \subset X$ is a disk containing $x$ such that, $A_x \cap N =
	\emptyset$ if $x \notin N$.  The second equivalence is defined
	in the dual way. 	

\end{exa}

 \paragraph{The paracyclic duality.} 

Next, we describe an identification of $\Lambda(X,N)$ with its opposite category
$\Lambda(X,N)^\op$  which
will play an important role in interpreting the Verdier duality for perverse sheaves. 
We start with the following remarks. For a closed subset $Z$ of a topological space $Y$
we denote by $\mathring{Z}$ the interior of $Z$. The next two propositions are
then clear. 

\begin{prop}
(a) For a para-disk $(A,A')\subset X$ the pair $(A,A')^*:= (A, \del A\-( \mathring{A'}))$ 
is again a para-disk.

(b) 
Let $\sigma: I\times S^1\times I\times S^1$ be the involution $(t,\theta) \mapsto (1-t, \theta)$.
For a pant cobordism $P\subset I\times S^1$  the subset $P^* = \sigma (I\times S^1) \-\mathring{ P}$
is again a pant cobordism. \qed
\end{prop}

\begin{prop}\label{prop:Lambda'}
 Let $i: \Lambda(X,N)' \subset \Lambda(X,N)$ denote the full subcategory
	consisting of those Milnor disks $(A,A')$ such $\partial A \cap N = \emptyset$. Then the
	inclusion $i$ is an equivalence of categories. \qed
\end{prop}

\begin{prop}\label{rem:selfdual} 
 We have  a  perfect duality (which we call the {\em paracyclic duality})
 \[
		\xi: \Lambda(X,N) \overset{\simeq}{\lra}  \Lambda(X,N)^\op
	\]
defined on objects by the association $(A,A')\mapsto (A,A')^*$. \qed
\end{prop}

\noindent{\sl Proof:} Using Proposition \ref{prop:Lambda'}, it suffices to define a duality on
the equivalent subcategories
		$\xi': \Lambda(X,N)' \overset{\simeq}{\lra}  \Lambda(X,N)'^{\op}$, which is given on objects
		by the desired formula $(A,A')\mapsto (A,A')^*$.

	To do this, suppose we have  a morphism $f$ represented by $(H,P)$, note that we may assume, replacing $(H,P)$ by
	an equivalent representative if needed, that special points enter
	in $I \times S^1 \setminus P$ and exit in $\mathring{P}$. Then we define $\xi(f)$ to be
	represented by $(H(1-t,-), P^*)$. It is straightforward to
	verify that this association yields a well-defined functor squaring to the identity,
	i.e., giving a perfect duality.  \qed

 \vskip .2cm

	Note, that the paracyclic duality $\xi$ interchanges the subcategories $S(X,N)$ and
	$B(X,N)$, identifying them as opposite to one another, and  restricts to a self-duality of
	$M(X,N)$.

\subsection{The directed paracyclic category and its localization} 

Let $(X,N)$ be as before. In this section we  exhibit  $\Lambda(X,N)$  as a localization of another
category $\Lambdad(X,N)$  which we call the directed paracyclic category. This latter category turns
out to be more suitable for the use of Kan extensions. 

\begin{defi}\label{defi:directed}
We define the {\em directed paracyclic category} $\Lambdad(X,N)$ exactly as in Definition
\ref{defi:directed} but replacing condition \ref{condition:ent-ex} by the following:
\begin{enumerate}

\item[(Ent)] \label{condition:entrance} For every $x \in N$, we have
				\begin{enumerate}[label=(Ent\arabic*)]
					\item if $x \in A_{t_0} =  H_{t_0}(\DD)$ for $t_0 \in I$ then, for all $t \ge
						t_0$, we have $x \in A_t$,
					\item if $x \in A'_{t_0}= H_{t_0}(P_{t_0})$ for $t_0 \in I$ then, for all $t \ge
						t_0$, we have $x \in A'_t$.
				\end{enumerate}
	
\end{enumerate}

A morphism $f: (A,A') \to (B,B')$ in $\Lambdad(X,N)$ is called a {\em weak equivalence} if either
\begin{enumerate}[label=(\roman*)]
	\item $f$ is an isomorphism, or
	\item\label{item-weak} $f$ can be represented by a pair $(P,H)$ such that $H_0^{-1}(A')
		\subset P$ is a homotopy equivalence and $H^{-1}(N) \subset P$.
\end{enumerate}
We denote $W \subset \Mor(\Lambdad(X,N))$ the set of weak equivalences. 
\end{defi}

\begin{rems}\label{rem:Ent1-2}
(a) 
The condition (Ent) is a $2$-step version of the entrance path condition: if a special point $x$ enters $A_{t_0}$,
then it stays in all the $A_t$ for all $t\geq t_0$, and similarly for $A'_{t_0}$. 

(b) The condition \ref{item-weak} in the definition of a weak equivalence means that a special point $x$
is allowed to enter $A'_{t_0}\subset A_{t_0}$ from the outside of $A_{t_0}$ and stay there for all $t\geq t_0$. 
\end{rems}

We also denote by $\Sd(X,N), \Bd(X,N), \Md(X,N) \subset \Lambdad(X,N)$ the full subcategories of standard disks, bounded disks, and   Milnor disks respectively.

\begin{prop}\label{prop:localization} The natural morphism
	\[
		\pi: \Lambdad(X,N) \lra \Lambda(X,N) 
	\]
	exhibits $\Lambda(X,N)$ as a localization of $\Lambdad(X,N)$ along $W$.
	\end{prop}

Here by ``localization" we mean  $\Lambdad(X,N)[W^{-1}]$, the Gabriel-Zisman localization  in the
sense of ordinary categories \cite{gabriel-zisman}.  In fact, one can prove  stronger statements,
identifying $\Lambda(X,N)$ with the $\oo$-categorical localization or with the Dwyer-Kan simplicial
localization \cite{dwyerkan} of $\Lambdad(X,N)$ with respect to  $W$. This can be done by adapting
our proof below by  using a hammock-type model for the Dwyer-Kan  localization. We will not need
this generalization for our purposes except for a very particular case in  Lemma \ref{lem:bloc}
below, which is easily proved directly.

	\begin{proof}
	 Recall that in $\Lambda(X,N)$ a special point $x$ is allowed to exit $A_{t_0}$
	through $A'_{t_0}$. This process is inverse to entering $A_{t_0}$
	through $A'_{t_0}$  from the outside which is, according to Remark \ref{rem:Ent1-2}(b),
	a general form of a weak equivalence (apart from an isomorphism). 
	Indeed, the composite process (entering $A_{t_0}$ through $A'_{t_0}$ from the outside and
	then bouncing back to the original position) is connected to the identity 
	by a homotopy $\alpha$ as in Definition \ref{defi:undirected}.

	Therefore
	  the functor $\pi$ inverts weak equivalences and we obtain an induced functor 
	  $\overline{\pi}:
	\Lambdad(X,N)[W^{-1}] \to \Lambda(X,N)$. 
	We claim that $\ol\pi$ is an equivalence. 
	To this end, we  study a typical Hom-set
	\be\label{eq:Hom-lambdad}
	\Hom_{\Lambdad(X,N)[W^{-1}]}(A,A'), (C,C')). 
	\ee
  By definition, cf.  \cite{gabriel-zisman} \S I.1,  an  element 
  of this set is an equivalence class of zig-zags  
   \begin{equation}\label{eq:rep}
	(A, A')=	  (A_1, A'_1) \buildrel {w_1}\over\leftarrow  (B_1, B'_1)\buildrel {f_1}\over\to  (A_2, A'_2) 
		 \buildrel {w_2}\over\leftarrow  \cdots \buildrel {f_{n-1}} \over\to 
		 (A_n, A'_n) = (C,C')
		 \end{equation}
	of arbitrary length, with $w_i\in W$. The equivalence relation on the set of such zig-zags
	is generated by two elementary moves:
	\begin{itemize}
	\item[(M1)] For any factorization
	\[
	\xymatrix{
	(A_i, A'_i) & \ar[l]_{w_i} (B_i, B'_i) \ar[r]^{\hskip -0.5cm f_{i+1}}& (A_{i+1}, A'_{i+1})
	\\
	(B_{i-1}, B'_{i-1})\ar[u]^{f_i} \ar[ur]_{g}&&
	}
	\]
	we can replace the fragment $\buildrel f_i\over\to\buildrel w_i\over\leftarrow \buildrel f_{i+1}\over\to$
	with $\buildrel f_{i+1}g\over\to$. 
	
	\item[(M2)] For any factorization
	\[
	\xymatrix{
	(B_{i-1}, B'_{i-1}) \ar[r]^{\hskip 0.5cm f_{i-1}}& (A_i, A'_i)\ar[dr]_{h}
	 & \ar[l]_{w_i} (B_i, B'_i) 
	\ar[d]^{f_i} 
	\\
	&& (A_{i+1}, A'_{i_1})
	}
	\]
	we can replace the fragment $\buildrel f_{i-1}\over\to\buildrel w_i\over\leftarrow \buildrel f_i\over\to$
	with $\buildrel hf_{i-1}\over\to$. 
	\end{itemize}
	These two moves imply the {\em hammock move}, which 
	is at the basis of
	  Dwyer-Kan localization theory  \cite{dwyerkan} (except that we don't assume that the
	  vertical morphisms are weak equvialences):
	\begin{itemize}
	\item[(H)] Any two zig-zags  connected by a {\em hammock}, 
	 i.e., by a commutative diagram 	\[
	\xymatrix@C=1em@R=1em{
	& (B_1, B_1') \ar[dl]_{w_1} \ar[dd]
	 \ar[r]^{f_1}& (A_2, A'_2) \ar[dd] & \ar[l]_{w_2} (B_2, B'_2)\ar[dd]
	\ar[r]^{\hskip 0.5cm f_2}& \cdots & \ar[l]_{\hskip -0.8cm w_{n-1}}
	(B_{n-1}, B'_{n-1}) \ar[dr]^{f_n} \ar[dd] &
	\\
	(A,A') &&&&&&(C,C')
	\\
	& (\wt B_1, \wt B_1') \ar[ul]^{\wt w_1}
	 \ar[r]^{\wt f_1}& (\wt A_2,\wt A'_2)& \ar[l]_{\wt w_2} (\wt B_2, \wt B'_2)
	\ar[r]^{\hskip  0.5cm \wt f_2}& \cdots & \ar[l]_{\hskip -0.8cm \wt w_{n-1}}
	(\wt B_{n-1}, \wt B'_{n-1}) \ar[ur]_{\wt f_n}&
	}
	\]
   	are equivalent.
	\end{itemize}

	We now compare this with $\Hom_{\Lambda(X,N)}((A,A'), (C,C'))$. 
	An element  $f$ of this latter set is an equivalence class of pairs $(P,H)$
	as in Definition \ref{defi:undirected}. As usual, we write $A_t= H_t(\DD)$,
	$A'_t=H_t(P_t)$. Without loss of generality, we can assume that:
	
	\begin{itemize} 
	\item $P$ is smooth as a manifold with corners, i.e.,  
	 the part of $\del P$ lying over the open interval $(0,1)\subset I$
	is smooth.
	
	\item
	 The projection of this part of $\del P$ to $(0,1)$ is a Morse function. 
	This implies that for all but finitely many   values of $t$
	(which we call {\em critical values})
	the slice $P_t$ has one of the three forms listed in Definition \ref{def:paradisk}
	and therefore $(A_t, A'_t)$ is a para-disk. 
	
	\item The moments  $t_1 <\cdots < t_n$, $t_i\in I$,  of exit of special points
	 $x\in N$ out of
	$A_t$ (happening through $A'_t$) are non-critical.   
	\end{itemize}
	
	Let $t'_i >t_i$, $i=1,\cdots, n$,  be sufficently close.
	 As explained in the beginning of the proof,
	the restriction of $(P,H)$ to the preimage of each interval $[t_i, t'_i]$
	can be seen as an inverse of a weak equivalence in $\Lambdad(X,N)$. 
	while the restriction to each interval in the complement of the union of
	the $[t_i, t'_i]$, is a morphism in $\Lambdad(X,N)$. Therefore we
	can associate to $(H,P)$ a zig-zag  \eqref {eq:rep}. 
	
	We claim that different choices of $(H,P)$ representing the same morphism
	$f$, give rise to equivalent zig-zags. Any two such different choices are, by 
	Definition \ref{defi:undirected}, related by a reparemetrization
	$\varphi: I\times \DD \to I \times \DD$ and a homotopy $\alpha: I^2 \times \DD \to X$. By choosing $\alpha$ generic enough, we see that   any two choices are
	connected by a sequence of the following moves and their inverses:
	\begin{enumerate} 
		\item[(M$'$1)]  replacing a representative $(P,H)$ with a  representative
		 $(P, \wt H)$ which, locally around $t \in I$,   avoids the special point contained in $A_t'$: 
		 \begin{equation}\label{eq:m1}
			 \begin{tikzcd}
	\def\svgwidth{5cm}
	\raisebox{-.5\height}{
\begingroup%
  \makeatletter%
  \providecommand\color[2][]{%
    \errmessage{(Inkscape) Color is used for the text in Inkscape, but the package 'color.sty' is not loaded}%
    \renewcommand\color[2][]{}%
  }%
  \providecommand\transparent[1]{%
    \errmessage{(Inkscape) Transparency is used (non-zero) for the text in Inkscape, but the package 'transparent.sty' is not loaded}%
    \renewcommand\transparent[1]{}%
  }%
  \providecommand\rotatebox[2]{#2}%
  \newcommand*\fsize{\dimexpr\f@size pt\relax}%
  \newcommand*\lineheight[1]{\fontsize{\fsize}{#1\fsize}\selectfont}%
  \ifx\svgwidth\undefined%
    \setlength{\unitlength}{279.41500824bp}%
    \ifx\svgscale\undefined%
      \relax%
    \else%
      \setlength{\unitlength}{\unitlength * \real{\svgscale}}%
    \fi%
  \else%
    \setlength{\unitlength}{\svgwidth}%
  \fi%
  \global\let\svgwidth\undefined%
  \global\let\svgscale\undefined%
  \makeatother%
  \begin{picture}(1,0.70409073)%
    \lineheight{1}%
    \setlength\tabcolsep{0pt}%
    \put(0,0){\includegraphics[width=\unitlength,page=1]{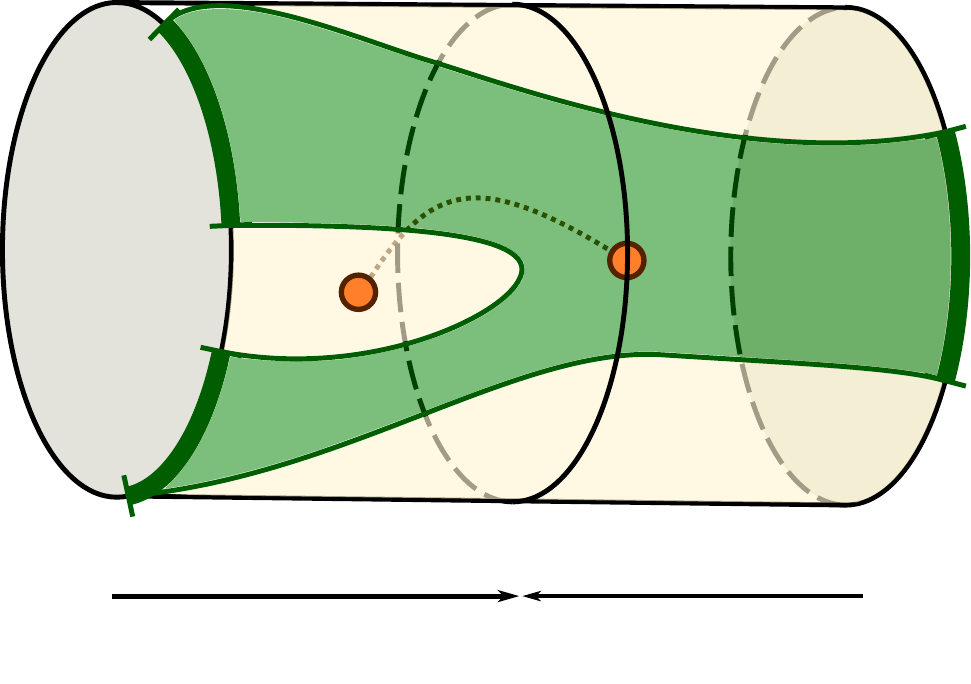}}%
    \put(0.30484604,0.02532146){\color[rgb]{0,0,0}\makebox(0,0)[lt]{\lineheight{1.25}\smash{\begin{tabular}[t]{l}$f_i$\end{tabular}}}}%
    \put(0.68063105,0.02723869){\color[rgb]{0,0,0}\makebox(0,0)[lt]{\lineheight{1.25}\smash{\begin{tabular}[t]{l}$w_i$\end{tabular}}}}%
  \end{picture}%
\endgroup%
} & \leadsto & 
	\def\svgwidth{5cm}
	\raisebox{-.5\height}{
\begingroup%
  \makeatletter%
  \providecommand\color[2][]{%
    \errmessage{(Inkscape) Color is used for the text in Inkscape, but the package 'color.sty' is not loaded}%
    \renewcommand\color[2][]{}%
  }%
  \providecommand\transparent[1]{%
    \errmessage{(Inkscape) Transparency is used (non-zero) for the text in Inkscape, but the package 'transparent.sty' is not loaded}%
    \renewcommand\transparent[1]{}%
  }%
  \providecommand\rotatebox[2]{#2}%
  \newcommand*\fsize{\dimexpr\f@size pt\relax}%
  \newcommand*\lineheight[1]{\fontsize{\fsize}{#1\fsize}\selectfont}%
  \ifx\svgwidth\undefined%
    \setlength{\unitlength}{279.41500824bp}%
    \ifx\svgscale\undefined%
      \relax%
    \else%
      \setlength{\unitlength}{\unitlength * \real{\svgscale}}%
    \fi%
  \else%
    \setlength{\unitlength}{\svgwidth}%
  \fi%
  \global\let\svgwidth\undefined%
  \global\let\svgscale\undefined%
  \makeatother%
  \begin{picture}(1,0.69502866)%
    \lineheight{1}%
    \setlength\tabcolsep{0pt}%
    \put(0,0){\includegraphics[width=\unitlength,page=1]{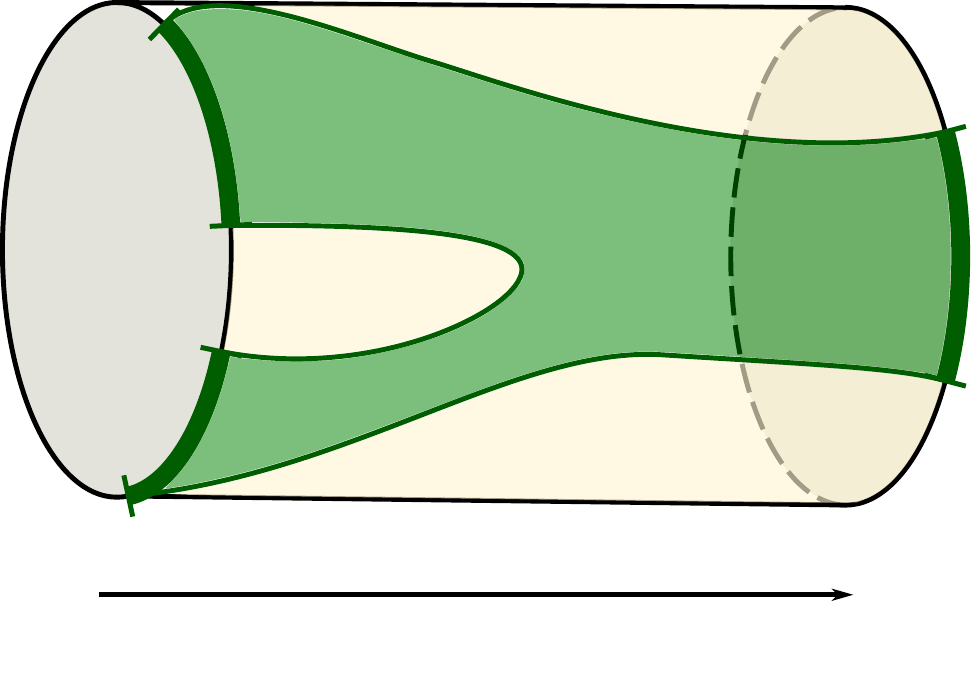}}%
    \put(0.47992317,0.02233321){\color[rgb]{0,0,0}\makebox(0,0)[lt]{\lineheight{1.25}\smash{\begin{tabular}[t]{l}$g$\end{tabular}}}}%
  \end{picture}%
\endgroup%
}
			\end{tikzcd}
		\end{equation}

%
%
%
%
%
%
%
%
 Denote by $\buildrel w_i\over\lla$ the slice of $(P,H)$ 
from the moment of exit of $x$ until shortly afterwards and by $\buildrel f_i\over\lra$ the slice from shortly before exit to the moment of exit,
see \eqref{eq:m1}. 
We see that we have three morphisms $g, w_i, f_i$  $\Lambdad(X,N)$
  and a factorization $w_i g=f_i$ in $\Lambdad(X,N)$ represented by an appropriate
  homotopy $\alpha$. Therefore the move (M$'$1) yields two zig-zags connected by the move (M1).

		\item[(M$'$2)] replacing a representative $(P,H)$ with a representative
		 $(\wt P, H)$
			where $\wt P$ is obtained by deforming $P$ in a suitable way
			locally around one of the  exit moments  $t_i$,  so that two intervals in $A'_{t_i}$
			are replaced by one: 
		 \begin{equation}\label{eq:m2}
			 \begin{tikzcd}
	\def\svgwidth{5cm}
	\raisebox{-.52\height}{
\begingroup%
  \makeatletter%
  \providecommand\color[2][]{%
    \errmessage{(Inkscape) Color is used for the text in Inkscape, but the package 'color.sty' is not loaded}%
    \renewcommand\color[2][]{}%
  }%
  \providecommand\transparent[1]{%
    \errmessage{(Inkscape) Transparency is used (non-zero) for the text in Inkscape, but the package 'transparent.sty' is not loaded}%
    \renewcommand\transparent[1]{}%
  }%
  \providecommand\rotatebox[2]{#2}%
  \newcommand*\fsize{\dimexpr\f@size pt\relax}%
  \newcommand*\lineheight[1]{\fontsize{\fsize}{#1\fsize}\selectfont}%
  \ifx\svgwidth\undefined%
    \setlength{\unitlength}{282.61110179bp}%
    \ifx\svgscale\undefined%
      \relax%
    \else%
      \setlength{\unitlength}{\unitlength * \real{\svgscale}}%
    \fi%
  \else%
    \setlength{\unitlength}{\svgwidth}%
  \fi%
  \global\let\svgwidth\undefined%
  \global\let\svgscale\undefined%
  \makeatother%
  \begin{picture}(1,0.68964043)%
    \lineheight{1}%
    \setlength\tabcolsep{0pt}%
    \put(0,0){\includegraphics[width=\unitlength,page=1]{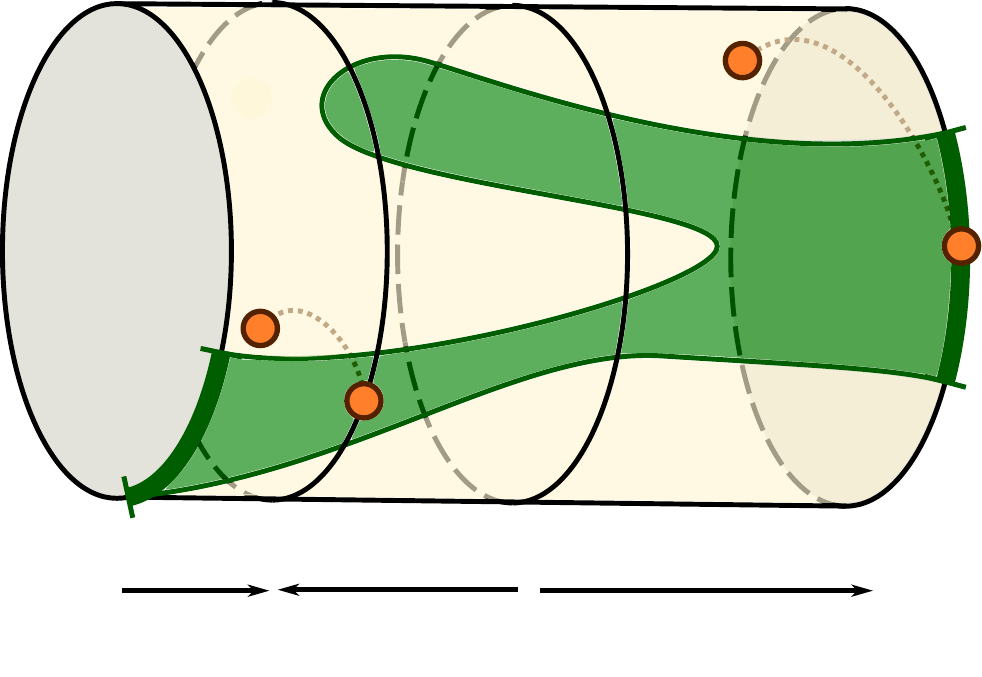}}%
    \put(0.18018443,0.0052336){\color[rgb]{0,0,0}\makebox(0,0)[lt]{\lineheight{1.25}\smash{\begin{tabular}[t]{l}$f_i$\end{tabular}}}}%
    \put(0.69294108,0.00333801){\color[rgb]{0,0,0}\makebox(0,0)[lt]{\lineheight{1.25}\smash{\begin{tabular}[t]{l}$f_{i+1}$\end{tabular}}}}%
    \put(0.39912485,0.00618143){\color[rgb]{0,0,0}\makebox(0,0)[lt]{\lineheight{1.25}\smash{\begin{tabular}[t]{l}$w_i$\end{tabular}}}}%
  \end{picture}%
\endgroup%
} & \leadsto & 
	\def\svgwidth{5cm}
	\raisebox{-.5\height}{
\begingroup%
  \makeatletter%
  \providecommand\color[2][]{%
    \errmessage{(Inkscape) Color is used for the text in Inkscape, but the package 'color.sty' is not loaded}%
    \renewcommand\color[2][]{}%
  }%
  \providecommand\transparent[1]{%
    \errmessage{(Inkscape) Transparency is used (non-zero) for the text in Inkscape, but the package 'transparent.sty' is not loaded}%
    \renewcommand\transparent[1]{}%
  }%
  \providecommand\rotatebox[2]{#2}%
  \newcommand*\fsize{\dimexpr\f@size pt\relax}%
  \newcommand*\lineheight[1]{\fontsize{\fsize}{#1\fsize}\selectfont}%
  \ifx\svgwidth\undefined%
    \setlength{\unitlength}{282.61110179bp}%
    \ifx\svgscale\undefined%
      \relax%
    \else%
      \setlength{\unitlength}{\unitlength * \real{\svgscale}}%
    \fi%
  \else%
    \setlength{\unitlength}{\svgwidth}%
  \fi%
  \global\let\svgwidth\undefined%
  \global\let\svgscale\undefined%
  \makeatother%
  \begin{picture}(1,0.69493538)%
    \lineheight{1}%
    \setlength\tabcolsep{0pt}%
    \put(0,0){\includegraphics[width=\unitlength,page=1]{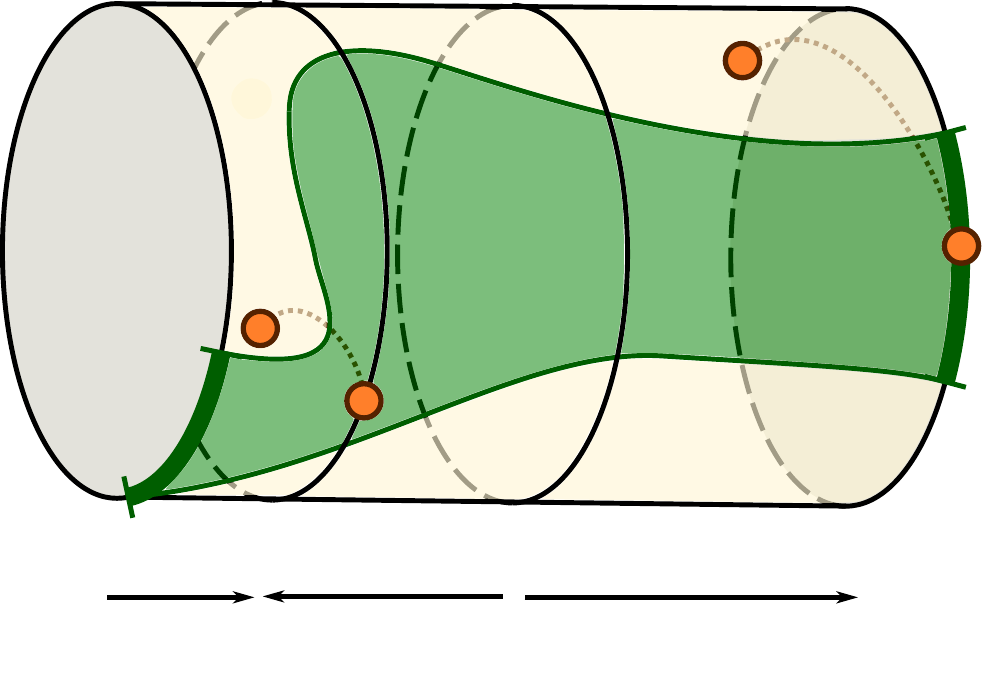}}%
    \put(0.38480428,0.00942904){\color[rgb]{0,0,0}\makebox(0,0)[lt]{\lineheight{1.25}\smash{\begin{tabular}[t]{l}$\tilde{w}_i$\end{tabular}}}}%
    \put(0.14615067,0.00333801){\color[rgb]{0,0,0}\makebox(0,0)[lt]{\lineheight{1.25}\smash{\begin{tabular}[t]{l}$\tilde{f_i}$\end{tabular}}}}%
    \put(0.67901311,0.00389327){\color[rgb]{0,0,0}\makebox(0,0)[lt]{\lineheight{1.25}\smash{\begin{tabular}[t]{l}$\tilde{f}_{i+1}$\end{tabular}}}}%
  \end{picture}%
\endgroup%
}
			\end{tikzcd}
		\end{equation}
	\end{enumerate}
	Making four slices of each the two cobordisms as in \eqref{eq:m2},
	we get two zig-zags connected by a hammock: 
		\[
	\xymatrix@R=1em{
	& (A_i, A'_i) \ar[dd] & \ar[l]_{w_i} (B_i, B'_i)\ar[dd] \ar[dr]^{f_{i+1}} & 
	\\
	(B_{i-1}, B'_{i-1}) \ar[ur]^{f_i} \ar[dr]_{\wt f_i} 
	&&& (A_{i+1}, A'_{i+1})
	\\
	& (\wt A_i, \wt A'_i) & \ar[l]^{\wt w_i} (\wt B_i, \wt B'_i)\ar[ur]_{\wt f_{i+1}}
	 & 
	}
	\]
	so they are equivalent by the hammock move. Therefore the entire zigzags corresponding to
	$(P,H)$ and $(\wt P, H)$ are equivalent as well. 
	
	In this way we define a functor $\Lambdad(X,N)[W^{-1}]\to \Lambda(X,N)$
	which is easily seen to be quasi-inverse to $\ol\pi$. 
		
\end{proof}

\begin{cor}\label{cor:milnorlocalization} The functor $\pi$ from Proposition \ref{prop:localization}
	induces an equivalence $\Sd (X,N) \simeq S(X,N)$ and localizations $\Md(X,N) \to M(X,N)$, 
	$\Bd(X,N) \to B(X,N)$. 
\end{cor}

\subsection{Constructible sheaves with values in $\oo$-categories}
\label{subsec:constructible}

Let $(X,N)$ be a stratified surface, let $\Op(X)$ denote the poset of open subsets of $X$, and let
$\D$ be an $\infty$-category.
The following  is an $\oo$-categorical analog of the discussion for abelian categories in
\S \ref{par:sheaf-abel}. 

\begin{lem}\label{lem:descent} Given a functor $\F: \N(\Op(X))^{\op} \to \D$, an open subset $U
	\subset X$, and an open cover $\U = \{U_i\}_{i \in I}$ of $U$, the following conditions are
	equivalent:
	\begin{enumerate}[label=(\roman *)]
		\item Denote by $\Op(X)/\U$ the poset of open subsets $V \subset X$ such that $V
			\subset U_i$ for some $i \in I$. Then the canonical map
			\[
				\F(U) \lra \lim \F|(\Op(X)/\U)^{\op}
			\]
			is an equivalence in $\D$.
		\item Denote by $\P(I)$ the poset of nonempty finite subsets of $I$ and consider the
			inclusion $\P(I) \subset \Op(X)^{\op}, J \mapsto \cap_{j \in J} U_j$. Then
			the canonical map
			\[
				\F(U) \lra \lim \F|\P(I)
			\]
			is an equivalence in $\D$.
	\end{enumerate}
\end{lem} 
\begin{proof} The inclusion $\P(I)^{\op} \subset \Op(X)/\U$ is $\infty$-cofinal.
\end{proof}

A {\em $\D$-valued sheaf} on $X$ is a functor 
\[
	\F: \N(\Op(X))^{\op} \to \D
\]
such that, for every open $U \subset X$ and every open cover $\U$ of $U$, the equivalent conditions
of Lemma \ref{lem:descent} hold. We denote by 
\[
	\Sh(X;\D) \subset \Fun(\N(\Op(X))^{\op},\D)
\]
the full subcategory spanned by the $\D$-valued sheaves on $X$.

Let  $(\Disko(X,N, \leq)$ be the poset of standard pairs $(U,\emptyset)$ ordered by inclusion. 
We will consider it as a category. A morphism in $(\Disko(X,N, \leq)$ (i.e., an inclusion
$U_1\subset U_2$ of standard disks) will be called a {\em weak equivalence}, if
$|N\cap U_1|=|N\cap U_2|$. We denote by $W$ the set of weak equivalences. 
The map
\[
	i: \Disko(X,N) \subset \Op(X),\; (U,\emptyset) \mapsto U
\]
identifies  $\Disko(X,N)$  with a full subposet of $\Op(X)$.  A
sheaf $\F$ in $\Sh(X;\D)$ is called {\em constructible} if its restriction $\F|\Disko(X,N)^{\op}$
maps weak equivalences to equivalences in $\D$. We denote the full subcategory of $\Sh(X;\D)$
spanned by the constructible sheaves by $\Sh(X,N ;\D)$. 

\begin{rem}\label{rem:derived} Let $\A$ be an abelian category with enough injectives and let
	$\D = \Dp(\A)$ denote the corresponding (left-bounded) derived $\infty$-category as defined in
	\cite[1.3.2.8]{HA}. We equip the stable $\infty$-category $\Sh(X;\D)$ with the $t$-structure
	$(\Sh(X;\D_{\ge 0}), \Sh(X;\D_{\le 0}))$ where the $t$-structure on $\D$ is the one from
	\cite[1.3.2.19]{HA}. The heart of this $t$-structure is equivalent to $\Sh(X,\A)$. 
	Then, using the recognition principle for derived $\infty$-categories (\cite[1.3.3.7]{HA}),
	we obtain an equivalence of $\infty$-categories
	\[
		\Dp(\Sh(X;\A)) \overset{\simeq}{\lra} \Sh(X;\Dp(\A)).
	\]
	In particular, the $\infty$-category $\Sh(X;\D(\A))$ is really an enhancement of the
	ordinary derived category of complexes of $\A$-valued sheaves. Further, this equivalence
	identifies our constructible category $\Sh(X,N; \Dp(\A))$ with the more traditional
	derived constructible category, defined as the full subcategory of $\Dp(\Sh(X;\A))$ spanned by
	objects with constructible cohomology sheaves. 
\end{rem}

We denote by $\Disko(X,N)[W^{-1}]_\oo$ the $\infty$-categorical localization of $\Disko(X,N)$ along the
weak equivalences $W$. In particular, we may identify 
\[
	\Fun(\Disko(X,N)[W^{-1}]_\oo^{\op}, \D) \,  \subset \,  \Fun(\Disko(X,N)^{\op}, \D)
\]
with the full subcategory spanned by those functors that map weak
equivalences in $\Disko(X,N)$ to equivalences in $\D$. 

\begin{prop} \label{prop:constructible} The functor
	\[
		i^*: \Sh(X,N;\D) \lra \Fun(\Disko(X,N)[W^{-1}]_\oo^{\op}, \D)
	\]
	is an equivalence of $\infty$-categories. 
\end{prop}
\begin{proof}
	Let $\F: \Op(X)^{\op} \to \D$ be a presheaf on $X$ such that $\F|\Disko(X,N)^{\op}$ sends weak
	equivalences to equivalences in $\D$. We claim that the following conditions
	are equivalent:
	\begin{enumerate}
		\item $\F$ is a sheaf.
		\item $\F$ is a right Kan extension of $\F|\Disko(X,N)^{\op}$. 
	\end{enumerate}
	The claim immediately implies the statement of the Proposition. The reason why this
	statement is not completely formal is that in condition (2), we do not assume that the
	restriction of $\F$ to $\Disko(X,N)^{\op}$ satisfies a descent condition. We rather need
	to convince ourselves that this is automatic due to the assumption that $\F$ is
	constructible. 

	(1) $\Rightarrow$ (2): Suppose that $\F$ is a sheaf. We need to show that, for every open $U
	\subset X$, $\F(U)$ is the limit of the diagram $\F|(\Disko(X,N)/U)^{\op}$. We interpret the
	set $\U = \Disko(X,N)/U$ as an open cover of $U$ so that this statement follows immediately
	from the hypothesis that $\F$ is a sheaf. 

	(2) $\Rightarrow$ (1): Suppose that $\F$ is a right Kan extension of $\F|\Disko(X,N)^{\op}$.
	Let $U \subset X$ be an open subset, and let $\U \subset \Op(X)$ be an open cover of $U$. Let
	$\Op(X)/\U$  (resp. $\Disko(X,N)/\U$) denote the subposet of $\U$ consisting of those opens
	$V$ (resp. $V \in \Disko(X,N)$) such that $V \subset U_i$ for some $U_i \in \U$. We need
	to show that the map
	\[
		\F(U) \to \lim \F|(\Op(X)/\U)^{\op}
	\]
	is an equivalence.
	Since $\F|(\Op(X)/\U)^{\op}$ is a right Kan extension of $\F|(\Disko(X,N)/\U)^{\op}$, it
	suffices to show that the composite 
	\[
		\F(U) \to \lim \F|(\Op(X)/\U)^{\op} \to \lim \F|(\Disko(X,N)/\U)^{\op}
	\]
	is an equivalence. Via the pointwise formula for $\F(U)$, we deduce that it suffices to show
	that $\F|(\Disko(X,N)/U)^{\op}$ is a right Kan extension along $i^{\op}$ where
	\[
		i: \Disko(X,N)/\U \subset \Disko(X,N)/U. 
	\]
	To this end, let $D \in \Disko(X,N)$ with $D \subset U$. We need to show that $\F(D)$ is a
	limit of $\F|(i/D)^{\op}$. Denote $\E = i/D$ and introduce the category $\L$ with
	\begin{itemize}
		\item the set of objects of $\L$ is the set of objects of $\E$,
		\item a morphism between objects $V$ and $V'$ of $\L$ is a homotopy class of paths
			$\gamma$ in $\Emb(V,D)$ such that $\gamma(0)$ is the embedding $V \subset
			D$, $\gamma(1)$ is a homeomorphism $V \cong V'$, and, if $\gamma(t)(V)$
			contains the special point for some $t$, then $\gamma(t')(V)$ contains the
			special point for all $t' \ge t$.
	\end{itemize}
	Denote by $\pi: \E \to \L$ the natural functor. We will show that $\pi$ is an
	$\infty$-cofinal localization at the set of weak equivalences in $\E$.\\
	
	{\bf Step 1.} $\pi$ is $\infty$-cofinal. To show this claim, we need to show that, for every $V \in \L$, the
	category $V/\pi$ is weakly contractible. To this end, we consider the space $E = P'\Emb(V,D)$ of paths
	$\gamma$ in $\Emb(V,D)$, that satisfy: if $\alpha(t)(V)$ contains the special point, then
	$\alpha(t')(V)$ contains the special point for all $t' \ge t$. We then deduce that $V/\pi$
	is weakly contractible, by applying Lemma \ref{lem:vankampen} to the functor
	\[
		V/\pi \lra \U(E), (V \overset{[\gamma]}{\to} V') \mapsto U([\gamma]),
	\]
	where $U([\gamma])$ is the open subset of $E$ consisting of paths that
	end in an embedding $V \hra V'$ and whose associated homotopy class, obtained by composing
	with any path of embeddings from $V \hra V'$ to $V \cong V'$, agrees with $\gamma$.\\

	{\bf Step 2.} For $V \in \L$, denote by $j: (V/\pi)^{\cong} \subset V/\pi$ the inclusion of
	the full subcategory spanned by the isomorpisms in $\L$. By a similar argument as in Step 1,
	using Lemma \ref{lem:vankampen}, it follows that $j$ is $\infty$-coinitial. It is then
	that, for every $\infty$-category $\D$ with limits, the unit $\id \to \pi_*\pi^*$
	is an equivalence, and the counit $\pi^*\pi_* \to \id$ is an equivalence on those functors
	$\E \to \D$ that map weak equivalences to equivalences.  This implies that $\pi^*$ is fully
	faithful with essential image consisting precisely of these latter functors $\E \to \D$.\\

	Now, equipped with this statement, we show that $\F(D)$ is a limit of $\F|(i/D)^{\op}$.
	Namely, by assumption, $\F|(i/D)^{\op}$ maps weak equivalences to equivalences so that it
	is equivalent to $(\pi^{\op})^*\G$ for some functor $\G: \L^{\op} \to \D$. Since $\pi$ is
	$\infty$-cofinal, we may compute the limit of $\F|(i/D)^{\op}$ as the limit of $\G$. But
	now the category $\L^{\op}$ has an initial object given by a disk $D' \subset D$ so 
	that, if $D$ contains a special point, then $D'$ also contains the special point. In any
	case, we have that $D' \subset D$ is a weak equivalence. Therefore, we obtain the desired
	equivalence $\F(D) \simeq \lim \G \simeq \F|(i/D)^{\op}$.
\end{proof}

In our treatment of Milnor sheaves it will be important to have a good control on the boundary of
disks which is why we now switch from open disks to closed disks. Let $\Diskoc(X,N)$ denote the
poset of all open and closed disks in $X$ containing at most one special point. We denote by
$\Disko(X,N) \subset \Diskoc(X,N)$ and $\Diskc(X,N) \subset \Diskoc(X,N)$ the subsets of open and
closed disks, respectively. The poset $\Diskoc(X,N)$ comes equipped with a set of weak equivalences
$W$ given by those inclusions of disks that preserve the number of special points. 

\begin{prop}
	Let $(X,N)$ be a stratified surface and let $\D$ be an $\infty$-category. There are
	equivalences of $\oo$-categories
	\[
		 \Fun(\Diskc(X,N)[W^{-1}]_\oo, \D) \longleftarrow  \Fun(\Diskoc(X,N)[W^{-1}]_\oo, \D)
			 \longrightarrow \Fun(\Disko(X,N)[W^{-1}]_\oo, \D).
		\]
\end{prop}
\begin{proof}
	We claim that the subcategory 
	\[
		\Fun(\Diskoc(X,N)[W^{-1}]_\oo, \D) \, \subset \, \Fun(\Diskoc(X,N), \D)
	\] 
	can be identified with the subcategory of left Kan extensions along $i: \Diskc(X,N) \subset
	\Diskoc(X,N)$ and the subcategory of right Kan extensions along $j: \Disko(X,N) \subset
	\Diskoc(X,N)$. To verify the first claim, suppose that $\F:\Diskoc(X,N) \to \D$ is functor
	sending weak equivalences in $\Diskc(X,N)$ to equivalences in $\D$. The pointwise Kan
	extension formula at $U \in \Disko(X,N)$ exhibits $\F(U)$ as the colimit over $i/U$. If $U$
	contains a special point, then we may replace the category $i/U$ by the cofinal subcategory
	$(i/U)'$ consisting of those closed disks that contain the special point (otherwise, we set
	$(i/U)' = i/U$). The category $(i/U)'$ is filtered and hence contractible and the diagram
	$\F|(i/U)'$ consists of equivalences. Hence, by Lemma \ref{lem:conlim}, $\F$ is a left Kan
	extension of $\F|\Disko(X,N)$ if and only if, for every $A \in (i/U)'$, the map $F(A) \to
	\F(U)$ is an equivalence. It is now an immediate consequence of the two-out-of-three property
	of equivalences that $\F$ is a left Kan extension of $\F|\Disko(X,N)$ if and only if $\F$ sends
	all weak equivalences to equivalences in $\D$. The second claim regarding right Kan
	extensions along $j$ follows from an essentially identical argument. 
\end{proof}

Finally, we would like to provide an explicit description of the localization 
$\Diskc(X,N)[W^{-1}]_\oo$ which
will provide the starting point for our discussion of Milnor disks.  

\begin{prop} The functor 
	\begin{equation}\label{eq:disklocalization}
			\pi: \Diskc(X,N) \lra S(X,N),\; A \mapsto (A,\emptyset)
	\end{equation}
	exhibits the ordinary category $S(X,N)$ as  an $\infty$-categorical localization along the weak equivalences of
	$\Diskc(X,N)$, i.e., identifies it with  $\Diskc(X,N)[W^{-1}]_\oo$ as an $\oo$-category.
	 In particular, for every $\infty$-category $\D$, the functor
	\[
		\pi^*: \Fun(S(X,N), \D) \lra \Fun(\Diskc(X,N), \D)
	\]
	is fully faithful with essential image consisting of those functors that send weak
	equivalences in $\Diskc(X,N)$ to equivalences in $\D$.
\end{prop}
\begin{proof}
	Let $(A,\emptyset) \in S(X,N)$. Suppose first that $A$ does not contain a
	special point. Then we have 
	\[
		(\pi/(A,\emptyset))^{\simeq} = \pi/(A,\emptyset).
	\]
	Further, we claim that $\pi/(A,\emptyset)$ is contractible. To this
	end, consider the topological space $P$ of continuous paths $[0,1] \to \mathring{X}
	\setminus N$ ending in $\mathring{A}$. To an object $(B, \alpha: (B,\emptyset) \to
	(A,\emptyset))$, we associate the open subset of $P$ consisting of those paths that
	start in $\mathring{B}$ and lie in the same homotopy class as the class of paths 
	that arises from the isotopy comprising $\alpha$. This association defines a functor
	\[
		\pi/(A,\emptyset) \lra \U(P)
	\]
	which satisfies the hypothesis of Lemma \ref{lem:vankampen} thus proving the contractibility
	of $\pi/(A,\emptyset)$. 

	Now suppose that $A$ does contain a special point $x \in N$. Then we first claim that the
	inclusion 
	\[
		j: (\pi/(A,\emptyset))^{\simeq} \subset \pi/(A,\emptyset)
	\]
	is cofinal. To this end, we need to show that, given an object $b = (B, \alpha: (B,\emptyset) \to
	(A,\emptyset))$ of $\pi/(A,\emptyset)$, the category $b/j$ is contractible. We
	consider the space $Q$ of paths in $\mathring{X} \setminus{(N \setminus {x})}$ starting in
	$\mathring{B}$ and ending in $x$. To an object $b' = (B', \alpha: (B',\emptyset) \to
	(A,\emptyset))$ of $b/j$, we associate the open subset of $Q$ consisting of paths that
	lie in $\mathring{B'}$. An application of Lemma \ref{lem:vankampen} proves the claim.
	Finally, an argument similar to the above shows that $(\pi/(A,\emptyset))^{\simeq}$ is
	contractible so that the result follows from Proposition \ref{prop:oo-localiz-dual}. 
\end{proof}

As a consequence of the results of this section, we thus obtain the following:

\begin{cor}\label{cor:constructible} Let $(X,N)$ be a stratified surface and let $\D$ be an
	$\infty$-category. Then there is an equivalence 
	\[
		\Sh(X,N;\D) \simeq \Fun(S(X,N)^{\op}, \D)
	\]
	of $\infty$-categories.
\end{cor}

\begin{rem} In view of Example \ref{exa:paraexit}, Corollary \ref{cor:constructible} recovers the presentation of the
	constructible derived category in terms of the exit path category. Nevertheless, the description
	in terms of $S(X,N)$ will be more convenient in what follows. 
\end{rem}

\subsection{Verdier duality}
\label{sec:verdier}

In this section we assume that $\D$ is a stable   $\oo$-category.

Recall that in Corollary \ref{cor:constructible}, we have identified the  $\infty$-category
of  constructible sheaves on $(X,N)$ with values in any $\infty$-category
 $\D$ with the $\infty$-category
$\Fun(S(X,N)^{\op}, \D)$ of presheaves on the category $S(X,N) \subset \Lambdad(X,N)$. In this
section, under the assumption that $\D$ is stable, we
illustrate the use of Kan extensions among subcategories of $\Lambdad(X,N)$ to provide a
proof of Verdier duality. This treatment can be regarded as an introduction to the techniques to be
used in \S \ref{sec:milnor} to establish our description of $\Sh(X,N;\D)$ as Milnor sheaves.

Along with $S=\Sd=\Sd(X,N)$ and $\Bd=\Bd(X,N)$ defined earlier,
 we  consider the following full subcategories of $\Lambdad(X,N)$ consisting, respectively, of the following disks:
\begin{itemize}
	 \item $\Dd$: disks $(A,A')$ of the form 
			\[
				\disk{0}{1}, \disk{1b}{1} \subset X
			\]
		where $A' \subset \partial A$ is a single closed interval and $A \cap N \subset A'$, 
	\item $\Vd$: disks $(A,A')$ of the form
			\[
				\disk{1}{1} \subset X
			\]
		where $A' \subset \partial A$ is a single closed
		interval and $(A \setminus A') \cap N$ is a singleton.
\end{itemize}
The category $S$ is equivalent to the entrance path category of $(X,N)$
so that, for any stable $\infty$-category $\D$, the $\infty$-category $\Fun(S^{\op},\D)$ can be
identified with the category of constructible sheaves on $(X,N)$ valued in $\D$ (Corollary
\ref{cor:constructible}). We set $\Qd := S \cup \Bd \cup \Dd \cup \Vd$.  

\begin{thm}\label{thm:verdierkan}
	Let $\D$ be a stable $\infty$-category. And let
	\[
		\F: \Qd \lra \D
	\]
	be a functor. Then the following are equivalent:
	\begin{enumerate}[label=(\roman*)]
		\item \label{kan:1} $\F$ satisfies the following conditions:
			\begin{enumerate}[label=(\arabic*)]
				\item $\F|S \cup \Dd$ is a right Kan extension of $\F|S$, and
				\item $\F$ is a left Kan extension of $\F|S \cup \Dd$.
			\end{enumerate}
		\item \label{kan:2} 
			\begin{enumerate}[label=(\arabic*)]
				\item $\F|\Bd$ maps weak equivalences in $\Bd$ to equivalences in $\D$,
				\item $\F|\Bd \cup \Dd$ is a left Kan extension of $\F|\Bd$, and
				\item $\F$ is a right Kan extension of $\F|\Bd \cup \Dd$.
			\end{enumerate}
	\end{enumerate}
\end{thm}

Before we provide a proof of the Theorem, we explain its implications. 
The following lemma generalizes one of the statements of Corollary
\ref{cor:milnorlocalization}. 

\begin{lem}\label{lem:bloc} Let $B \subset \Lambda(X,N)$ denote the full subcategory spanned by the
	bounded disks. Then the restriction 
	$\Bd \to  B$
	of the canonical functor $\Lambdad(X,N) \to \Lambda(X,N)$ exhibits $B$ as an
	$\infty$-categorical localization of $\Bd$ along the weak equivalences. 
\end{lem}

\begin{proof}
	This follows immediately from Proposition \ref{prop:oo-localiz-dual}.
	\end{proof}

\begin{cor} There is a canonical equivalence of stable $\infty$-categories
	\[
		\delta: \Fun(S^{\op},\D) \simeq \Fun(S, \D)
	\]
	identifying constructible sheaves and constructible cosheaves valued in $\D$. 
\end{cor}
\begin{proof}
 Let  $\Fun(\Qd,\D)',  \Fun(\Qd,\D)''\subset \Fun(\Qd,\D)$ be the full
($\oo$-)subcategories consisting of functors
	satisfying the conditions \ref{kan:1} and  \ref{kan:2}  of Theorem \ref{thm:verdierkan}, respectively. 
By Proposition
	\ref{prop:kanres} the restriction functor
	\[
		p: \Fun(\Qd,\D)' \lra \Fun(S,\D),
	\]
	 is an equivalence.
	For the same reason, the restriction functor
	\[
		q: \Fun(\Qd,\D)'' \lra \Fun(\Bd[W^{-1}]_\oo,\D),
	\]
	  is an equivalence. 
	Since \ref{kan:1} and \ref{kan:2} are equivalent, and $\Bd[W^{-1}]_\oo \simeq B$ by Lemma
	\ref{lem:bloc}, we obtain an equivalence 
	\[
		\Fun(S,\D) \simeq \Fun(B,\D)
	\]
	by composing an inverse of $p$ with $q$. The equivalence $B \simeq S^{\op}$ induced by the
	duality $\xi$ then yields the desired result.
\end{proof}

\begin{proof}[Proof of Theorem \ref{thm:verdierkan}]

	Let $\F: \Qd \lra \D$ be a functor.

	We will provide concrete interpretations of the Kan extension conditions in
	\ref{kan:1} and \ref{kan:2} so as the claimed equivalence will become an apparent
	consequence of the stability of the $\infty$-category $\D$. 

	We begin with \ref{kan:1}: For every $(A,A') \in \Dd$, the category
	$(A,A')/S$ is empty so that $\F|S \cup \Dd$ is a right Kan extension of $\F|S$ if and
	only if, for every $(A,A') \in \Dd$, we have $\F(A,A') \simeq 0$.

	Suppose now that $\F$ is a left Kan extension of $\F|S \cup \Dd$. 
	We first determine the value of $\F$ at 
	\[
		(A,A')\, =\,  \disk{0}{s}\, \in\, \Bd
	\]
	as determined by the pointwise formula \eqref{eq:kan-pointwise}. 
	The overcategory $S \cup \Dd/(A,A')$ admits an $\infty$-cofinal subcategory depicted by 
	\begin{equation}\label{eq:cofinalsub1}
			\begin{tikzcd}
				&&  \rotdisk{0}{1}{90} & & \\
				& \rotdisk{0}{1}{180}\ar{ur}\ar{dl} & \ar{l} \disk{0}{0} \ar{r}\ar{d}\ar{u}\ar{dll}\ar{drr} &
				\rotdisk{0}{1}{0}\ar{ul}\ar{dr} & & \subset  \disk{0}{s}\\
				\rotdisk{0}{1}{225} & &  \rotdisk{0}{1}{-90} \ar{rr} \ar{ll}& &
				\rotdisk{0}{1}{-45} 
			\end{tikzcd}
	\end{equation}
	where the morphisms around the boundary of the $2$-simplex are given by rotation by the
	smallest possible angle so that a full turn is obtained by traversing the boundary once.
	Thus, the value of $\F$ at $(A,A')$ is determined by the colimit cone (i.e. biCartesian
	cube)
	\begin{equation}\label{eq:cube1}
			\begin{tikzcd}[column sep=.2cm,row sep=.2cm]
				\F(\disk{0}{0}) \ar{rr} \ar{dd}\ar{dr} & & \F(\rotdisk{0}{1}{0})
				\ar{dd} \ar{dr}\\
				& \F(\rotdisk{0}{1}{180}) \ar{dd}\ar{rr} & & \F(\rotdisk{0}{1}{90})\ar{dd} \\
				\F(\rotdisk{0}{1}{-90}) \ar{rr}\ar{dr} & &
				\F(\rotdisk{0}{1}{-45})\ar{dr}\\ 
				& \F(\rotdisk{0}{1}{225}) \ar{rr} & & \F(\disk{0}{s})
			\end{tikzcd}
	\end{equation}
	In particular, since $\F|\Dd \simeq 0$, this biCartesian cube induces an equivalence
	\[
		\F(\disk{0}{s}) \simeq \F(\disk{0}{0})[2].
	\]
	Similarly, for a disk of the form
	\[
		(A,A') = \disk{1b}{s} \in \Bd
	\]
	the overcategory $S \cup \Dd/(A,A')$ admits an $\infty$-cofinal subcategory depicted by 
	\begin{equation}\label{eq:cofinalsub2}
			\begin{tikzcd}
				&&  \rotdisk{1b}{1}{90} & & \\
				& \rotdisk{0}{1}{180}\ar{ur}\ar{dl} & \ar{l} \disk{0}{0} \ar{r}\ar{d}\ar{u}\ar{dll}\ar{drr} &
				\rotdisk{0}{1}{0}\ar{ul}\ar{dr} & & \subset  \disk{1b}{s}\\
				\rotdisk{0}{1}{225} & &  \rotdisk{0}{1}{-90} \ar{rr} \ar{ll}& &
				\rotdisk{0}{1}{-45} 
			\end{tikzcd}
	\end{equation}
	again exhibiting an equivalence $\F(\disk{1b}{s}) \simeq \F(\disk{0}{0})[2]$. More
	precisely, we observe that any weak equivalence
	\[
		\disk{0}{s} \to \disk{1b}{s} 
	\]
	in $\Bd$ induces an equivalence $\F(\disk{0}{s}) \to \F(\disk{1b}{s})$ in $\D$, since the
	induced map relating the above $\infty$-cofinal subcategories \eqref{eq:cofinalsub2} and
	\eqref{eq:cofinalsub1} becomes a pointwise equivalence upon applying $\F$. 

	We next describe the value
	of $\F$ at 
	\[
		(A,A') = \disk{1}{1} \in \Vd. 
	\]
	To this end, we argue that the overcategory $(S \cup
	\Dd)/(A,A')$ contains an $\infty$-cofinal subcategory of the form
	\[
			\begin{tikzcd}
				\dot\;\disk{0}{0} \ar{r}\ar{d}\ar[dashed]{drrr} & \dot\;\disk{0}{1}\ar[dashed]{drr}\\
				\disk{1}{0} \ar[dashed]{rrr} &&& \disk{1}{1}.
			\end{tikzcd}
	\]
	In particular, the value of $\F$ at $(A,A')$ is determined by the colimit cone (i.e. biCartesian
	square)
	\begin{equation}\label{eq:cube2}
			\begin{tikzcd}
				\F(\dot\;\disk{0}{0}) \ar{r} \ar{d} & \F(\dot\;\disk{0}{1}) \ar{d}\\
				\F(\disk{1}{0}) \ar{r} & \F(\disk{1}{1}). 
			\end{tikzcd}
	\end{equation}
	Since the top--right object is a zero object, this diagram exhibits $\F(A,A')$ as a
	cofiber (cone) of the morphism
	\[
			\begin{tikzcd}
				\F(\dot\;\disk{0}{0}) \ar{r} & \F(\disk{1}{0}).
			\end{tikzcd}
	\]

	Finally, it remains to characterize the value of $\F$ at 
	\[
		(A,A') = \disk{1}{s} \in \Bd.
	\]
	Similarly, as in Step 1, the overcategory $S \cup \Dd \cup \Vd/(A,A')$ admits an
	$\infty$-cofinal subcategory depicted by 
	\[
			\begin{tikzcd}
				&&  \rotdisk{1}{1}{90} & & \\
				& \rotdisk{1}{1}{180}\ar{ur}\ar{dl} & \ar{l} \disk{1}{0} \ar{r}\ar{d}\ar{u}\ar{dll}\ar{drr} &
				\rotdisk{1}{1}{1}\ar{ul}\ar{dr} & & \subset  \disk{1}{s}\\
				\rotdisk{1}{1}{225} & &  \rotdisk{1}{1}{-90} \ar{rr} \ar{ll}& &
				\rotdisk{1}{1}{-45} 
			\end{tikzcd}
	\]
	where the morphisms around the boundary of the $2$-simplex are given by rotation by the
	smallest possible angle. Thus, the value of $\F$ at $(A,A')$ is determined by the colimit cone (i.e. biCartesian
	cube)
	\begin{equation}\label{eq:cube3}
			\begin{tikzcd}[column sep=.2cm,row sep=.2cm]
				\F(\disk{1}{0}) \ar{rr} \ar{dd}\ar{dr} & & \F(\rotdisk{1}{1}{0})
				\ar{dd} \ar{dr}\\
				& \F(\rotdisk{1}{1}{180}) \ar{dd}\ar{rr} & & \F(\rotdisk{1}{1}{90})\ar{dd} \\
				\F(\rotdisk{1}{1}{-90}) \ar{rr}\ar{dr} & &
				\F(\rotdisk{1}{1}{-45})\ar{dr}\\ 
				& \F(\rotdisk{1}{1}{225}) \ar{rr} & & \F(\disk{1}{s})
			\end{tikzcd}
	\end{equation}
	In conclusion, we may characterize the functors $\F: \Qd \lra \D$ satisfying the Kan extension
	conditions of \ref{kan:1} as those functors for which $\F|\Dd \simeq 0$ and further the square 
	\eqref{eq:cube2} as well as the cubes \eqref{eq:cube1} and  \eqref{eq:cube3} are
	biCartesian. 

	We now discuss the Kan extension conditions of \ref{kan:2}. A similar argumentation as the
	one for \ref{kan:1} show that a functor $\F$ satisfies the conditions of \ref{kan:2}(2) and \ref{kan:2}(3) if and
	only if 
	\begin{enumerate}
		\item $\F|\Dd \simeq 0$, 
		\item the cubes \eqref{eq:cube1} and \eqref{eq:cube3} are limit
			cones, and hence biCartesian, 
		\item the square
			\begin{equation}\label{eq:dualcube2}
			\begin{tikzcd}
				\F(\rotdisk{1}{1}{180}) \ar{r} \ar{d} &
				\F(\rotdisk{1b}{1}{180}) \ar{d}\\
				\F(\disk{1}{s}) \ar{r} & \F(\rotdisk{1b}{s}{90})),
			\end{tikzcd}
			\end{equation}
			is biCartesian.
	\end{enumerate}
	Thus, to finish the proof, we have to argue why, assuming further \ref{kan:2}(1), 
	\eqref{eq:cube2} being biCartesian is equivalent to \eqref{eq:dualcube2} being biCartesian
	(in the presence of the remaining conditions). To show this, consider the commutative diagram
	in $\Qd$ depicted by
	\begin{equation}\label{eq:cubestack}
			\begin{tikzcd}[column sep=.2cm,row sep=.2cm]
				\disk{0}{0} \ar{rr} \ar{dd}\ar{dr} & & \rotdisk{0}{1}{0}
				\ar{dd} \ar{dr}\\
				& \rotdisk{0}{1}{180} \ar{dd}\ar{rr} & & \rotdisk{0}{1}{90}\ar{dd} \\
				\rotdisk{1}{0}{0} \ar{rr}\ar{dr}\ar{dd} & &
				\rotdisk{1}{1}{0}\ar{dr}\ar{dd}\\ 
				& \rotdisk{1}{1}{180} \ar{dd}\ar{rr} & & \rotdisk{1}{1}{90}\ar{dd} \\
				\rotdisk{1}{1}{-90} \ar{rr}\ar{dr}\ar{dd} & &
				\rotdisk{1}{1}{-45}\ar{dr}\ar{dd}\\ 
				& \rotdisk{1}{1}{225} \ar{rr}\ar{dd} & & \disk{1}{s} \ar{dd}\\
				\rotdisk{1b}{1}{-90} \ar{rr}\ar{dr} & &
				\rotdisk{1b}{1}{-45}\ar{dr}\\ 
				& \rotdisk{1b}{1}{225} \ar{rr} & & \disk{1b}{s}.
			\end{tikzcd}
	\end{equation}
	consisting of three stacked cubes, and further, the diagram in $\D$ obtained by applying
	$\F$. The middle cube is identical to \eqref{eq:cube3} which is biCartesian. Furthermore,
	the cube given by the composite of the three cubes coincides may be decomposed as
	\begin{equation}
			\begin{tikzcd}[column sep=.2cm,row sep=.2cm]
				\disk{0}{0} \ar{rr} \ar{dd}\ar{dr} & & \rotdisk{0}{1}{0}
				\ar{dd} \ar{dr}\\
				& \rotdisk{0}{1}{180} \ar{dd}\ar{rr} & & \rotdisk{0}{1}{90}\ar{dd} \\
				\rotdisk{0}{1}{-90} \ar{rr}\ar{dr}\ar{dd} & &
				\rotdisk{0}{1}{-45}\ar{dr}\ar{dd}\\ 
				& \rotdisk{0}{1}{225} \ar{rr}\ar{dd} & &
				\disk{0}{s}\ar{dd}\\
				\rotdisk{1b}{1}{-90} \ar{rr}\ar{dr} & &
				\rotdisk{1b}{1}{-45}\ar{dr}\\ 
				& \rotdisk{1b}{1}{225} \ar{rr} & & \disk{1b}{s}.
			\end{tikzcd}
	\end{equation}
	where the top cube coincides with \eqref{eq:cube1} and the bottom cube is biCartesian
	since all vertical maps are weak equivalences which, 
	by \ref{kan:2}(1), are mapped to equivalences by $\F$. Thus the composite cube is
	biCartesian as well. 
	The front face of the top cube in \eqref{eq:cubestack} is biCartesian, since it contains two parallel arrows 
	that are equivalences. Proposition \ref{prop:cuberecursive} implies that the top cube is
	biCartesian if and only if its back face, which coincides with \eqref{eq:cube2}, is
	biCartesian. By the same argument, the bottom cube will be biCartesian if and only if its
	front face, which coincides with \eqref{eq:dualcube2}, is biCartesian. As a consequence, the
	two-out-of-three property for the pasting of
	biCartesian cubes (Proposition \ref{prop:2out3}) implies that \eqref{eq:cube2} is
	biCartesian if and only if \eqref{eq:dualcube2} is biCartesian, concluding our argument.
\end{proof}

\section{Milnor sheaves}
\label{sec:milnor}

By Corollary \ref{cor:constructible}, the $\infty$-category of constructible sheaves $\Sh(X,N;\D)$ with values in a stable
$\infty$-category $\D$ may be parametrized in terms of standard disks: there is an equivalence
\begin{equation}\label{eq:constructible}
		\Sh(X,N;\D) \simeq \Fun(S(X,N)^{\op}, \D).
\end{equation}
If $\D$ is the derived category of an abelian category $\A$, then this equivalence restricts to an equivalence 
\[
	\Sh(X,N;\A) \simeq \Fun(S(X,N)^{\op}, \A).
\]
In other words, the equivalence \eqref{eq:constructible} is compatible with the standard
$t$-structure on $\Sh(X,N;\D)$. In this section, we provide yet another parametrization of
$\Sh(X,N;\D)$, in terms of Milnor disks, which is in the same sense compatible with the perverse
$t$-structure. In particular, it provides an intrinsically abelian description of the category
of perverse sheaves. 

\subsection{Constructible sheaves as Milnor sheaves}

Let $(X,N)$ be a stratified surface and let $\Lambdad(X,N)$ denote its directed paracyclic category.
A {\em collared cut} of an object $(A,A') \in \Lambdad(X,N)$ consists of
\begin{itemize}
	\item a {\em cut} $\alpha$, by which we mean an embedding $\alpha: I \to A$ with
		$\alpha^{-1}(\partial A) = \{0,1\}$. We denote the two connected components of the
		complement if $\alpha(I)$ in $A$ by $U_1$ and $U_2$. 
	\item a {\em collar} for $\alpha$, by which we mean a continuous map $G: [-1,1] \times I \to
		A$ such that
		\begin{itemize}
			\item $G(0,t) = \alpha(t)$, 
			\item for every $s \in [-1,1]$, the map $G(s,-)$ is a cut, 
			\item $G([-1,1] \times I) \cap \partial A' = \emptyset$,
			\item $G(\{-1,1\} \times I) \cap N = \emptyset$,
		\end{itemize}
		We denote $C = G([-1,1] \times I)$ and $A_1 = U_1 \cup C$ and $A_2 = U_2 \cup C$.
\end{itemize}
Associated to a collared cut, there is a commutative square
\begin{equation}\label{eq:cutsquare}
		\begin{tikzcd}
			(C,C \cap A') \ar{d} \ar{r} &  (A_1, A_1 \cap A') \ar{d}\\
			(A_2, A_2 \cap A') \ar{r} &  (A, A')
		\end{tikzcd}
\end{equation}
in $\Lambdad(X,N)$.

\begin{defi} Let $\D$ be a pointed $\infty$-category. A functor $\F: \Md(X,N) \to \D$ is called
	a {\em Milnor cosheaf} if
	\begin{enumerate}
		\item $\F$ maps weak equivalences in $\Md(X,N)$ to equivalences in $\D$,
		\item $\F$ maps objects of the form
			\[
				\disk{1b}{1},\disk{0}{1} \in \Md(X,N) 
			\]
			to a zero object,
		\item for every object $(A,A')$ of $\Md(X,N)$ and for every collared cut of
			$(A,A')$, such that the associated diagram \eqref{eq:cutsquare} takes
			values in $\Md(X,N)$, $\F$ maps \eqref{eq:cutsquare} to a coCartesian square
			in $\D$.
	\end{enumerate}
	Dually, $\F: \Md(X,N)^{\op} \to \D$ is called a {\em Milnor sheaf} if $\F^{\op}$ is a
	Milnor cosheaf. 
	We denote by $\Fun^{\sharp}(\Md(X,N),\D)$ the $\infty$-category of Milnor cosheaves and by
	$\Fun^{\sharp}(\Md(X,N)^{\op},\D)$ the $\infty$-category of Milnor sheaves defined as full
	subcategories of the respective functor categories. 
\end{defi}

\begin{defi} Let $(A,\emptyset)$ be an object of $S(X,N)$. We denote by $\Lambda_A$ the subcategory
	of $\Md(X,N)$ with objects $(A,A')$ and morphisms, represented by an isotopy $H: I \times
	\DD \to X$ such that, for every $t \in I$, $H_t(\DD) = A$. We further denote by
	$\Lambda_A^+ = \Lambda_A \cup (A,\emptyset)$ obtained by adjoining the initial object
	$(A,\emptyset)$. 
\end{defi}

\begin{rem}\label{rem:paracyclic} For every $(A, \emptyset)$, the category $\Lambda_{A}$
	is equivalent to the paracyclic category $\Lambda_{\infty}$.
\end{rem}

We say that a functor $\F: \Md(X,N) \lra \D$ is {\em locally Segal} if, for every $(A,\emptyset)
\in S(X,N)$, with $\partial A \cap N = \emptyset$, the object $\F|\Lambda_{A}$ is a Segal object,
i.e., the restriction along an embedding $\Delta \to \Lambda_{A}$ is Segal.

\begin{prop}\label{prop:milnminimal} 
	Let $\D$ be a stable $\infty$-category and let $\F: \Md(X,N) \to \D$ be a functor. Then $\F$
	is a Milnor cosheaf if and only if the following hold:
	\begin{enumerate}
		\item\label{miln:0} $\F$ maps weak equivalences to equivalences in $\D$.
		\item\label{miln:1} $\F$ is locally Segal.
		\item\label{miln:2} For every $x \in N$, $\F$ maps any square of the form 
			\[
			\begin{tikzcd}
				\disk{0}{1} \ar{r}\ar{d} & \disk{1}{1}\ar{d}\\
				\disk{0}{2} \ar{r} & \disk{1}{2}
			\end{tikzcd}
			\]
			to a coCartesian square in $\D$.
	\end{enumerate}
\end{prop}
\begin{proof}
	Condition \ref{miln:0} appears directly in the definition of a Milnor cosheaf.

	Every Milnor cosheaf satisfies the conditions \ref{miln:1} and \ref{miln:2} since, using
	the version \eqref{eq:segal2} of the Segal conditions, the respective coCartesian squares
	all arise from collared cuts. 

	Suppose now that $\F$ satisfies \ref{miln:1} and \ref{miln:2}. Let $(A,\emptyset) \in
	\Md(X,N)$ with $A \cap N = \emptyset$. Then the local Segal conditions
	imply that \eqref{eq:cutsquare} is coCartesian for every cut which only intersects one of
	the boundary intervals. If the cut $\alpha$ intersects two boundary intervals, then it is
	straightforward to deduce that \eqref{eq:cutsquare} is coCartesian by considering cuts
	$\alpha_1$ of $(A_1, A_1 \cap A')$ and $\alpha_2$ of 
	$(A_2,A_2 \cap A')$ which are obtained by sliding the endpoint of
	$\alpha$ out of the boundary interval towards the two possible directions (In the language
	of \cite{HSS}, this amounts to the statement that every $1$-Segal object is $2$-Segal).

	By the exact same argumentation, we deduce that \eqref{eq:cutsquare} is coCartesian for
	$(A,\emptyset) \in \Md(X,N)$ with $A \cap N = \{x\}$ as long as the cut
	$\alpha$ runs through the special point $x$. It remains to verify the coCartesianess of
	\eqref{eq:cutsquare} for a cut $\alpha$ which does {\em not} run through $x$. But this case
	can be reduced to \eqref{miln:2} by induction on the number of boundary intervals: The
	induction step is obtained by introducing one additional cut which runs either through the
	special point $x$, or lies completely in the component of $A \setminus \alpha(I)$ which
	does not contain $x$.
\end{proof}

We denote by 
\[
	\Mdp(X,N) \subset \Lambdad(X,N)
\]
the full subcategory spanned by the standard and Milnor disks and by
\[
	\Lambdad(X,N)_{\le n} \subset \Mdp(X,N)
\]
the full subcategory consisting of objects $(A,A')$ such that $A'$ has at most $n$ connected
components.

Further, we denote by 
\[
	\Dd(X,N) \subset \Lambdad(X,N)
\]
the full subcategory of objects $(A,A') \in \Lambdad(X,N)$ of the form
\[
	\disk{0}{1},\disk{1b}{1} \in \Lambdad(X,N).
\]

\begin{thm}\label{thm:main} Let $\F: \Mdp(X,N) \lra \D$ be a functor.
	Then the following are equivalent:
	\begin{enumerate}
		\item\label{main:1} $\F|\Dd(X,N) \simeq 0$ and $\F$ is a left Kan extension of
			$\F|S(X,N) \cup \Dd(X,N)$. 
		\item\label{main:2} $\F|\Md(X,N)$ is a Milnor cosheaf and $\F$ is a right Kan extension of 
			$\F|\Md(X,N)$. 
	\end{enumerate}
\end{thm}
\begin{proof}
	Suppose that $\F$ is a left Kan extension of $\F|S(X,N) \cup \Dd(X,N)$. To show that $\F|\Md(X,N)$ is
	a Milnor cosheaf, we verify conditions \ref{miln:1} and \ref{miln:2} of Proposition \ref{prop:milnminimal}.
	Let $(A,A') \in \Md(X,N)$ with $A \cap N$ empty. Then the inclusion
	\[
		\Lambdap_{A, \le 1}/(A,A') \subset (S(X,N) \cup \Dd(X,N))/(A,A')
	\]
	is an equivalence of categories and hence $\infty$-cofinal. In particular, by Proposition
	\ref{prop:localpara}, $\F|\Lambda_{A}$ is a Segal object. 

	Now let 
	\[
		(A,A') = \disk{1}{1} \in \Md(X,N)
	\]
	such that $A \cap N = \{x\} \subset A \setminus A'$ is a singleton.  
	Then the inclusion 
	\[
		(S(A,\{x\}) \cup \Dd(A,\{x\})/(A,A') \subset (S(X,N) \cup \Dd(X,N))/(A,A'),
	\]
	where the first undercategory is taken in $\Md(U,\{x\})$, is an equivalence, in
	particular $\infty$-cofinal. Now the category $(S(A,\{x\}) \cup D(A,\{x\})/(A,\varphi)$ is equivalent
	to the category depicted by
	\[
			\begin{tikzcd}
				\arrow[loop left]{l}{\ZZ}\dot\;\disk{0}{0} \ar{r}{\ZZ} \ar{d} & \dot\;\disk{0}{1} \ar[dashed]{d}\\
				\disk{1}{0} \ar[dashed]{r} & \disk{1}{1}. 
			\end{tikzcd}
	\]
	where the automorphisms $\ZZ$ correspond to the disk moving around the special point $x$, so
	that it is equivalent to the category depicted by
	\[
			\begin{tikzcd}
				\arrow[loop left]{l}{\ZZ}\bullet \ar{r}{\ZZ} \ar{d} & \bullet \\
				\bullet.
			\end{tikzcd}
	\]
	This latter category contains the $\infty$-cofinal subcategory
	\[
			\begin{tikzcd}
				\bullet \ar{r} \ar{d} & \bullet \\
				\bullet,
			\end{tikzcd}
	\]
	so that the pointwise left Kan extension condition for $\F(A,A')$ is thus equivalent to the
	square
	\begin{equation}\label{eq:pushout1}
			\begin{tikzcd}
				\F(\dot\;\disk{0}{0}) \ar{r} \ar{d} & \F(\dot\;\disk{0}{1}) \ar{d}\\
				\F(\disk{1}{0}) \ar{r} & \F(\disk{1}{1}). 
			\end{tikzcd}
	\end{equation}
	being coCartesian. 
	For a more general $(A,A') \in M(X,N)$ with $A \cap N = \{x\} \subset A \setminus A'$, by a similar
	argument, the category $(S(X,N) \cup D(X,N))/(A,A')$ contains 
	a cofinal subcategory $C$ of the form
	\begin{equation}\label{eq:cofinal1}
			\begin{tikzcd}
				& & \dot\;\disk{0}{0} \ar{dll} \ar{dl} \ar{dr} & &\\ 
				\disk{1}{0}\ar[dashed,swap]{drr}{f} & \dot\;\disk{0}{1}\ar[dashed]{dr}{f_1} & \dots &
				\dot\;\disk{0}{1}\ar[dashed]{dl}{f_n} \\
				&& (A,A') &&
			\end{tikzcd}
	\end{equation}
	where 
	\[
		f: \disk{1}{0} \to (A,A')
	\]
	is given by the constant isotopy and 
	the morphisms 
	\[
		f_i: \dot\;\disk{0}{1} \to (A,A')
	\]
	enter the special point and map the unique interval to the $i$th interval of $A$ (with
	respect to some chosen order). We have a functor from the category
	\[
			\begin{tikzcd}
				&  & 0 \ar{dll}\ar{dl} \ar{dr} &\\ 
				1 & 2 & \dots & n 
			\end{tikzcd}
	\]
	to $\on{Cat}/C$ by associating to $0$ the subcategory
	\[
		\begin{tikzcd}
			\dot\;\disk{0}{0} \ar{r} & \disk{1}{0} \ar[dashed]{r}{f} & (A,A') 
	 	\end{tikzcd}
	\]
	of $C$ and to $i > 0$ the subcategory 
	\[
			\begin{tikzcd}
				\dot\;\disk{0}{0} \ar{r}\ar{d} & \disk{0}{1}\ar[dashed]{d}{f_i}\\
				\disk{1}{0}\ar[dashed]{r}{f} & (A,A').
			\end{tikzcd}
	\]
	An application of \cite[4.2.3.10]{HTT}, using that \eqref{eq:pushout1} is a pushout, implies
	that the pointwise left Kan condition for $(A,A')$ is equivalent to the diagram 
	\[
			\begin{tikzcd}
				&  \F(\disk{1}{0}) \ar{dl}\ar{dr} & &\\ 
				\F(\disk{1}{1})\ar[swap]{dr}{g_1} & \dots &
				\F(\disk{1}{1})\ar{dl}{g_n} \\
				& \F(A,A') &&
			\end{tikzcd}
	\]
	being a colimit cone. Here the maps $g_i: \disk{1}{1} \to (A,A')$ are morphisms in
	$\Lambda_{A}$ which move the interval into the various intervals comprising $A'$. In
	particular, this implies that the diagram $\F|\Lambda_{A}^+$ is a left Kan extension of
	its restriction to $\F|\Lambda_{A, \le 1}^+$ so that $\F|\Lambda_{A}$ satisfies
	the Segal conditions by Proposition \ref{prop:localpara}. We have thus shown that $\F$ is locally Segal. 

	A similar argument shows that the value of $\F$ at a disk $(A,A')$ with
	$A' \cap N = \{x\}$ is determined by the colimit cone
	\begin{equation}
			\begin{tikzcd}
				& & \F(\dot\;\disk{0}{0}) \ar{dll} \ar{dl} \ar{dr} & &\\ 
				\F(\disk{1b}{1}) \ar[dashed,swap]{drr}{f} & \F(\dot\;\disk{0}{1})\ar[dashed]{dr}{f_1} & \dots &
				\F(\dot\;\disk{0}{1}) \ar[dashed]{dl}{f_n} \\
				&& \F(A,A') &&
			\end{tikzcd}
	\end{equation}
	Further, the Segal conditions for $\F$ at a disk $(A_0,A_0')$, obtained by moving $(A,A')$
	away from the special point $x$ so that $A_0 \cap N = \emptyset$, imply that the diagram  
	\begin{equation}
			\begin{tikzcd}
				& & \F(\dot\;\disk{0}{0}) \ar{dll} \ar{dl} \ar{dr} & &\\ 
				\F(\dot\;\disk{0}{1}) \ar[dashed,swap]{drr}{f} & \F(\dot\;\disk{0}{1})\ar[dashed]{dr}{f_1} & \dots &
				\F(\dot\;\disk{0}{1}) \ar[dashed]{dl}{f_n} \\
				&& \F(A_0,A'_0) &&
			\end{tikzcd}
	\end{equation}
	is a colimit cone. Since the map 
	\[
		\F(\dot\;\disk{0}{1}) \lra \F(\disk{1b}{1})
	\]
	is, as a map between zero objects, an equivalence, we deduce from the induced map on colimit
	cones that the map $\F(A_0,A_0') \to \F(A,A')$ is an equivalence as well. In particular,
	$\F$ maps weak equivalences to equivalences in $\D$. 

	Condition \ref{miln:2} follows by applying \cite[4.2.3.10]{HTT} to \eqref{eq:cofinal1}
	for $n=2$ with respect to the functor from the category 
	\[
			\begin{tikzcd}
				0 \ar{d}\ar{r} & 2\\ 
				1  
			\end{tikzcd}
	\]
	into $\on{Cat}/C$ which associates to $0$ the subcategory
	\[
		\begin{tikzcd}
			\dot\;\disk{0}{0} \ar{r} & \disk{0}{1} \ar[dashed]{r}{f_1} & (A,A') 
	 	\end{tikzcd}
	\]
	to $1$ the subcategory
	\[
			\begin{tikzcd}
				\dot\;\disk{0}{0} \ar{r}\ar{d} & \disk{1}{0}\ar[dashed]{d}{f}\\
				\disk{0}{1}\ar[dashed]{r}{f_1} & (A,A')
			\end{tikzcd}
	\]
	and to $2$ the subcategory
	\[
			\begin{tikzcd}
				\dot\;\disk{0}{0} \ar{r}\ar{d} & \disk{0}{1}\ar[dashed]{d}{f_2}\\
				\disk{0}{1}\ar[dashed]{r}{f_1} & (A,A').
			\end{tikzcd}
	\]

	The above statements impy that $\F|\Md(X,N)$ is a Milnor sheaf. It remains to show that $\F$ is a right Kan
	extension of $\F|\Md(X,N)$. To this end, let 
	\[
		(A,\emptyset) = \disk{1}{0} \in S(X,N) 
	\]
	with $A \cap N = \{x\}$ a singleton. Then it is easily seen that the inclusion
	\[
		(A,\emptyset)/\Lambda_{A} \subset (A,\emptyset)/\Md(X,N),
	\]
	where the left-hand overcategory is taken in the category $\Lambdap_{A}$, is
	$\infty$-coinitial. Thus, by Proposition \ref{prop:localpara} below, the value of $\F$ at
	$(A,\emptyset)$ is given by right Kan extension of $\F|\Md(X,N)$. 

	Finally, consider 
	\[
		(A,\emptyset) = \disk{0}{0} \in S(X,N) 
	\]
	with $A \cap N$ empty. Again, we consider the inclusion
	\[
		j: (A,\emptyset)/\Lambda_{A} \subset (A,\emptyset)/\Md(X,N).
	\]
	We claim that $j$ is $\infty$-coinitial. To this
	end, we have to verify, for every $f: (A,\emptyset) \to (A_1,A_1') \in
	(A,\emptyset)/M(X,N)$, that $j/f$ is contractible. 
	This statement is clear if $A_1 \cap N = \emptyset$. Suppose now that $A_1 \cap N = \{x\}
	\subset A_1 \setminus A'_1$. In this case, we proceed by exhibiting a contractible
	$\infty$-cofinal subcategory of $j/f$: Fix an object $a_0$
	of $j/f$ whose underlying disk $(B,B')$ has $|\pi_0(B')| = |\pi_0(A'_1)|+1$ boundary
	components and such that the map $(B, B') \to (A_1,A'_1)$ includes $|\pi_0(A'_1)|-1$ intervals
	of $B'$ into respective intervals of $A'_1$ and includes the two intervals adjacent to the entry
	location of $x$ into the remaining interval of $A'_1$. There are objects $\{a_i | i \in \ZZ\}$
	of $j/f$ which differ from $a_0$ in that the entry point of $x$ lies $i$ segments in $S^1
	\setminus B'$ away from the entry point of $x$ for $a_0$. For $i \in \ZZ$, we denote by
	$a_i^+$ and $a_i^-$ the two objects of $j/f$ obtained by omitting one of the intervals of
	$B'$ adjacent to the entry point of $x$. The full subcategory of $j/f$ spanned by these
	objects has the form:
	\[
		\begin{tikzcd}
			\dots& a_{-1} & \ar{l}	a_{-1}^+ = a_0^- \ar{r} & a_0 & a_0^+ =a_1^- \ar{l}
			\ar{r} & a_1 & \dots .
		\end{tikzcd}
	\]
	It is now straighforward to verify that this subcategory is cofinal in $j/f$ and, since it
	is further contractible, the claim follows. Finally, the contractibility of $j/f$ in the
	remaining case where $x \in A_1' \cap N$ is immediate.

	Therefore, by Proposition \ref{prop:localpara}, the value of $\F$ at $(A,\emptyset)$ is
	also given by right Kan extension of $\F|\Md(X,N)$ so that, in conclusion, $\F$ is a right Kan
	extension of $\F|\Md(X,N)$. 

	The converse implication $\ref{main:2} \Rightarrow \ref{main:1}$ is a consequence of the above
	argumentation and the converse implication $(2) \Rightarrow (1)$ of Proposition
	\ref{prop:localpara}.  
\end{proof}

\begin{rem}\label{rem:cycliccosheaves}
	In the context of Theorem \ref{thm:main}, let
	\[
		\Fun^{\sharp}(\Mdp(X,N),\D) \subset \Fun(\Mdp(X,N),\D)
	\]
	denote the full subcategory consisting of those functors that satisfy the equivalent conditions
	\ref{main:1} and \ref{main:2}. By arguments analogous to the ones in the proof of Proposition
	\ref{prop:milnminimal}, it can be shown that the objects of $\Fun^{\sharp}(\Mdp(X,N),\D)$ are
	precisely the {\em cyclic cosheaves}, namely functors $\F: \Mdp(X,N) \to \D$ such that
	\begin{enumerate}
		\item $\F$ maps objects of the form
			\[
				\disk{1b}{1},\disk{0}{1} \in \Mdp(X,N) 
			\]
			to zero objects in $\D$,
		\item for every object $(A,A')$ and for every collared cut of $(A,A')$,
			$\F$ maps the associated square \eqref{eq:cutsquare} to a coCartesian square
			in $\D$.
	\end{enumerate}
\end{rem}

\begin{cor}\label{cor:milnorparam} Let $(X,N)$ be a stratified surface and $\D$ a stable $\infty$-category. Then there are
	equivalences of stable $\infty$-categories
	\[
		\begin{tikzcd}
			&\ar[swap]{dl}{\rho_1} \Fun^{\sharp}(\Mdp(X,N), \D)
			\ar{dr}{\rho_2} & \\
			\Fun^{\sharp}(\Md(X,N), \D) & & \Fun(S(X,N), \D)
		\end{tikzcd}
	\]
	given by restriction along $\Md(X,N) \subset \Mdp(X,N)$ and $S(X,N) \subset
	\Mdp(X,N)$, respectively. In particular, via the equivalence $\Sh(X,N;\D) \simeq
	\Fun(S(X,N)^{\op},\D)$ from Corollary \ref{cor:constructible}, there is a canonical equivalence
	\[
		\Sh(X,N;\D) \simeq \Fun^{\sharp}(\Md(X,N)^{\op}, \D). 
	\]
\end{cor}
\begin{proof}
	We apply Theorem 4.3.2.15 of \cite{HTT}.
	The fact that $\rho_1$ is an equivalence is then an immediate consequence of the equivalence
	Theorem \ref{thm:main}. The functor $\rho_2$ is an equivalence by Theorem \ref{thm:main}
	combined with the observation that a functor $S(X,N) \cup \Dd(X,N) \to \D$ is a right Kan
	extension of its retriction to $S(X,N)$ if and only if $\F|\Dd(X,N) \simeq 0$, i.e., two
	successive applications of loc. cit.
\end{proof}

\begin{thm}\label{thm:milnorsheaves}
	Let $\D = \D(\A)$ be the derived $\infty$-category of an abelian category $\A$. 
	Then the equivalence
	\[
		\Sh(X,N;\D) \simeq \Fun^{\sharp}(\Md(X,N)^{\op}, \D). 
	\]
	from Corollary \ref{cor:milnorparam} restricts to an equivalence
	\[
		\PS(X,N;\A) \simeq \Fun^{\sharp}(\Md(X,N)^{\op}, \A)
	\]
	identifying perverse sheaves on $(X,N)$ with Milnor sheaves valued in $\A$. 
\end{thm}
\begin{proof}
	Under the equivalence 
	\[
		\rho: \Sh(X,N;\D) \simeq \Fun^{\sharp}(\Md(X,N)^{\op}, \D),
	\]
	the value of a Milnor sheaf $\rho(\F)$ on a Milnor disk $(A,A') \in \Md(X,N)$ is
	equivalent to the value of the corresponding constructible sheaf $\F$ on a Milnor pair $(U,U')$ where
	$U$ is a sufficiently small open disk containing the closed disk $A$ and $U'$ is
	a union of open disks where each disk contains one of the intervals comprising $A'$. 
	Thus, by Proposition \ref{prop:purity}, a constructible sheaf $\F \in \Sh(X,N;\D)$ is
	perverse if and only if $\rho(\F)$ takes values in $\A$. Further, since all horizontal
	morphisms that arise in the Milnor sheaf conditions admit sections, they are Cartesian in
	$\A$ if and only if they are Cartesian in $\D(\A)$. This proves the claim. 
\end{proof}

\begin{cor}\label{cor:milnorsheaves}
	Let $\A$ be an abelian category. Then we have an natural equivalence
	\[
		\Fun^{\sharp}(\Md(X,N)^{\op}, \A) \simeq \Fun^{\sharp}(M(X,N)^{\op}, \A)
	\]
	where $M(X,N) \subset \Lambda(X,N)$ is the full subcategory of the (undirected) paracyclic
	category of $(X,N)$ spanned by the Milnor disks.
\end{cor}
\begin{proof} 
	This follows from the observation  that $\Md(X,N) \to M(X,N)$ is a localization along the
	weak equivalences (Corollary \ref{cor:milnorlocalization}).
\end{proof}

\subsection{Verdier duality for perverse sheaves}

\begin{prop} Let $\A$ be an abelian category. Then the self-duality
	\[
		\xi: \Lambda(X,N) \lra \Lambda(X,N)^{\op}
	\]
	induces an equivalence 
	\[
		\xi^*: \Fun^{\sharp}(M(X,N)^{\op},\A) \overset{\simeq}{\lra}
		\Fun^{\sharp}(M(X,N),\A)
	\]
	between Milnor sheaves and cosheaves. 
\end{prop}
\begin{proof}
	The Milnor sheaf conditions (in terms of face maps) get swapped with the dual conditions (in
	terms of degeneracy maps), cf. the proof of Proposition \ref{prop:combcech}.
\end{proof}

\begin{rem} Suppose $\A$ is an abelian category with exact duality $\delta$. Then the
	resulting anti-equivalence $\delta \circ \xi^*$ of $\Fun^{\sharp}(M(X,N)^{\op},\A)$
	can be identified with the Verdier self-duality of $\PS(X,N;\A)$. Note that, even more
	classically, we may understand the perfect pairing between $\RG(A,A';\F)$ and $\RG(A, \partial
	A \setminus \mathring A';\F^{\vee})$ as an elementary instance of Lefschetz duality for manifolds
	with boundary.  
\end{rem}

\subsection{Paracyclic Segal objects}
\label{sec:combinatorial}

Let $\D$ be an $\infty$-category with finite colimits. A cosimplicial object $X:
\Delta \to \D$ is called a {\em Segal object}, if it satisfies the {\em Segal conditions}: for
every $n \ge 1$, the map
\begin{equation}\label{eq:segal1}
		X_1 \amalg_{X_0} \dots \amalg_{X_0} X_1 \lra X_n 
\end{equation}
induced by the inclusions $[1] \cong \{i,i+1\} \subset [n]$ is an equivalence. Equivalently, $X$ is
a Segal object if, for every $1 \le m < n$, the square
\begin{equation}\label{eq:segal2}
		\begin{tikzcd}
			X_0 \ar{r}\ar{d} & X_m \ar{d}\\
			X_{n-m} \ar{r} & X_n
		\end{tikzcd}
\end{equation}
induced by the diagram
\[
	\begin{tikzcd}
		\{m\} \ar{r} \ar{d}& \{0,1,\dots,m\} \ar{d} \\
	\{m,m+1,\dots,n\}\ar{r} & \{0,1,\dots,n\} 
	\end{tikzcd}
\]
is a pushout square in $\D$.

\begin{prop}\label{prop:localpara}
	Let $\D$ be a stable $\infty$-category, and let $\Lambdap$ be the augmented paracyclic category
	obtained from $\Lambda$ by adjoining an initial object $\emptyset$. Let $\J \subset
	\Lambdap$ denote the full subcategory spanned by $\emptyset$ and $\langle 0 \rangle$. Then
	for a functor
	\[
		\F: \Lambdap \lra \D,
	\]
	the following conditions are equivalent:
	\begin{enumerate}
		\item $\F$ is a left Kan extension of $\F|\J$.
		\item $\F|\Delta$ satisfies the Segal conditions and $\F$ is a right Kan
			extension of $\F|\Lambda$.
	\end{enumerate}
\end{prop}
\begin{proof}
	We make two observations:
	\begin{itemize}
		\item The inclusion
			\[
				\Delta \subset \Lambda
			\]
			is coinitial: the category $\Delta/\cn$ is the category of simplices of the
			simplicial object $\Hom_{\Lambda}(-,\cn)|\Delta^{\op}$ whose geometric
			realization is homeomorphic to $|\Delta^n| \times \RR$. 
		\item The inclusion
			\[
				(\Delta^+)_{\le 1}/[n] \subset (\Lambdap)_{\le 1}/\cn
			\]
			is an equivalence and hence cofinal.
	\end{itemize}
	Therefore, we have reduced the proof of Proposition \ref{prop:localpara} to the statement of
	Proposition \ref{prop:combcech} below. 
\end{proof}

\begin{prop}\label{prop:combcech}
	Let $\D$ be a stable $\infty$-category, and let $\Delta^+$ be the augmented simplex category
	obtained from $\Delta$ by adjoining an initial object $\emptyset$. Let $\J \subset
	\Delta^+$ denote the full subcategory spanned by $\emptyset$ and $[0]$. Let 
	\[
		X: \Delta^+ \lra \D
	\]
	be an augmented cosimplicial object in $\D$. Then the following conditions are equivalent:
	\begin{enumerate}
		\item $X$ is a left Kan extension of $X|\J$.
		\item $X|\Delta$ satisfies the Segal conditions and $X$ is a right Kan
			extension of $X|\Delta$.
	\end{enumerate}
\end{prop}
\begin{proof}
	(1) $\Rightarrow$ (2): Suppose that $X$ is a left Kan extension of $X|\J$. The pointwise formula for Kan
	extensions implies that, for every $n \ge 1$, $X_n$ is a colimit of the restriction of $X$ to
	$\J/[n]$. We define a functor $f$ from the poset 
	\[
		\I = \{0,1\} \leftarrow \{1\} \rightarrow \{1,2\} \leftarrow \{2\} \rightarrow \dots
		\leftarrow \{n-1\}
		\rightarrow \{n-1,n\}
	\]
	to $(\Set_{\Delta})_{/\N(\J/[n])}$ sending a set $I$ to the nerve of the subposet of $\J/[n]$
	consisting of those maps with image contained in $I$. By \cite[4.2.3.10]{HTT}, we may
	compute the colimit of $X|(\J/[n])$ as the colimit of the diagram
	\[
		\I \to \D, I \mapsto \colim \F|f(I)
	\]
	yielding the $n$th Segal condition. 

	To show that $X$ is a right Kan extension of $X|\Delta$ first note that, since
	$X|\Delta$ is Segal, by Lemma \ref{lem:deloop} below, it is a right Kan extension of
	$X|(\Delta_{\le 1})$. Therefore, it suffices to show that $X$ is a right Kan extension
	of $X|(\Delta_{\le 1})$. By the pointwise criterion, this is equivalent to the
	statement that $X$ maps the diagram
	\[
		\begin{tikzcd}
			\emptyset \ar{r}\ar{d} & \{0\} \ar{d}\\
			\{1\} \ar{r} & \{0,1\}
		\end{tikzcd}
	\]
	in $\Delta^+$ to a pullback square in $\D$. But, since $X$ is a left Kan extension of
	$\J$, it maps the square to a pushout square in $\D$, so that the statement follows
	since $\D$ is stable. 

	(2) $\Rightarrow$ (1): Suppose that $X|\Delta$ satisfies the Segal conditions. Then,
	by the above arguments, $X$ is left Kan extension of $X|\J$ if and only if it maps
	the square 
	\[
		\begin{tikzcd}
			\emptyset \ar{r}\ar{d} & \{0\} \ar{d}\\
			\{1\} \ar{r} & \{0,1\}
		\end{tikzcd}
	\]
	to a pushout square. But, by the last part of the argument of (1) $\Rightarrow$ (2), this is
	equivalent to $X$ being a right Kan extension of $X|\Delta$, concluding the argument.
\end{proof}

\begin{lem}\label{lem:deloop} Let $\D$ be a stable $\infty$-category, and let $Y: \Delta \to \D$ be a
	cosimplical object in $\D$. Let $\Delta_{\le 1} \subset \Delta$ denote the full
	subcategory spanned by the objects $[0]$ and $[1]$. Then $Y$ is a Segal object if and only
	if $Y$ is a right Kan extension of its restriction $Y|(\Delta_{\le 1})$. 
\end{lem}
\begin{proof}
	Suppose $Y$ satisfies the Segal conditions. We need to verify that, for every $n \ge 2$,
	$Y_n$ is a limit of $Y|([n]/\Delta_{\le 1})$. We prove the statement by induction on
	$n$ starting with $n=2$. Consider the commutative diagram in $\Delta$ depicted by
	\[
		\begin{tikzcd}
			\{1\} \ar{r}\ar{d} & \{0,1\}\ar{d} \ar{r} & \{1\} \ar{d}\\
			\{1,2\} \ar{r}\ar{d} & \{0,1,2\}\ar{d} \ar{r} & \{1,2\} \ar{d}\\
			\{1\} \ar{r} & \{0,1\} \ar{r} & \{1\}.
		\end{tikzcd}
	\]
	Since all horizontal and vertical composites yields the identity on the respective object,
	$Y$ maps all $2x1$ and $1x2$ rectangles to biCartesian squares in $\D$. The Segal condition
	for $n=2$ is equivalent to $Y$ mapping the top left square to a pushout, and hence
	biCartesian, square. The pointwise condition on $Y$ being a right Kan extension is equivalent
	to $Y$ mapping the bottom right square to a pullback, hence biCartesian, square. But, by the
	two-out-of-three property for biCartesian squares (\cite[4.4.2.1]{HTT}), the top left
	square is biCartesian if and only if the bottom right square is biCartesian. Therefore, for
	$n=2$, the Segal condition is equivalent to the corresponding pointwise Kan extension
	criterion for $Y_2$.

	Assume that the $n$th Segal condition is equivalent to the pointwise Kan extension formula
	for $Y_n$. Consider the diagram 
	\[
		\begin{tikzcd}
			\{n\} \ar{r}\ar{d} & \{0,1,\dots,n\}\ar{d} \ar{r} & \{n\} \ar{d}\\
			\{n,n+1\} \ar{r}\ar{d} & \{0,1,\dots,n+1\}\ar{d} \ar{r} & \{n,n+1\} \ar{d}\\
			\{n\} \ar{r} & \{0,1,\dots,n\} \ar{r} & \{n\}
		\end{tikzcd}
	\]
	in $\Delta$. A similar argument to the case $n=2$ implies the equivalence of the $(n+1)$st
	Segal condition and the pointwise Kan extension formula for $Y_{n+1}$, concluding the argument.
\end{proof}

\section{Perverse sheaves on $(\CC, \{0\})$}\label{sec:per-D-paracyc}

In this chapter we consider the classical cae when $X = \CC$ is the complex plane and $N=\{0\}$. The
corresponding category of perverse sheaves is well known but our approach provides a new point of
view on it which will be crucial in the further work on categorical generalization to perverse
schobers. In what follows we compare the two approaches and diskuss the concepts they lead to. 

\subsection{The classical $(\Phi, \Psi)$-description} 

Let $\A$ be a Grothendieck abelian category. The following result goes back to the
early days of the theory of perverse sheaves \cite{beil-gluing, GGM}. It was
originally formulated for perverse sheaves of vector spaces but the proof
given in \cite{GGM} generalizes easily to the $\A$-valued case. 

\begin{prop}\label{prop:phi-psi}
	The category $\PS(\CC,\{0\};\A)$ is equivalent to the category of data $(\Phi, \Psi, a,b)$ where
	$\Phi$ and $\Psi$ are objects of $\A$ and
	\begin{equation}\label{eq:data}
		 \begin{tikzcd}
			 \Phi \ar[bend left=20]{r}{a} & \Psi \ar[bend left=20]{l}{b}
		 \end{tikzcd}
	 \end{equation}
	 are morphisms such that the {\em monodromy transformations}
	 \begin{equation}\label{eq:conditions}
		 T_\Psi: = \Id_\Psi - ab \quad \text{and} \quad T_\Phi: \Id_\Phi-ba \quad 
	 \end{equation}
	 are isomorphisms. In fact, $T_\Psi$ being an isomorphism is equivalent to $T_\Phi$ being an
	 isomorphism. 
	 \qed
\end{prop}

For a given perverse sheaf $F \in \PS(\CC,\{0\};\A)$ the corresponding objects $\Phi=\Phi(F)$ and
$\Psi=\Psi(F)$ are called the objects of {\em vanishing} and {\em nearby cycles} of $F$.  We will
now describe the relationship between the classification data in Proposition \ref{prop:phi-psi} and
our description of perverse sheaves as Milnor sheaves from Corollary \ref{cor:milnorsheaves}. 

\subsection{From a Milnor sheaf to vanishing and nearby cycles}
\label{subsec:milntoclass}

Let $\F: M(\CC,\{0\})^{\op} \to \A$ be a Milnor sheaf. We will explain how to most directly extract from $\F$ the
classification data \eqref{eq:data} and verify conditions \eqref{eq:conditions}.
First, we define 
\[
	\Psi = \F(A,A') \quad \text{where} \quad (A,A') = \disk{0}{2}
\]
is any disk that does not contain the origin $0$. Further, we set 
\[
	\Phi = \F(B,B') \quad \text{where} \quad (B,B') = \disk{1}{1} 
\]
is any disk containing $0$ in its interior. The descent conditions force rotation of $(A,A')$ by $\pi$
to be multiplication by $-1$: in the local model explained in \S \ref{sec:combinatorial}, this
automorphism corresponds to the paracyclic shift on the \v{C}ech nerve of $0 \to \Psi[1]$. The
monodromy transformation $T_{\Psi}$ is obtained by moving $(A,A')$  as a rigid body
(parallel to itself) in a circle around the origin $0 \in
\CC$. The monodromy $T_{\Phi}$ is induced by rotating $(B,B')$ by an angle of $2\pi$ around the center of the disk $B$. The map
\[
	a: \Phi \lra \Psi
\]
is obtained from the morphism in $M(\CC,\{0\})$ that is represented by a bordism of the form 
\begin{equation}\label{eq:borda}
		\begin{tikzcd}[row sep={1em,between origins}]
			& \includegraphics[width=3cm]{bordism_a.pdf} & \\
			\bdisk{0}{2} \ar{rr} & & \bdisk{1}{1}
		\end{tikzcd}
\end{equation}
while the morphism $b$ corresponds to the dual of \eqref{eq:borda}: 
\begin{equation}\label{eq:bordb}
		\begin{tikzcd}[row sep={1em,between origins}]
			& \includegraphics[width=3cm]{bordism_b.pdf} & \\
			\bdiskdual{0}{2} & &\ar{ll}  \bdiskdual{1}{1}.
		\end{tikzcd}
\end{equation}
To obtain the relations \eqref{eq:conditions}, we investigate the descent condition for
\[
	\F(C,C') \quad \text{where} \quad (C,C') = 
	\raisebox{-.4\height}{\includegraphics[scale=.3]{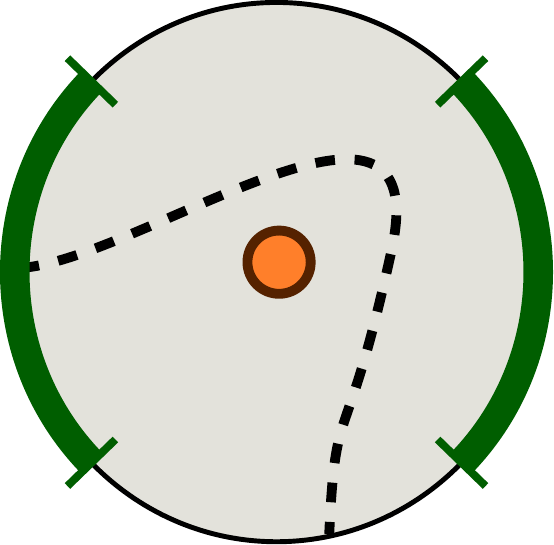}}
\]
Namely, introducing the depicted cut, the corresponding descent condition \eqref{eq:cutsquare}
provides a direct sum decomposition $\F(C,C') \cong \Phi \oplus \Psi$. We then directly observe
that, with respect to that decomposition, the transformation induced on $\F(C,C')$ by rotating
$(C,C')$ around its center by $2 \pi$,  is given by the matrix 
\begin{equation}
	Q = \left( \begin{array}{rr} T_{\Phi} & 0 \\ 0 & T_{\Psi} \end{array} \right).
\end{equation}
On the other hand, this transformation comes equipped with a square root, induced by rotating
$(C,C')$ around its center by $\pi$. A somewhat more careful analysis shows that, in terms of the above direct sum
decomposition, this transformation can be described by the matrix
\begin{equation}
	P = \left( \begin{array}{rr} -\id & b \\ -a & \id \end{array} \right).
\end{equation}
Now the relation $P^2 = Q$ implies the desired relations \eqref{eq:conditions}. Note that, in order
to extract the above data, various choices have to be made -- the advantage of the description of
$\F$ lies in the intrinsic nature of the parametrizing category $M(\CC,\{0\})$ of Milnor disks.

\subsection{The equivalence of classical and Milnor sheaf descriptions} 

In this section, we elaborate on the discussion in \S \ref{subsec:milntoclass} to provide a direct
argument for why these descriptions are equivalent. This can, of course, also be indirectly deduced by
combining our Corollary \ref{cor:milnorsheaves} and \cite{GGM}, but it is nevertheless interesting
to provide an explicit dictionary. 

\paragraph{The Milnor sheaf description.} 

\begin{prop}\label{prop:PSD-Milnor}
        Let $\A$ be an abelian category. Let $\DD \subset \CC$ be the 
	unit disk. Then the restriction along $\Lambda_{\DD} \subset M^+(\CC,\{0\})$ induces a
	fully faithful functor
	\[
		\Fun^{\sharp}(M^+(\CC,\{0\})^{\op},\A) \overset{\simeq}{\lra}
		\Fun(\Lambda_{\DD}^{\op}, \A)
	\]
	with essential image given by those paracyclic objects whose underlying simplicial
	object satisfies the Segal conditions. 
\end{prop}
\begin{proof} 
	For notational convenience, we replace $\A$ by $\A^{\op}$ and prove the cosheaf version of
	the statement. 
	By Corollary \ref{cor:milnorlocalization}, the category $M^+$ may be described as the
	localization of its directed variant $\Md^+$. In the statement of the proposition, we may
	therefore replace the category $\Fun^{\sharp}(M^+(\CC,\{0\}),\A)$ by the equivalent
	category $\Fun^{\sharp}(\Md^+(\CC,\{0\}),\A)$ where here, the superscript $\sharp$ also
	contains the requirement that weak equivalences be sent to isomorphisms in $\A$.  
	We now focus on the following collections of objects of $\Md^+$ (and the subcategories they
	span):
	\begin{itemize}
		\item $\Md_0$: all objects $(A,A')$ where $0 \in A \setminus A'$,
		\item $\Md_1$: $\Md_0$ together with all objects of the form
			\[
				(A,A') = \disk{1b}{1}
			\]
		\item $\Md_2$: $\Md_1$ together with all objects $(A,A')$ such that $0 \in A'$, 
		\item $\Md_3$: $\Md_2$ together with all objects of the form
			\[
				(A,A') = \disk{0}{1}
			\]
	\end{itemize}
	The fact that the restriction functor of the proposition is an equivalence now follows from
	the statement that the functors $\F \in \Fun^{\sharp}(\Md^+(\CC,\{0\}),\A)$ can be characterized
	by the following conditions:
	\begin{enumerate}
		\item The paracyclic object $\F|\Md_0 \simeq \Lambda_{\DD}$ satisfies the Segal conditions.
		\item $\F$ is obtained from its restriction to $\Md_0$ via a sequence of left (resp.
			right) Kan extensions as
			indicated in 
			\[
				\begin{tikzcd}
					\Md_0 \ar{r}{\text{right}} & \Md_1 \ar{r}{\text{left}} & \Md_2
					\ar{r}{\text{left}} & \Md_3 \ar{r}{\text{right}} & \Md^+(\CC,\{0\})
				\end{tikzcd}
			\]
			The details are left to the reader.
	\end{enumerate}
\end{proof}

\begin{cor}
	\label{cor:local}
		The category of Milnor sheaves on $(\CC,\{0\})$ with values in $\A$, and therefore the category of
		perverse sheaves on $(\CC,\{0\})$, is equivalent to the category $\A_{\Lambda_{\infty}}^{\on{Seg}}$
		of paracyclic objects in $\A$ whose underlying simplicial object satisfies the Segal
		conditions. 
\end{cor}

In what follows, we provide the relation to the more traditional classification of Proposition
\ref{prop:phi-psi} by means of a paracyclic nerve construction which can also be regarded as a
special instance of a duplicial variant of the Dold--Kan correspondence established in
\cite{DK}. 

\paragraph{Paracyclic structures on the nerve of a Picard groupoid.} 
To compare Propositions \ref{prop:phi-psi} and \ref{prop:PSD-Milnor}
in a direct way, we assume for simplicity that $\A=\Ab$
is the category of abelian groups. It is classical that a simplicial set 
is Segal if and only
if it is isomorphic to the nerve of a small category. The categories relevant for us are
are  Picard groupoids of a particular type. 

We recall (cf. \cite{deligne-SGA}) that a  {\em Picard groupoid} is a symmetric monoidal
category $(\P, \otimes, \one)$ in which each object is invertible with respect to $\otimes$
and each morphism is invertible with respect to the composition.

\begin{ex}\label{ex:2-term-picard}
Let $E^\bullet$ be a 2-term complex of abelian groups situated in degrees $[-1,0]$.
It will be suggestive for us to write $E^\bullet$ as $\{\Psi\buildrel b\over\to \Phi\}$
with $\Phi$ in degree $0$ and $\Psi$ in degree $(-1)$. To such a datum one
associates a Picard groupoid $[E^\bullet] = [\Psi\buildrel b\over\to \Phi]$ with: 
    \[
    \begin{gathered}
    \Ob\,  [\Psi\buildrel b\over\to \Phi] \,\,=\,\,\Phi; 
    \\
    \Hom(\phi', \phi) \,\,=\,\,\bigl\{ \psi\in\Psi\, \bigl| \, b(\psi) = \phi-\phi' \bigr\}. 
    \end{gathered}
    \]
 Composition of morphisms is given by addition of the $\psi$.   
 The tensor product of objects is given by addition of the $\phi$. 
 We note that the set of all morphisms in $[\Psi\buildrel b\over\to \Phi]$ 
 (i.e., the disjoint union of all  the $ \Hom(\phi, \phi') $) can be described as
 \[
 \Mor \,  [\Psi\buildrel b\over\to \Phi] \,\,=\,\, \Psi \oplus \Phi, 
 \]
 with the source and target maps $s,t: \Mor\to\Ob$ given by
 \begin{equation}
	 \label{[b]-st}
	 s(\psi,\phi)= \phi - b(\psi) , \quad t(\psi,\phi) = \phi. 
 \end{equation}
See \cite{deligne-SGA} for more details.
\end{ex}

The nerve $N [\Psi\buildrel b\over\to \Phi]$ is a simplicial abelian group
with $n$-simplices 
\begin{equation}
	\label{eq:N_n}
	N_n   [\Psi\buildrel b\over\to \Phi] \,\,=\,\,  \Psi^{\oplus n}\oplus \Phi. 
\end{equation}
Passing from a 2-term complex $\{  \Psi\buildrel b\over\to \Phi\}$ to the simplicial object
$N [\Psi\buildrel b\over\to \Phi]$
is a particular case of the Dold-Kan correspondence  between non-positively graded cochain complexes
of abelian groups and simplicial abelian groups, see \S \ref{subsection:para-DK} below. 

 \begin{prop}\label{prop:perv-para}
Let $b: \Psi \to \Phi$ be a  morphism of abelian groups. Then the following are in bijection:
\begin{itemize}
\item[(i)] Morphisms $a: \Phi\to \Psi$ such that the data
 $(\Phi, \Psi, a,b)$ satisfy the conditions of Proposition \ref{prop:phi-psi}, i.e., define a perverse sheaf $F\in\PS(D,0;\Ab)$.
 
  \item[(ii)] Extensions of the structure of a simplicial abelian on
  $ N [\Psi\buildrel b\over\to \Phi]$ to 
 that of a paracyclic abelian group,
 i.e.,  systems of  automorphisms $t_n \in\Aut\bigl(N_n  [\Psi\buildrel b\over\to \Phi]\bigr)$ 
 (actions of  the $\tau_n\in \Aut_{\Lambda_\oo}\cn $) satisfying the
 relations dual to those imposed in Definition \ref{def:paracyc}(a). 
\end{itemize} 
Under this bijection,   the automorphism $t_n^{n+1}$  corresponds, via the identification \eqref{eq:N_n},  to
the  direct sum $T_\Psi^{\oplus n}\oplus T_\Phi$ of the monodromies. 
\end{prop}

\noindent{\sl Proof:} Explicitly, the convention \eqref{[b]-st} on labelling the source and target of
a morphism implies that the 
simplicial face  and degeneracy operators  on
 $N_n  [\Psi\buildrel b\over\to \Phi]$ are given by
\[
 \del_i:  \Psi^{\oplus n}\oplus \Phi \to  \Psi^{\oplus(n-1)} \oplus \Phi, \quad 
(\psi_1,\cdots, \psi_n;\phi ) \mapsto
\begin{cases}
( \psi_2,\cdots, \psi_n; \phi), &  i=0, 
\\
( \psi_1, \cdots, \psi_i+\psi_{i+1}, \cdots\psi_n;\phi), & 1\leq i <  n;
\\
(\psi_1, \cdots, \psi_{n-1}; \phi-b(\psi_n)), & i=n; 
 \end{cases}   
 \]
 \[
s_j: \Psi^{\oplus n} \oplus \Phi \to \Psi^{\oplus (n+1)}\oplus \Phi, \quad 
(\psi_1,\cdots, \psi_n; \phi) \mapsto
 (\psi_1, \cdots, \psi_{j-1},0, \psi_{j+1}, \cdots \psi_n; \phi),  \,\, j=0,\cdots, n.
 \]
Now, let $a: \Phi\to\Psi$ be as in (a).  For each $n\geq 0$,  define an endomorphism 
$t_n$ of $\Psi^{\oplus n}\oplus\Phi$ by
\be\label{t_n-explicit}
t_n(\psi_1,\cdots, \psi_n, \phi) \,=\, \bigl(-\psi_1-\cdots - \psi_n + a(\phi), \psi_2, \cdots, 
\psi_{n-1};  \phi-b(\psi_n)\bigr). 
\ee
We then check directly that the relations dual to those of Definition \ref{def:paracyc}(a)
are satisfied. We also check that $t_n^{n+1} = T_\Psi^{\oplus n}\oplus T_\Phi$
which implies that $t_n$ is invertible. 


Conversely, suppose we have automorphisms $t_n$ as in (b).  The relation $\del_0t_n=\del_n$
implies that $t_n$ has the form
\[
t_n(\psi_n, \cdots, \psi_n; \phi) \,\,=\,\,\biggl( -\sum_{i=1}^n x^{(n)}_i(\psi_i) + a_n(\phi) ; \, \psi_2, \cdots, \psi_{n-1}, \phi-b(\psi_n)\biggr)
\]
for some linear maps $x^{(n)}_i: \Psi\to\Psi$ and $a_n: \Phi\to\Psi$. We denote $a_1=a$ and will prove that 
\be\label{eq:x-i-n}
x^{(n)}_i=\Id, \quad a_n=a, \quad  \forall \, n, \,\, i=1,\cdots, n,
\ee
 i.e., that all the $t_n$ are given by the formula \eqref{t_n-explicit}. This will imply the invertibilty of
$T_\Psi=\Id-ab$ and $T_\Phi=\Id-ba$ by identifying $t_n^{n+1}$ as above. 

The  equalities \eqref{eq:x-i-n} are proved recursively, using the relations of $\Lambda^\infty$. To start,  the relation $\del_1t_2=t_1\del_0$ implies that
\[
\del_1t_2(\psi_1, \psi_2; \phi) \,=\, (-x^{(2)}_1\psi_1 - x^{(2)}_2\psi_2 + a_2\phi + \psi_1; \, \phi-b\psi_2)
\]
 is equal to
 \[
 t_1\del_0(\psi_1, \psi_2; \phi) \,=\, (-x^{(1)}_1\psi_2+a\phi;\,  \phi-b\psi_2),
 \]
which entails
\[
x^{(2)}_2 = x^{(1)}_1, \,\,\, x^{(2)}_1=\Id. 
\]
The relation $\del_2t_2=t_1\del_1$ then implies that 
\[
\del_1 t_2(\psi_1, \psi_2;\psi) \,=\, (-x^{(2)}_1\psi_1-x^{(2)}_2\psi_2 + a_2\phi; \, \phi-b\psi_2-b\psi_1)
\]
  is equal to
  \[
  t_1 \del_1(\psi_1, \psi_2;\psi) \,=\, (-x^{(1)}_1\psi_1 -x^{(1)}_1 \psi_2 + a(\phi);\,  \phi-b\psi_2-b\psi_1),
  \]
  which entails
  \[
  x^{(2)}_1 = x^{(2)}_2 = x^{(1)}_1, \,\,\, a_2=a.
  \]
  Since we already know that $x^{(2)}_1=\Id$, we see that $x^{(2)}_2=x^{(1)}_1=\Id$. Continuing like this, we prove 
  \eqref{eq:x-i-n}. 
   \qed
   
   \begin{rem}  
   One can consider paracyclic structures on the nerves of more general Picard groupoids, not necessarily those corresponding 
   to 2-term complexes. It would be interesting to
   understand the relation of such structures to perverse sheaf-like objects.
   We recall \cite{johnson-osorno}  that Picard groupoids correspond to spectra (stable homotopy types in the sense of
   homotopy topology) which have only two nontrivial homotopy groups in
   adjacent degrees, say only $\pi_0$ and $\pi_1$ or only $\pi_1$ and $\pi_2$.
   
   More generally, unstable homotopy types with only $\pi_1$ and $\pi_2$ nontrivial,
   are described by {\em crossed modules}, see, e.g.,  \cite{noohi}, which are 2-term complexes of
   possibly nonabelian groups
   \[
   G^\bullet \,=\,\bigl\{ G^{-1} \buildrel\del\over\lra G^0\bigr\}
   \]
 with a compatible action of $G^0$ on $G^{-1}$.  A crossed module $G^\bullet$
 gives rise to a non-abelian Picard groupoid (also known as a {\em $2$-group})
 $[G^\bullet]$, defined similarly to Example \ref{ex:2-term-picard}. One can
 ask about the meaning of paracyclic structures on the nerve of $[G^\bullet]$
 and the possibility of defining perverse sheaves of nonabelian groups
 in one complex dimension. 
   
   \end{rem}

\subsection{Relation to the duplicial Dold-Kan correspondence}
   \label{subsection:para-DK}
   
   \paragraph{The classical Dold-Kan.} 
   Let $\A$ be an abelian category and $\Co^{\leq 0}(\A)$ be the (abelian) category of cochain complexes over $\A$ situated
   in degrees $\leq 0$. As usual, by $\A_\Delta$ we denote the category of simplicial objects of $\A$. 
   The {\em Dold-Kan correspondence}, see, e.g., \cite{goerss-jardine},  is the pair of mutually quasi-inverse (in particular, adjoint) equivalences
   of categories
   \[
   \CDK: \A_\Delta \buildrel \sim\over \longleftrightarrow \C^{\leq 0}(\A): \NDK, 
   \]
  defined as follows. The functor $\CDK$, called the {\em normalized chain complex functor}, takes $A_\bullet\in\A_\Delta$
  to the complex $\CDK(A_\bullet)$ with
  \[
  \CDK^{-n}(A_\bullet) \,=\,\bigcap_{i=1}^{n}\,  \Ker \{\del_i: A_n\lra A_{n-1}\}, \quad n\geq 0,
  \] 
 with the differential given by the remaining face map $\del_0$.   
 
 The functor $\NDK$ called the {\em Dold-Kan nerve}, takes a complex $(E^\bullet, d_E)\in \Co^{\leq 0}(\A)$ into the simplicial
 object $\NDK(E^\bullet)$ with
 \[
 \NDK(E^\bullet)_n \,=\, Z^0(\Delta^n, E^\bullet), 
 \]
the object of degree $0$ simplicial (hyper)cocycles on $\Delta^n$ with values in $E^\bullet$. That is, denoting $\Delta^n_m$
the set of $m$-simplices of $\Delta^n$, 
\[
 Z^0(\Delta^n, E^\bullet) \,\subset \, \prod_{m\geq 0} (E^{-m})^{\Delta^n_m}
\]  
 is given by the following ``end'' condition: the action of the morphism induced by each $d_E: E^{-m} \to E^{-m+1}$ is equal to the
 action of the morphism induced by $\sum (-1)^i\del_i: \Delta^n_{m+1} \to\Delta^n_m$.   
 
 \begin{exas}
 (a) Let $\A=\Ab$. An element of $ Z^0(\Delta^n, E^\bullet)$ is in this case a rule $\gamma$ associating:
 \begin{enumerate}
 \item[(0)] To each vertex $e_i$, $0\leq i\leq n$,  of $\Delta^n$,  an element $\gamma_i\in E^0$.
 
 \item[(1)]  To each edge (possibly degenerate) $e_{ij}$, $0\leq i\leq j\leq n$,  of $\Delta^n$, an element $\gamma_{ij}\in E^{-1}$ so that
 $d_E(\gamma_{ij}) = e_j-e_i$.
 
 \item[(2)] To each $2$-face  (possibly degenerate)   $e_{ijk}$, $0\leq i\leq j \leq k\leq n$,  of $\Delta^n$, an element $\gamma_{ijk}\in E^{-2}$
 so that $d_E(\gamma_{ijk}) = \gamma_{jk} - \gamma_{ik} + \gamma_{ij}$. 
 
 \item[($\cdots$)] And so on. 
 \end{enumerate}
 
 (b) In particular, if $E^\bullet = \{E^{-1}\to E^0\}$ is a 2-term complex of abelian groups, then $\NDK(E^\bullet) = \N[E^\bullet]$
 is the usual nerve of the Picard groupoid $[E^\bullet]$. 
 \end{exas}
 
 Proposition \ref{prop:perv-para} extends verbatim to the following.
 
 \begin{prop}\label{prop:para-PS-A}
 Let $\A$ be any abelian category and $b: \Psi \to\Phi$ be a morphism in $\A$. Then the following are in bijection:
 
\begin{itemize}
 \item[(i)]  
 Morphisms $a: \Phi\to \Psi$ such that the data
 $(\Phi, \Psi, a,b)$ satisfy the conditions of Proposition \ref{prop:phi-psi}, i.e., define a perverse sheaf $F\in\PS(D,0;\A)$.
 
  \item[(ii)] Extensions of the simplicial object structure on $\NDK\{\Psi\buildrel b\over\to\Phi\}$ to a structure of a paracyclic object.
  \qed
 \end{itemize}

 \end{prop}

 \paragraph{The duplicial Dwyer-Kan correspodence.}
 Let $\A$ be an abelian category and $E^\bullet\in \Co^{\leq 0}(\A)$. Proposition \ref{prop:para-PS-A} leads to the
 following question: what is the meaning of a paracyclic structure on $\NDK(E^\bullet)$ extending the given
 simplicial structure? An answerto that can be given by the results of Dwyer-Kan \cite{DK} which we recall.  
 
 
 We define the {\em duplex category}  $\Xi$ to have objects $\cn$, $n \in \NN:= \ZZ_{\geq 0} $. A
morphism from $\cm$ to $\cn$ consists of a weakly monotone map $f: \NN \to \NN$ satisfying
te following periodicity condition: for all
$i \in \NN$,  we have $f(i + m +1) = f(i) + n + 1$. The simplex category $\Delta$ is naturally a subcategory
of $\Xi$ obtained by restricting to those morphisms between $\cm$ and $\cn$ that map the interval
$[0,m]$ to $[0,n]$. A {\em duplicial object} in a category $\C$ is a functor $\Xi^{\op} \to \C$. 

 
 We recall \cite {elmendorf} that the paracyclic category $\Lambda_\oo$ can be defined in a very similar
 way, except we consider weakly mononote maps $f: \ZZ\to\ZZ$ (instead of $\NN\to\NN$) satisfying the
 same periodicity condition. In particular, the shift map
  \[
	\tau_n: \cn \lra \cn,\quad  i \mapsto i+1 
\] 
 is invertible as en element of $\Hom_{\Lambda_\oo}(\cn, \cn)$ (with $i$ running in $\ZZ$) but is not
 invertible as an element of   $\Hom_{\Xi}(\cn, \cn)$ (with $i$ running in $\NN$). In fact, comparing 
 \cite{DK} and \cite{elmendorf}, leads to the following.
 
 \begin{prop}
 $\Lambda_\oo\simeq \Xi[\tau_n^{-1}|\, n\geq 0] $ is identified with the localization of $\Xi$ with respect to 
  the morphisms $\tau_n$, $n\geq 0$.\qed
 \end{prop} 
 
 In fact, the powers $\tau_n^{n+1}$ forming a central system (a natural transformation from 
 $\Id_\Xi$ to itself),
it is easy to see that the $\oo$-categorical localization $ \Xi[\tau_n^{-1}|\, n\geq 0]_\oo$ is also identified with $\Lambda_\oo$. 
In particular, $\Xi$, like $\Lambda_\oo$, is generated by the coface and codegeneracy morphisms
\[
\begin{gathered}
\delta^{(n)}_i: \langle n-1\rangle \to \cn, \quad (n\geq 1, \,\, 0\leq i\leq n),
\\
\sigma^{(n)}_i: \cn \to \langle n-1\rangle, \quad (n\geq 1, \,\, 0\leq i\leq n),
\end{gathered}
\]
satisfying the same quadratic relations as in $\Lambda_\oo$. The morphisms
\[
\delta^{(n)}_i, \,\, 0\leq i\leq n, \quad \sigma^{(n)}_i,\,\, 0\leq i\leq n-1, 
\]
generate the simplex category $\Delta\subset\Xi$. The shift map is expressed as $\tau_n=\delta^{(n-1}_0 \sigma^{(n)}_{n-1}$. 
Accordingly, a duplicial object $Y_\bullet$ in a category $\C$
 can be identified with a sequence of objects $X_0, X_1, \dots$
equipped with face and degeneracy maps
\[
\begin{gathered} 
	\del_i: X_n \lra X_{n-1}\quad \text{($n \ge 1$, $0 \le i \le n$)}, 
 \\
	s_i: X_{n-1} \lra X_{n}\quad \text{($n \ge 1$, $0 \le i \le n$)}
	\end{gathered}
\]
subject to   relations dual to those among the $\delta^{(n)}_i, \sigma^{(n)}_i$. The action of $\tau_n$ is then
$t_n = s_{n+1}\del_0: Y_n\to Y_n$. A paracyclic object is a duplicial object such that all the $t_n$ are isomorphisms.


Following Dwyer-Kan, we  call a {\em connective ducomplex} in $\A$ a diagram
\[
\xymatrix{
\cdots \ar@/^1ex/[r]^{\delta}  & \ar@/^1ex/[l]^{d} B^{-2}  \ar@/^1ex/[r]^{\delta}  & \ar@/^1ex/[l]^{d} B^{-1}
 \ar@/^1ex/[r]^{\delta}  & \ar@/^1ex/[l]^{d} B^0
 }
\]
satisfying $d^2=0$, $\delta^2=0$ and no further relations. We denote $\DC^{\leq 0}(\A)$ the category
of connective ducoplexes in $\A$.

\begin{thm}[(Dwyer-Kan)]\label{thm:dwyer-kan-duplex}
(a) 	There is an equivalence of categories
		\[
		\A_{\Xi} \overset{\simeq}{\lra} \DC^{\leq 0}(\A) 
		\]
	given by associating to a duplicial abelian group $A_{\bullet}$ the ducomplex $B^{\bullet}$
	with
	\begin{align*}
		B^{-n}  & = \bigcap_{ i=1}^n \,  \Ker\{\del_i: A_n\to A_{n-1}\},  \quad n\geq 0,
		\\ 
		d & =  \del_0: B^{-n} \to B^{-n+1}, \\
		\delta & = \sum_{i=0}^{n} (-1)^i s_i : B^{-n+1} \to B^{-n}.
	\end{align*}
	
	
	(b) Under this equivalence, paracyclic objects correspond to ducomplexes satisfying:
	\[
	\Id_{B^{-n}} + (-1)^n (d\delta-\delta d): B^{-n} \lra B^{-n} \quad \text {is invertible for any } n\geq 0. 
	\]
\end{thm}

\begin{proof} Part (a) is Theorem 3.5 of \cite{DK}. Part (b) follows from the interpretation of $t_n^{n+1}$
	in terms of ducomplexes given in Proposition 6.5 of \cite{DK}.\end{proof}

The equivalence between the descriptions of the category of perverse sheaves on $(\CC,\{0\})$ from 
Proposition \ref{prop:PSD-Milnor} and Proposition \ref{prop:phi-psi}, respectively, is then a
consequence of restricting the equivalence from Theorem \ref{thm:dwyer-kan-duplex} to paracyclic Segal objects. 

\appendix

\section{$\oo$-categorical preliminaries}\label{subsec:inf-cat-prep}
 
\subsection{Generalities on $\oo$-categories} 

In the rest of the paper we will use freely the language of $\oo$-categories \cite{HTT}. The
following is intended to fix the terminology and notation and to recall the main tools that will be
used. 
 
We denote by $\sSet$ the category of simplicial sets. For a simplicial set $S=(S_n)_{n\geq 0}$ we
denote
\[
	\partial_i: S_n\to S_{n-1}, \,\,\, 0\leq i\leq n
\]
the simplicial face maps. By $\Delta^n\in\sSet$ we denote the standard $n$-simplex. 

Following \cite{HTT} we will use the term {\em $\infty$-category} for a weak Kan complex. Thus an
$\infty$-category $\C$ is a simplicial set $(\C_n)_{n\geq 0}$ satisfying the lifting condition for
intermediate horns $\Lambda^n_i \subset \Delta^n$, $0<i<n$. 
Any ordinary category can be considered as an $\infty$-category by passing to the nerve. 

Each $\infty$-category $\C$ contains the maximal Kan subcomplex $\C^{\on{Kan}}\subset\C$, which is
can be interpreted as ``the $\infty$-groupoid of equivalences in $\C$".

We follow the usual notation and terminology: $0$-simplices of $\C$ are called {\em objects}, and we
denote $\Ob(\C)=\C_0$, while $1$-simplices are called {\em morphisms} and we denote $\Mor(\C)=\C_1$.
For any two $x,y\in\Ob(\C)$ we denote $\Hom_\C(x,y)\subset\Mor(\C)$ the set of 1-simplices $f$ such
that $\partial_1(f)=x$ and $\partial_0(f)=y$. 

To any $\infty$-category $\C$ one associates its {\em homotopy category} $\Ho(\C)$ which is an
ordinary category with the set of objects $\Ob(\C)$ and $\Hom_{\Ho(\C)}(x,y)$ defined as the
quotient of $\Hom_\C(x,y)$ by the homotopy relation: $f\sim g$ if there is $\sigma\in \C_2$ with
$\partial_1(\sigma)=f$ and $\partial_2(\sigma)=g$. An {\em equivalence} in $\C$ is a morphism which
becomes an isomorphism in $\Ho(\C)$.

For an $\infty$-category $\C$ and any $x,y\in\Ob(\C)$,  the set $\Hom_\C(x,y)$ can be upgraded to a
simplicial set $\Map_\C(x,y)$  (the {\em mapping space}), in such a way as to make out of $\C$ a
category enriched in simplicial sets. See \cite{HTT} \S 1.2.2 for details.  
  
This leads to another point of view on $\infty$-categories: as categories enriched in topological
spaces (or simplicial sets). Several $\infty$-categorical concepts can be formulated in this
language.  
For example, an {\em initial object} of an $\infty$-category $\C$ is an object $0$ such that, for
each $x\in\Ob(\C)$ the space $\Map_\C(0,x)$ is contractible. 
 
\begin{ex}[(Kan simplicial sets as an $\oo$-category)]  Any Kan simplicial set is an $\oo$-category. 
 The $\oo$-category $\Sp$ of {\em spaces} is defined as the simplicial nerve of the category
 of Kan simplicial sets \cite[\S 1.2.16]{HTT}. 
\end{ex}

\subsection{Dg-categories} 
We denote by $\Ab$ the category of abelian groups and by
 $C(\Ab)$ the category of cochain complexes of abelian groups, with its standard
symmetric monoidal structure. 
By a {\em dg-category} we mean a category $\Ac$ enriched in $C(\Ab)$. For such $\Ac$ we have
the ordinary categories $Z^0(\Ac)$,  $H^0(\Ac)$ with the same objects as $\Ac$ and
\[
\Hom_{Z^0(\Ac)}(x,y) \,=\, Z^0 \Hom^\bullet_\Ac(x,y) , \quad 
\Hom_{H^0(\Ac)}(x,y) \,=\, H^0 \Hom^\bullet_\Ac(x,y). 
\]
Here $Z^0$ is the subgroup of $0$-cocycles in the $\Hom$-complex. 
 

A dg-category $\Ac$ gives an $\infty$-category 
$N_\dg(\Ac)$   known as the {\em dg-nerve} of $\Ac$. 
As a simplicial set, $N_\dg(\Ac)$ was introduced in \cite{HS}. 
For a given $n\geq 0$ the set
$N_\dg(\Ac)_n$ consists of weakly commutative $n$-simplices in $\Ac$ (called {\em Sugawara simplices}  in \cite{HS}), 
which are data
of:
\[
\begin{gathered}
x_0, \cdots, x_n\in \Ob(\Ac); 
\\
u_{ij}\in\Hom^0_\Ac(x_i, x_i), \,\,\, d(u_{ij})=0, \quad i<j; 
\\u_{ijk}\in\Hom^{-1}_\Ac(x_i, x_k), \,\,\, d(u_{ijk}) = u_{jk}u_{ij}-u_{ik}, \quad i<j<k; \\
\text{ and so on.}
\end{gathered}
\]
It was shown in \cite{HA} that $N_\dg(\Ac)$ is in fact a $\oo$-category.  By construction, we have
\[
\Ho(N_\dg(\Ac)) \,\simeq \, H^0(\Ac). 
\]

\subsection{The derived $\oo$-category of an abelian category}
\label{par:dercat}

Let $\A$ be a Grothendieck abelian category. In particular, $\A$ has enough injectives. 
Denote by $\Co(\A)$  the dg-category of all cochain complexes over $\A$.
Thus $Z^0(\Co(\A))$ is the ``usual'' category of complexes (morphsims = morphisms of complexes) and 
$H^0(\Co(\A))$ is the homotopy category.
The classical (unbounded) derived category of $\A$, denoted $\D(\A)$,  is defined as the categorical localization
of $H^0(\Co(\A))$ by the class of quasi-isomorphisms. It is a triangulated category. 

 The  {\em (unbounded) derived $\oo$-category} of $\A$, denoted $\D(\A)$,  can be  defined 
 in one of two equivalent ways, see \cite{HA} \S 1.3.5, esp. Prop. 1.3.5.16 and before. 
 \begin{itemize}
 \item[(i)] As the $\oo$-categorical localization of the  usual (abelian) category  $Z^0(C(\A))$ by the
 class of quasi-isomorphisms.
 
 \item[(ii)] As the full $\oo$-subcategory in $N_\dg(\Co(\A))$ spanned by {\em fibrant complexes}. 
 A fibrant complex is a possibly unbounded
 complex of injective objects with some additional properties, see  \cite{HA} \S 1.3.5 and \cite{spaltenstein}.
 \end{itemize}
 We have
 \[
 \Ho (\D_\oo(\A)) \,\simeq \, \D(\A). 
 \]

\subsection{Stable $\infty$-categories}

The derived $\infty$-categories from \S \ref{par:dercat} are examples of stable $\infty$-categories. 
Here, we recall the definition of a stable $\infty$-category from \cite{HA} and discuss some basic
results that we will use.
Let $\D$ be a pointed $\infty$-category and consider a square
\begin{equation}\label{eq:stablesquare}
		\begin{tikzcd}
			X \ar{r}{f} \ar{d} & Y \ar{d}{g}  \\
			0 \ar{r} & Z
		\end{tikzcd}
\end{equation}
in $\D$ where $0$ is a zero object. The square is called a {\em fiber sequence} if it is a pullback
square. In this case, the morphism $f$ is called a fiber of $g$. Dually, the square is called a
{\em cofiber sequence} if it is a pushout square. In this case, we say that $g$ is a cofiber of
$f$. The category $\D$ is called {\em stable} if
\begin{enumerate}[label=\arabic*.]
	\item every morphism admits a fiber and a cofiber,
	\item a square of the form \eqref{eq:stablesquare} is a fiber sequence if and only if it is
		a cofiber sequence.
\end{enumerate}
We collect some basic results about stable $\infty$-categories (cf. \cite{HA}):

\begin{prop}\label{prop:basic} Let $\D$ be a stable $\infty$-category. Then:
	\begin{enumerate}
		\item $\D$ admits finite limits and colimits. 
		\item The homotopy category of $\D$ admits a triangulated structure.
		\item\label{prop:basic.3} A square 
			\[
			\begin{tikzcd}
				X \ar{r} \ar{d} & Y \ar{d}  \\
				X' \ar{r} & Y' 
			\end{tikzcd}
			\]
			in $\D$ is Cartesian if and only if it is coCartesian. 
	\end{enumerate}
\end{prop}
Cartesian squares in a stable $\infty$-category (which are hence also coCartesian) will be called
biCartesian squares. The statement of Proposition \ref{prop:basic} \ref{prop:basic.3} has a useful generalization to
higher--dimensional cubes: Let $\D$ be a stable $\infty$-category, $n \ge 1$, and let
$\P(\{1,\dots,n\})$ be the poset of all subsets of the set $\{1,\dots,n\}$. Consider a 
diagram
\[
	q: \N(\P(\{1,\dots,n\})) \lra \D
\]
which, due to the apparent isomorphism $\N(\P(\{1,\dots,n\})) \cong (\Delta^1)^n$, has the shape of
an $n$-dimensional cube. Note that, we may either interpret $q$ as
\begin{enumerate}[label=\arabic*.]
	\item a cone over the diagram $q|\N(\P(\{1,\dots,n\}) \setminus \{\emptyset\})$, or
	\item a cone under the diagram $q|\N(\P(\{1,\dots,n\}) \setminus \{\{1,\dots,n\}\})$.
\end{enumerate}
If the first cone is a limit cone, then we call $q$ {\em Cartesian}, if the second cone is a colimit cone,
then we call $q$ {\em coCartesian}. We recall some results from \cite{HA}:

\begin{prop}\label{prop:cube} A cube $q$ is Cartesian if and only if it is coCartesian. 
\end{prop}
\begin{proof} \cite[1.2.4.13]{HA}
\end{proof}

We will refer to cubes in a stable $\infty$-category which are Cartesian (and hence coCartesian) as
{\em biCartesian} generalizing the above terminology in the case $n=2$. We further recall:

\begin{prop}\label{prop:cuberecursive} An $n$-cube is biCartesian if and only if the $(n-1)$-cube,
	obtained by passing to cofibers along all morphisms parallel to one coordinate axis, is
	biCartesian.  
\end{prop}
\begin{proof}\cite[1.2.4.15]{HA} 
\end{proof}
 
We also note the following immediate consequences of Proposition \ref{prop:cuberecursive}.

\begin{prop}\label{prop:cubefaces} Suppose we are given an $n$-cube $q$ with one face $f$ biCartesian.
	Then $q$ is biCartesian if and only if the face parallel to $f$ is also biCartesian.  
\end{prop}

\begin{prop}\label{prop:2out3}
	The property for cubes being biCartesian satisfies the two-out-of-three property with
	respect to pasting of cubes. 
\end{prop}

\subsection{Limits and Kan extensions} 

Let $\C$ be an $\oo$-category. As usual, a ($\oo$-) functor $F: I\to\C$ where $I$ is a (small) $\oo$-category, will
be called a {\em diagram} in $\C$. We will also use the notation $(F_i)_{i\in I}$ for such a diagram, with $F_i=F(i)$, $i\in\Ob(I)$. 
The $\oo$-categorical limit and colimit of $(F_i)_{i\in I}$ (when they exist)  will be denoted by
\[
	\varprojlim_{i\in I}{}^\C \, F_i, \quad\quad  \varinjlim_{i\in I} {}^\C\,  F_i, 
\]
or, in the functor notation, simply $\varprojlim F$, $\varinjlim F$. 
 
Let $\alpha: I\to J$ be a functor of small $\oo$-categories and $F: I\to\C$ be another functor.  In
this case we can speak about the {\em left} and {\em right Kan extensions} which are functors
\[
	\alpha_! F, \alpha_* F: J\lra \C
\]
characterized by universal properties. More precisely, the functors (when they exist)
\[
	\alpha_!: \Fun(I,\C)\to \Fun(J, \C), \,\, F\mapsto \alpha_!F, \quad  \alpha_*:  \Fun(I,\C)\to \Fun(J, \C), \,\, F\mapsto \alpha_* F
\]
are, resectively, left and right adjoints to the pullback functor
\[
	\alpha^*: \Fun(J,\C) \lra \Fun(I,\C), \quad G\mapsto \alpha^*G := G\circ \alpha.
\]
see \cite{HTT} \S 4.3.  While the general concept of adjunction in the $\oo$-categorical context is
somewhat subtle, see \cite{HTT} \S 5.2,  it implies  identifcations (weak equivalences) of mapping
spaces having the familiar shape, which in our case read:
\begin{equation}
	\label{eq:kan-adj}
	\Map_{\Fun(J,\C)} (G,\alpha_*F) \,\simeq\, \Map_{\Fun(I,\C)} (\alpha^* G, F), \quad \Map  (\alpha_!F, G) \,\simeq \,
	\Map  (F, \alpha^*G)
\end{equation}
for any $F\in \Fun(I,\C)$, $G\in\Fun(J, \C)$. 

We recall the {\em pointwise formulas} for Kan extensions which describe their values on an object
$j\in J$. More precisely, assuming the existence of all the relevant (co)limits, we have
\begin{equation}
	\label{eq:kan-pointwise}
	(\alpha_! F)(j) \,=\, \varinjlim{}_{i\in \alpha/j}^\C \, \,F(j), \quad \quad
	(\alpha_*(F)_j \,=\,  \varprojlim{}_{i\in j/\alpha }^\C  \,\,F(i),
\end{equation}
where $\alpha/j$ resp. $j/\alpha$ are the {\em overcategory} and {\em undercategory}, whose objects
are pairs 
\[
	\bigl(i\in \Ob(I), a: \alpha(i)\to j\bigr), \text{ resp. }  \bigl(i\in\Ob(I), b: j\to \alpha(i)\bigr),
\]
see \cite{HTT} \S 1.2.9.   
 
Recall, further, that $\alpha$ is called {\em $\oo$-cofinal}, resp. {\em $\oo$-coinitial}, if any
$j/\alpha$,
resp. $\alpha/j$ is contractible, (i.e., its nerve is a contractible simplicial set). If this is the case, then
for any $G; J\to\C$ we have equivalences 
\[
	\varinjlim \alpha^*G \,\simeq \, \varinjlim \, G, \quad \text{resp.} \quad 
  	\varprojlim \alpha^*G \,\simeq \, \varprojlim \, G. 
\]
in $\C$.
 
The following result \cite[4.3.2.15]{HTT} will be a fundamental tool for us to establish equivalences of $\oo$-categories.
 
\begin{prop}\label{prop:kanres}
	Let $J$ be an $\infty$-category and $I \subset J$ a full subcategory. Let $\C$ be an
	$\infty$-category with colimits let $\E \subset \Fun(I,\C)$ be a full subcategory and let $\Ec^!$  be the full
 	subcategory in $ \Fun(J,\C)$ spanned by functors of the form $\alpha_! F$ for $F$ in $\Ec$.
	Then the restriction functor $\Ec^! \to \Ec$ is an equivalence. The analogous statement
	holds for right Kan extensions if $\C$ has limits. 
\end{prop}
 
We also note, for future use, the following fact \cite[4.4.4.10]{HTT}. 
 
\begin{lem}\label{lem:conlim} Let $K$ be a weakly contractible simplicial set, and let $\C$ be an $\infty$-category.
	Let 
	\[
		F: K \to \C
	\]
	be a diagram sending every edge of $K$ to an equivalence in $\C$. Then a cone
	\[
		F^+: K^{\triangleright} \lra \C
	\]
	is a colimit cone if and only if every edge from a vertex in $K$ to the cone vertex $*$ is
	mapped to an equivalence in $\C$. In particular, for every vertex $k$ of $K$, the value
	$F(k)$ is a colimit of $F$. \qed
\end{lem}

\subsection{$\oo$-categorical localization}
Let $\C$ be an $\oo$-category and $W\subset\Mor(\C)$ be a set of $1$-morphisms. 
For any $\oo$-category $\E$ we denote
\[
\Fun(\C, \E)_{[W\to\Eq]} \,\subset \, \Fun(\C,\E)
\]
the full $\oo$-subcategory spanned by ($\oo$-)functors that take elements of $W$ to
equivalences in $\E$.  

\begin{defi}\label{def:oo-localiz}
Let $\pi:\C\to\Di$ be an $\oo$-functor. We say that $\pi$ {\em exhibits $\Di$ an an $\oo$-categorical
localization of $\C$ by $W$}, or, simply, that $\pi$ is an {\em $\oo$-localization of $\C$ by $W$},
if:

\begin{itemize}
\item[(1)] $\pi \in \Fun(\C, \Di)_{[W\to\Eq]}$. 

\item[(2)] For any $\oo$-category $\E$, composition with $\pi$ gives an equivalence
\[
\pi^*: \Fun(\Di,\E) \lra \Fun(\C, \E)_{[W\to\Eq]}.
\]

\end{itemize}
\end{defi}
Given $\C$ and $W$, the datum $(\Di, \pi)$ as above, is known to exist and be unique up to
a contractible space of choices, see \cite{HA}, \S 5.2.7. We will therefore  denote such $\Di$ by
$\C[W^{-1}]_\oo$. 


We will be particularly interested in the case when $\C$ is a usual category. In this case
$\C[W^{-1}]_\oo$ is the $\oo$-categorical analog of the Dwyer-Kan simplicial localization
\cite{dwyerkan, DKSS}. In particular, 
\[
\Ho \,\C[W^{-1}]_\oo \,= \, \C[W^{-1}]
\]
 is the usual categorical localization of $\C$ by $W$. 
 
 
 We will be further interested in the cases when $\C[W^{-1}]_\oo$ is equivalent to a usual 
 category, i.e., reduces to $\C[W^{-1}]$. 
 
 Given a functor $\pi: \C\to\Di$ of usual categories and an object $d\in\Di$, we denote by
 \[
 j_d: (d/\pi)^\iso \hookrightarrow d/\pi
 \]
 the embedding of the full subcategory spanned by pairs $(c\in\C, a: d\to\pi(c))$
 for which $a$ is an isomorphism in $\Di$. Then we have (cf. \cite{walde}):
 
 \begin{prop}\label{prop:oo-localiz}
 Let $\pi: \C\to\Di$ be a functor of usual categories and $W\subset\Mor(\C)$. Suppose that:
 \begin{itemize}
 \item[(1)] Elements of $W$ are precisely the morphims of $\C$ sent by $\pi$ into isomorphisms. 
 
 \item[(2)] For any $d\in\Di$ the category $(d/\pi)^\iso$ is contractible. 
 
 \item[(3)] For any $d\in\Di$ the functor $j_d$ is $\oo$-coinitial. 
 
 \end{itemize}
 Then $\pi$ exhibits $\Di$ as the $\oo$-categorical localization of $\C$ by $W$. 
 \end{prop} 
 
 \begin{proof} {\bf Step 1.} Let $\Ec$ be any $\oo$-category with limits and $G: \Di\to\E$
 be an $\oo$-functor. Then the natural transformation 
 $G\to \pi_*\pi^*G$ is an equivalence.
 Indeed, by the pointwise formula for Kan extensions and $\oo$-coinitiality of $j_d$, we have
 \[
 (\pi_*\pi^* G)(d) \,=\,\varprojlim_{c\in d/\pi} G(\pi(c)) \,=\,\varprojlim_{c\in (d/pi)^\iso} G(\pi(c)). 
 \]
 But $(d/\pi)^\iso$ consists of isomorphisms $d\to\pi(c)$, so by inverting them, we can say that it
 consists of isomorphisms $\pi(c)\buildrel\simeq \over \to d$. So we have
 \[
  (\pi_*\pi^* G)(d) \,=\,\varprojlim_{\{\pi(c)\buildrel\simeq \over \to d\}} G(\pi(c)) \,=\, G(d),
 \]
 since the last limit is taken over a cone-shaped diagram (one with an initial object).

 {\bf Step 2.} Further, let $F: \C\to\E$ be any $\oo$-functor which takes elements
 of $W$ into equivalences. Then the natural transformation $\pi^* \pi_* F\to F$ is an equivalence.
 Indeed, as above, for any $c\in\C$, 
 \be\label{eq:p^*p_*}
 (\pi^*\pi_* F)(c) \,=\,\varprojlim_{c'\in \pi(c)/\pi} F(c') \, = \, 
 \varprojlim_{c'\in (\pi(c)/\pi)^\iso} F(c').
 \ee
 But $ (\pi(c)/\pi)^\iso$ has, as objects, isomorphisms $\pi(c)\buildrel b\over\to \pi(c')$, while a morphism
\[
[\pi(c)\buildrel b\over\to \pi(c'_1)] \,\lra \, [\pi(c)\buildrel b\over\to \pi(c'_2)]
\] 
between two such objects is a morphism $u: c'_1\to c'_2$ in $\C$ such that the diagram
\[
\xymatrix{
& \pi(c'_1) \ar[d]^{\pi(u)}
\\
\pi(c) \ar[r]_{b_2} 
\ar[ur]^{b_1}& \pi(c'_2)
}
\]
commutes. This means that $\pi(u)$ is an isomorphism and so $u\in W$ by the assumption (1). 
Thus $F(u)$ is an equivalence. So the limit in \eqref{eq:p^*p_*} is a limit of a diagram of
equivalences parametrized by a category that is contractible by assumption (3). 
So (e.g., by Lemma \ref{lem:conlim} for limits instead of colimits) it is identified with any term
of the diagram, in particular, the natural map from this limit to $F(c)$ is an equivalence.

{\bf Step 3.}  Now consider the pullback functor
 \[
 \pi^*: \Fun(\Di,\E) \lra\Fun(\C, \E).
 \]
Step 1 implies that $\pi^*$ is fully faithful (the embedding of a full $\oo$-subcategory).
Step 2 means that the essential image of $F$ is $\Fun(\C, \E)_{[W\to\Eq]}$. This means that
$\pi$ satisfies the condition (2) of Definition \ref{def:oo-localiz} for any $\E$ with limits. 

Finally, we note that in the above reasoning it is not necessary to require that $\E$
has all limits as all the limits we need, automatically exist and are explicitly identified.
This proves Proposition \ref{prop:oo-localiz}. 
\end{proof}

Fir future use, we note a dual version of  Proposition \ref{prop:oo-localiz}.
For a functor $\pi: \C\to\Di$  of usual categories and an object $d\in\Di$, we consider the embedding
\[
j^d: (\pi/d)^\iso \hookrightarrow \pi/d
\]
where $(\pi/d)^\iso$ is the full subcategory of $\pi/d$ formed by pairs
$(c\in\C, b: \pi(c)\to d)$ for which $b$ is an isomorphism.

\begin{prop}\label{prop:oo-localiz-dual}
Let $\pi: \C\to\Di$ be a functor  of usual categories. Suppose that:
\begin{itemize}
 \item[(1)] Elements of $W$ are precisely the morphims of $\C$ sent by $\pi$ into isomorphisms. 
 
 \item[(2)] For any $d\in\Di$ the category $(\pi/di)^\iso$ is contractible. 
 
 \item[(3)] For any $d\in\Di$ the functor $j^d$ is $\oo$-cofinal. 
 
 \end{itemize}
 Then:
 \begin{itemize}
 \item[(a)]  $\pi$ exhibits $\Di$ as the $\oo$-categorical localization of $\C$ by $W$. 
 
 \item[(b)] For any $\oo$-category $\E$ with colimits and any functor $G: \Di\to\E$
 the natural transformaton $\pi_!\pi^* G\to \G$ is an equivalence.
 
 \item[(c)] For any functor $F: \C\to \E$ sending elements of $W$ to equivalences,
 the natural transformation $F\to \pi^*\pi_!F$ is an equivaence.
 
 \end{itemize}
\end{prop}

\begin{proof}
	Obtained from that of Proposition \ref{prop:oo-localiz} by dualization.
\end{proof}

\subsection{A covering lemma} 

We recall the following lemma which generalizes various classical statements of the kind that a
space is homotopy equivalent to the nerve of its sufficiently fine open covering. 

\begin{lem}\label{lem:vankampen} Let $T$ be a small category. Let $E$ be a topological space and let $\Op(E)$
	denote the poset of open subsets of $E$. Let
	\[
		\chi: T\lra \Op(E))
	\]
	be a functor. For any $e\in E$ let $\chi^{-1}(e)\subset T$ be full subcategory   spanned by $t$ such that
			$e \in \chi(t)$. Suppose that: 
 
	\begin{enumerate}
		\item for every $t \in T$, the open  set $\chi(t)$ is contractible,
		\item for every $e \in E$, the  category $\chi^{-1}(e)$ is  contractible.
	\end{enumerate}
	Then there is a weak homotopy equivalence $|\N(T)| \simeq E$.
\end{lem}

Note that the assumption (2) implies, in particular, that the $\chi(t)$ form an open covering of $E$,
as a contractible category is nonempty. 

\begin{proof}
	Let $\pi: K \to \N(T)$ be the relative nerve (\cite[3.2.5]{HTT}) associated to the functor 
	\[
		\N(T) \to \Sp,\; t \mapsto \Sing(\chi(t)).
	\]
	Then $\pi$ is a left fibration whose fibers are, by assumption, contractible Kan complexes.
	By \cite[2.1.3.4]{HTT} it is a trivial Kan fibration so that $|K| \simeq |N(T)|$. On
	the other hand, $\Sing(|K|)$ is a model for the colimit of $\pi$ (\cite[3.3.4.6]{HTT}) which
	by Lurie's Seifert-van Kampen Theorem \cite[A.3.1]{HA} is weakly equivalent to
	$\Sing(E)$.
\end{proof}

%
%

. 

\vskip 1cm

\small{

T.D.:  Universit\"at Hamburg,
Fachbereich Mathematik,
Bundesstrasse 55,
20146 Hamburg, Germany. Email: 
{\tt tobias.dyckerhoff@uni-hamburg.de}

\smallskip

M.K.: Kavli IPMU, 5-1-5 Kashiwanoha, Kashiwa, Chiba, 277-8583 Japan. Email: 
\hfil\break
{\tt mikhail.kapranov@protonmail.com}

 \smallskip 
 
 Y.S.:  Dept.  Math.,  Kansas State University, Manhattan, KS 66506 USA.
 Email: {\tt soibel@math.ksu.edu}

 }


\begin{thebibliography}{99}

 \bibitem {ayala-francis} D. Ayala, J. Francis. Factorization homology of
 topological manifolds.  {\em J. Topol.}  {\bf 8}   (2015) 1045-1084. 
 
  \bibitem{beil-gluing} A.~Beilinson.
\newblock How to glue perverse sheaves. In: $K$-theory, arithmetic and geometry (Moscow, 1984), 
\newblock {\em  Lecture Notes in Math.} {\bf 1289}, 
Springer-Verlag, 1987, 42 - 51.

\bibitem{BBD} A.~Beilinson, I.~Bernstein, P.~Deligne.
\newblock Faisceaux Pervers, {\em Ast\'erisque} {\bf 100}, 1982. 

\bibitem {BD} A. Beilinson, V. Drinfeld. Chiral Algebras.
Amer. Math. Soc. 2004. 

\bibitem{bergner} J. Bergner. The Homotopy Theory of  $(\oo,1)$-Categories. 
Cambridge Univ. Press,  2018. 

\bibitem{BK} A. I. Bondal, M. M. Kapranov. Enhanced triangulated categories. 
 {\em Math. USSR-Sb.}  {\bf 70}  (1991)  93-107.  
 
 \bibitem {borsuk-ulam} K. Borsuk, S. Ulam. On symmetric products of topological spaces.
 {\em Bull. Amer. Math. Soc.}  {\bf 37} (1931) 875-882. 
 
 \bibitem{bott} R. Bott. On the third symmetric potency of $S_1$. {\em Fundamenta Math.}
 {\bf 36} (1949) 236-244. 

\bibitem{cepek:thesis} A. Cepek. On configuration categories. {\em 2019 PhD Thesis}

\bibitem{cepek:ran} A. Cepek. Configuration spaces of $\RR^n$ by way of exit-path $\infty$-categories. arXiv:1910.11980

\bibitem{connes}
A.~Connes.
\newblock Noncommutative Geometry.
\newblock Academic Press,  1994.

\bibitem{deligne-SGA} P. Deligne. La formule de dualit\'e globale. SGA 4, Exp. XVIII. 

\bibitem{drinfeld}
V. ~Drinfeld.
\newblock  On the notion of geometric realization. 
\newblock{\em Mosc. Math. J. }  {\bf 4}  (2004)  619-626 . 

\bibitem{dugger} D. Dugger, S. Hollander, D. C. Isaksen. Hypercovers and
simplicial presheaves. arXiv:math/0205027. 

\bibitem{DKSS}
W.~G. Dwyer, P.~S. Hirschhorn, D.~M. Kan, and J.~H. Smith.
\newblock  Homotopy Limit Functors on Model Categories and Homotopical
  Categories. 
\newblock Amer. Math.  Soc.  Providence, RI, 2004.

\bibitem{dwyerkan}
W.~G. Dwyer and D.~M. Kan.
\newblock  Simplicial localizations of categories. 
\newblock {\em J. Pure Appl. Alg.} {\bf 17} (1980) 267-294. 

 \bibitem{DK} W. G. Dwyer, D. M. Kan. Normalizing the cyclic modules of Connes.
{\em Comm. Math. Helv.} {\bf 60} (1985) 582-600. 

\bibitem{dyckerhoff:DK}  T. Dyckerhoff. A categorified Dold-Kan correspondence.
arXiv:1710.08356. 
 

\bibitem{DK:crossed} T. Dyckerhoff, M. Kapranov. 
Crossed simplicial groups and structured surfaces, in: 
``Stacks and categories in geometry, topology, and algebra'', p. 37-110, 
{\em Contemp. Math.}  {\bf 643}, Amer. Math. Soc., Providence, RI, 2015.  

\bibitem{HSS} T. Dyckerhoff, M. Kapranov. Higher Segal Spaces. {\em Lect. Notes in Math.}{\bf 2244}, Springer-Verlag, 2019. 

\bibitem{elmendorf} A. D. Elmendorf. A simple formula for cyclic duality.
{\em Proc. AMS}, {\bf 118} (1993) 709-711. 

\bibitem{farb} B. Farb,  D. Margalit. A Primer on Mapping Class Groups. Princeton Univ. Press, 
2012. 
 
 \bibitem{gabriel-zisman} 
 P. Gabriel, M. Zisman. 
Calculus of Fractions and Homotopy Theory,
 Springer-Verlag,  1967. 

\bibitem{GGM} A.~Galligo, M.~Granger, P.~ Maisonobe.
\newblock $\Dc$-modules et faisceaux pervers dont le support singulier
est un croisement normal. \newblock {\em Ann. Inst. Fourier Grenoble},
{\bf 35} (1985) 1-48. 

\bibitem{goerss-jardine} P. G. Goerss,  J. F. Jardine.  Simplicial Homotopy Theory.  Birkh\"auser,
Boston, 1999. 

\bibitem{GM-IC-II} M. Goresky, R. MacPherson. Intersection homology II.
{\em Invent. Math.}  {\bf 72} (1983) 77-130. 

\bibitem{GM-morse} M. Goresky, R. MacPherson. Stratified Morse Theory. Springer-Verlag, 1988. 

   \bibitem{HS}  V.~A.~Hinich, V.~V.~ Schechtman.
   On homotopy limit of homotopy algebras. K-theory, arithmetic and geometry (Moscow, 1984-1986), p. 240-264, 
   {\em Lecture Notes in Math.}  {\bf 1289}, Springer-Verlag, 1987.
   
   \bibitem{jardine} J. F. Jardine. Simplicial presheaves.  {\em J. of Pure and Appl. Alg.}
   {\bf 47} (1987) 35-87. 
   
   \bibitem{johnson-osorno} 
   N. Johnson,  A. M. Osorno. Modeling stable one-types. 
   {\em Theory Appl. Categ.}  {\bf 26} (2012) 520–537. 


\bibitem {kapranov-schechtman:schobers} M. Kapranov, V. Schechtman. 
Perverse Schobers. arXiv:1411.2772.  

\bibitem {kapranov-schechtman:graphs} M. Kapranov, V. Schechtman. 
Perverse sheaves and graphs on surfaces.  arXiv:1601.01789.

\bibitem{kashiwara-schapira} M. Kashiwara, P. Schapira. Sheaves on Manifolds. Springer-Verlag, 1991. 

\bibitem{kuratowski} K. Kuratowski. Topology vol. I, Academic Press, 2014. 
 
\bibitem{laumon} G. Laumon. Vanishing cycles over a base of dimension $\geq 1$, in:
``Algebraic Geometry'', {\em Lecture Notes in Math.} {\bf 1016}, p. 143-150, Springer-Verlag, 1983.

\bibitem {loday} J.-L. Loday. Cyclic Homology. Springer-Verlag, 1992.

\bibitem{HTT} J. Lurie. Higher Topos Theory. Princeton Univ. Press, 2009. 

\bibitem{HA} J. Lurie. Higher Algebra. Available at:
$<$http://people.math.harvard.edu/~lurie/papers/HA.pdf$>$

\bibitem {lurie:SAG} J. Lurie. Spectral Algebraic Geometry. Available at: 
$<$https://www.math.ias.edu/~lurie/papers/SAG-rootfile.pdf$>$

 \bibitem {macpherson:intersection} R. D. MacPherson.
Intersection Homology and Perverse Sheaves, Report December 15, 1990.
Unpublished Notes: http://faculty. tcu. edu/gfriedman/notes/ih. pdf

\bibitem{mostovoy} J. Mostovoy. Lattices in $\CC$ and finite subsets of a circle.
{\em Amer. Math. Monthly} {\bf 111} (2004) 357-360. 

\bibitem{noohi} B. Noohi. Notes on 2-groupoids, 2-groups and crossed modules. 
    {\em Homology Homotopy Appl.}
    {\bf 9}  (2007), 75-106.
 

\bibitem{rezk} C. Rezk.  A model for the homotopy theory of  homotopy theory.
{\em Trans. Amer. Math. Soc.}  {\bf 353} (2001)  973-1007. 


 \bibitem{spaltenstein} N. Spaltenstein. Resolutions of unbounded complexes.
 {\em Compositio Math.}  {\bf 65}  (1988) 121-154.

\bibitem{treumann} D. Treumann. Exit paths and constructible stacks. 
{\em Compositio Math.} {\bf 145} (2009) 1504-1532. 

\bibitem{tuffley-s1} C. Tuffley. Finite subset spaces of $S^1$.
{\em Alg. and Geom. Topology}, {\bf 2} (2002) 1119-1145. 

\bibitem{walde} T. Walde. Walde, 2-Segal spaces as invertible infinity-operads. Algebraic \& Geometric Topology, to appear 

\bibitem{waldhausen} 
 F. Waldhausen. Algebraic K-theory of spaces, in: `` Algebraic and geometric topology'' (New Brunswick, N.J., 1983), 
 {\em  Lecture Notes in Math.} {\bf 1126}, 
  Springer-Verlag, 1985. 

\end{thebibliography}
\end{document}